\documentclass[reqno]{amsart}  
\usepackage{amsmath,amssymb,mathrsfs}

\theoremstyle{plain}
\newtheorem{theorem}{Theorem}[section]
\newtheorem{lemma}[theorem]{Lemma}
\newtheorem{remark}[theorem]{Remark}
\newtheorem{proposition}[theorem]{Proposition}

\numberwithin{equation}{section}

\numberwithin{equation}{section} \allowdisplaybreaks[1]

\theoremstyle{definition}

\theoremstyle{remark}

\newcommand{\C}{{\mathbb C}}

\newcommand{\bbZ}{{\mathbb Z}}

\newcommand{\sbm}[1]{\left[\begin{smallmatrix} #1
        \end{smallmatrix}\right]}

\newcommand{\BB}{{\mathbb B}}
\newcommand{\BC}{{\mathbb C}}\newcommand{\BD}{{\mathbb D}}

\newcommand{\BK}{{\mathbb K}}

\newcommand{\BP}{{\mathbb P}}

\newcommand{\BT}{{\mathbb T}}

\newcommand{\BZ}{{\mathbb Z}}
\newcommand{\cA}{{\mathcal A}}\newcommand{\cB}{{\mathcal B}}
\newcommand{\cC}{{\mathcal C}}\newcommand{\cD}{{\mathcal D}}
\newcommand{\cE}{{\mathcal E}}\newcommand{\cF}{{\mathcal F}}
\newcommand{\cG}{{\mathcal G}}\newcommand{\cH}{{\mathcal H}}

\newcommand{\cK}{{\mathcal K}}\newcommand{\cL}{{\mathcal L}}
\newcommand{\cM}{{\mathcal M}}\newcommand{\cN}{{\mathcal N}}

\newcommand{\cQ}{{\mathcal Q}}\newcommand{\cR}{{\mathcal R}}
\newcommand{\cS}{{\mathcal S}}\newcommand{\cT}{{\mathcal T}}
\newcommand{\cU}{{\mathcal U}}\newcommand{\cV}{{\mathcal V}}
\newcommand{\cW}{{\mathcal W}}\newcommand{\cX}{{\mathcal X}}
\newcommand{\cY}{{\mathcal Y}}\newcommand{\cZ}{{\mathcal Z}}
\newcommand{\al}{\alpha}
\newcommand{\be}{\beta}
\newcommand{\ga}{\gamma}\newcommand{\Ga}{\Gamma}
\newcommand{\de}{\delta}

\newcommand{\ze}{\zeta}

\newcommand{\la}{\lambda}

\newcommand{\si}{\sigma}

\newcommand{\vph}{\varphi}
\newcommand{\om}{\omega}\newcommand{\Om}{\Omega}
\newcommand{\mat}[2]{\ensuremath{\left[\begin{array}{#1}
#2
\end{array} \right]}}
\newcommand{\ov}[1]{{\overline{#1}}}

\newcommand{\inn}[2]{\ensuremath{\langle #1,#2 \rangle}}
\newcommand{\tu}[1]{\textup{#1}}
\newcommand{\half}{\frac{1}{2}}
\newcommand{\ands}{\quad\mbox{and}\quad}

\newcommand{\im}{\textup{Im\,}}

\newcommand{\diag}{\textup{diag\,}}
\newcommand{\spec}{r_\textup{spec}}

\newcommand{\tilG}{{\widetilde G}}

\newcommand{\tilW}{{\widetilde W}}
\newcommand{\fL}{{\mathfrak L}}
\newcommand{\fD}{{\mathfrak D}}
\newcommand{\sm}[1]{\begin{smallmatrix} #1
        \end{smallmatrix}}

\begin{document}

\begin{abstract}
The theory of Nevanlinna-Pick and Carath\'eodory-Fej\'er
interpolation for matrix- and operator-valued Schur class functions
on the unit disk is now well established.  Recent work has produced
extensions of the theory to a variety of multivariable settings,
including the ball and the polydisk (both commutative and
noncommutative versions), as well as a time-varying analogue.
Largely independent of this is the recent Nevanlinna-Pick
interpolation theorem by P.S. Muhly and B. Solel for an abstract
Hardy algebra set in the context of a Fock space built from a
$W^*$-correspondence $E$ over a $W^{*}$-algebra $\cA$ and a
$*$-representation $\sigma$ of $\cA$.  In this review we provide an
exposition of the Muhly-Solel interpolation theory accessible to
operator theorists, and explain more fully the connections with the
already existing interpolation literature.
The abstract point evaluation first introduced by Muhly-Solel leads to
a tensor-product type functional calculus in the main examples.
A second kind of point-evaluation for the $W^*$-correspondence
Hardy algebra, also introduced by Muhly and Solel, is here further
investigated, and a Nevanlinna-Pick theorem in this setting is proved.
It turns out that, when specified for examples, this alternative
point-evaluation leads to an operator-argument functional calculus
and corresponding Nevanlinna-Pick interpolation.
We also discuss briefly several Nevanlinna-Pick interpolation results
for Schur classes that do not fit into the Muhly-Solel
$W^*$-correspondence formalism.
\end{abstract}

\title[Nevanlinna-Pick interpolation: a survey]
{Multivariable operator-valued Nevanlinna-Pick interpolation: a survey}

\author[J.A.~Ball]{Joseph A. Ball}
\address{Department of Mathematics,
Virginia Tech, Blacksburg, VA 24061-0123, USA}
\email{ball@math.vt.edu}

\author[S.~ter Horst]{Sanne ter Horst}
\address{Department of Mathematics,
Virginia Tech, Blacksburg, VA 24061-0123, USA}
\email{terhorst@math.vt.edu}


\maketitle

\section{Introduction} \label{S:intro}
\setcounter{equation}{0}

The classical interpolation theorems of Nevanlinna \cite{Nev19} and Pick \cite{Pick},
now approaching the age of one hundred years, can be stated as follows:
{\em Given $N$ distinct points $\lambda_{1}$, $\dots$, $\lambda_{N}$ in the
unit disk ${\mathbb D} = \{ \lambda \in {\mathbb C} \colon
|\lambda|<1\}$ together with $N$ complex numbers $w_{1}$, $\dots$,
$w_{N}$, there exists a Schur class function $s$} (i.e., $s$ holomorphic
from the unit disk ${\mathbb D}$ to the closed disk
$\overline{\mathbb D}$) {\em such that
$$
  s(\lambda_{i}) = w_{i} \text{ for } i = 1, \dots, N
$$
if and only if the so-called Pick matrix
$$
   {\mathbb P} := \left[ \frac{1 - w_{i} \overline{w_{j}}}{1 -
   \lambda_{i} \overline{\lambda_{j}}} \right]_{i,j=1}^N
$$
is positive semidefinite.}  There is a parallel result usually
attributed to Carath\'eodory and Fej\'er as well as to Schur (see \cite{Caratheodory,
Fejer, SchurI, SchurII}) for the case
where one prescribes an initial segment of Taylor coefficients at the
origin rather than functional values at distinct points: {\em  Given
$N+1$ complex numbers $s_{0}, \dots, s_{N}$, there exists a Schur class
function $s$ such that
$$
   \frac{1}{ i!} \frac{d^{i}s}{d \lambda^i} (0) = s_{i} \text{ for } i
   = 0, \dots, N
$$
if and only if the $(N+1) \times (N+1)$ lower-triangular Toeplitz matrix
$$
    \begin{bmatrix} s_{0} &0 &\cdots &0 \\ s_{1} & s_{0} &\ddots &\vdots \\ \vdots & \ddots
    & \ddots & 0\\ s_{N} &  s_{N-1}  & \cdots   & s_{0} \end{bmatrix}
$$
is contractive.}

Over the years these results have been an inspiration for new mathematics,
often interacting symbiotically with assorted engineering applications.
We mention in particular that the operator-theoretic formulation of the
Nevanlinna-Pick/Carath\'e\-odory-Fej\'er interpolation problem due to Sarason
\cite{Sarason} led to the advance in operator theory known as commutant lifting
theory introduced by Sz.-Nagy-Foias \cite{NF68}; see also \cite{FF, FFGK98}.
Commutant lifting in turn provides a unifying framework for the handling of a
variety of interpolation problems of Nevanlinna-Pick/Carath\'eodory-Fej\'er
type for matrix- and operator-valued functions on the disk, and, more generally,
for left- and right-tangential interpolation with operator-argument
(\textbf{LTOA/RTOA}); cf., \cite{FFGK98}.
The operator-argument approach to interpolation emerged as one of the most popular
ways of handling Nevanlinna-Pick and Carath\'eodory-Fej\'er interpolation
conditions in a unified way and can be seen to be equivalent to the Sarason
formulation (see \cite{BGR, BB-AIP}).
We mention that it is the Sarason formulation which plays a prominent role
in connection with the $H^{\infty}$-control theory (see e.g.~\cite{BGR}) where
it takes the form of a model matching problem.

Our goal here is to survey recent generalizations of
the theory of Nevanlinna-Pick interpolation to several-variable
contexts. In addition there has been a lot of recent work on
noncommutative function theory, where one plugs in a tuple of
noncommuting operators (or matrices) as the function arguments and
outputs a matrix or operator.  In this regard, we mention the papers
\cite{Ag88, Ag90, AgMcC99, AgMcC00, AP00, BB02, BB04, BB05, BB-AIP,
BT, BTV, Bolo03, DP98, EP, mt, Quiggen, RR}  for the commutative setting and
\cite{BB07, CJ, Popescu98, Popescu03, Popescu06} for the
noncommutative setting.  We mention yet another direction having an
analogue of the Schur class and of Nevanlinna-Pick interpolation,
namely: the unit ball of the algebra of lower triangular matrices
acting as operators on $\ell^{2}({\mathbb Z})$; in the 1990s there was a lot of
activity on Beurling-Lax representation and an analogue of
\textbf{LTOA/RTOA} interpolation theory having connections with
robust control and model reduction for time-varying systems
(see \cite{ADD90, BGK92, DD92, DvdV,SCK}).

Going beyond these types of results is the
recent generalized Nevanlinna-Pick interpolation theorem
\cite{MS04} and the generalized Schur class  \cite{MSSchur} of
Muhly-Solel, where the concept of {\em completely positive kernel} or
{\em completely positive map} enters into various characterizations.
In particular, the Muhly-Solel result, when specialized to various
specific settings, is different from the standard
\textbf{LTOA} or \textbf{RTOA} formulation, in that the
point-evaluation is actually with a tensor-type functional calculus.
The simplest instance of this is what we call {\em Riesz-Dunford}
interpolation, where one is given operators $Z$ and $W$ on a Hilbert
space $\cZ$ and one seeks a scalar Schur class function $s$ so that
$s(Z) = W$.  We discuss how a solution of this problem can be had by
using existing theory for \textbf{LTOA}, for both the unit disk
setting and the right half-plane setting.  In particular, the theorem
for the right half-plane setting gives a solution which appears quite
different from that obtained by Cohen-Lewkowicz \cite{CL}.  We also
show how the Muhly-Solel solution criterion (involving complete
positivity of a kernel or of a map between $C^{*}$-algebras) can be
seen to be equivalent to the criterion (involving positivity of a
single block matrix) obtained from the \textbf{LTOA} theory.
A similar story holds in the setting of the unit ball (commutative
and noncommutative).  There is an analogue of the Riesz-Dunford
Nevanlinna-Pick interpolation problem in the setting of the unit ball
which is handled by the Muhly-Solel theory; to the uninitiated the
solution criterion looks somewhat strange since it involves {\em
complete positivity} (in the sense of \cite{BBLS})
rather than merely {\em positivity} of a
kernel, or, in another formulation, {\em complete positivity} of a
map between $C^{*}$-algebras rather than merely {positive
semidefiniteness} of a single block operator matrix.
We resolve this situation by giving a direct proof that the two
solution criteria are equivalent.
In the recent paper \cite{MSPoisson} Muhly and Solel introduced a second kind of
point-evaluation in the context of a generalized Poisson kernel. We develop here
some further properties of this point-evaluation and derive the corresponding
Nevanlinna-Pick interpolation theorem.
For instance, it is shown that the Muhly-Solel Hardy algebra with respect to this
point-evaluation is in general not multiplicative; a multiplicative law reminiscent
of that in the operator-argument functional calculus and time-varying system theory
literature is proved. This connection with time-varying systems was already hinted
upon in \cite{MSPoisson}. In fact, when specified for the standard one-variable
Schur class, the Nevanlinna-Pick interpolation theorem for this alternative
point-evaluation gives us precisely the corresponding left-tangential
operator-argument result. The precise connections with the time-varying-system
interpolation theory remain to be worked out.


There is yet another generalized theory of Schur class (or, more
precisely,  Schur-Agler
class) and Nevanlinna-Pick interpolation where one defines a
Schur class starting from a family of test functions (see
\cite{DMMcC, DMcC, McCS} and see also
\cite[Chapter 13]{AgMcC} for an introduction to this approach).
It is well appreciated that the Agler theory for the polydisk is an
example for this theory (indeed, this is the motivating example);
precise specification of what other examples can be covered, such as
the higher rank graph algebras of \cite{KP06} and the Hardy algebras
associated with product decompositions along semigroups more general
than ${\mathbb Z}$ (see \cite{S06}), is an ongoing area of
investigation \cite{BBDtHT}.

In this survey of interpolation problems of Nevanlinna-Pick type in a
variety of settings,  we discuss only criteria for existence of solutions;
we do not discuss characterizations of the various Schur classes via
realization as the transfer function of a conservative linear system
or of the construction of solutions or parametrization of solutions
of interpolation problems via linear-fractional formulas, although in
the various  cases often such topics are worked out in the literature.
In general we do not discuss the techniques used for establishing
these solution criteria;  let us only mention here that,
just as in the classical case,
there are a variety of techniques for analyzing these types of
multivariable interpolation problems.  We mention specifically
commutant lifting theory \cite{Popescu89a, Popescu89b, BT, BTV},
the ``lurking isometry'' method \cite{AgMcC99, AgMcC00, BT, BTV, BB02, BB04, BB05},
the Fundamental-Matrix-Inequality method of Potapov \cite{Bolo03, BB-AIP} as well as
the Ball-Helton Grassmannian Kre\u{\i}n-space method \cite{Fang}.

The paper \cite{BB-AIP} also surveys the various types of
operator-valued interpolation problems but with a focus on the
Drury-Arveson Schur-multiplier class.  In addition to \textbf{LTOA},
the paper \cite{BB-AIP} treats a more general version of
\textbf{LTOA} (where the joint spectra of $Z^{(1)}$, $\dots$,
    $Z^{(N)}$ are no longer required to be in the open ball ${\mathbb
    B}^{d}$), and a still more general Abstract Interpolation Problem
    (see \cite{KKY} for the single-variable version). These more
    general formalisms handle more general interpolation problems
    (e.g., boundary interpolation) which go beyond the original
    results based on the commutant lifting approach.  We do not
    discuss these more general problems here.

    With one exception (namely, the paper \cite{CL} in connection
    with Riesz-Dunford interpolation conditions),
    we focus on interpolation theory on the unit disk
    and multivariable generalizations of the unit disk; this means we
    ignore all the activity that has been going on of late on
    interpolation theory for the Nevanlinna class (holomorphic functions
    taking the right half plane into itself) and its multivariable
    generalizations.

  The paper is organized as follows. In Section \ref{S:1var} we review
  Nevanlinna-Pick interpolation for the single-variable case.
  Perhaps new is the derivation of the result for the tangential
  Riesz-Dunford interpolation problem due to Rosenblum-Rovnyak \cite[Section
  2.3]{RR} and for the full Riesz-Dunford interpolation problem as an
  application of known results for \textbf{LTOA}.
  In the third section various extensions of the one-variable theory
  for functions of several variables are discussed. We consider the
  unit ball case (both commutative and noncommutative) in Subsections
  \ref{subS:comball} and \ref{subS:noncomball}, and the generalization
  to so-called free semigroupoid algebras \cite{Muhly97, MS99, KP04a,
  KP04b} in Subsection \ref{subS:quiver1} where the starting point is a
  directed graph (also called a quiver). New in the free semigroupoid algebra
  setting is a more explicit formula for the point-evaluation and of the
  criterion for existence of solutions of the associated Nevanlinna-Pick
  interpolation problem. We also work out the Nevanlinna-Pick result for a
  specific quiver. This example corresponds to Nevanlinna-Pick interpolation
  for functions in a certain subalgebra of a matrix-valued $H^\infty$-space.
  In Subsection \ref{subS:polydisk} we develop
  parallel results for the Schur-Agler class in the polydisk setting
  for both the commutative and noncommutative settings. Most of the
  material in Sections \ref{S:1var} and \ref{S:multivar} should be well known
  to the experts in the interpolation theory community; the focus here
  is on the connections among the various results.
  Section \ref{S:C*-NP} contains the relatively new results on a
  generalized Schur class and Nevanlinna-Pick interpolation obtained
  by P.S. Muhly and B. Solel \cite{MS04,MSSchur}.
  The material is presented for the expert in operator
  theory not familiar with such notions as ``$W^*$-correspondence'' and
  ``Hilbert module'' which are more common in the operator algebra
  literature.
  The novel part of the exposition here is to define the point-evaluation
  in a direct way without making explicit use of the dual correspondence.
  Moreover, we prove a Nevanlinna-Pick theorem for the second kind of
  point-evaluation for the $W^*$-correspondence Schur class that was
  recently introduced in \cite{MSPoisson}.
  We recover the results for the single-variable case (Section \ref{S:1var})
  and the free semigroupoid algebra setting (Subsection \ref{subS:quiver1})
  as an application of the general Muhly-Solel theory; a key ingredient
  here is careful understanding of the connections between completely
  positive kernels and completely positive maps versus merely
  positive semidefinite operator matrices, including an application of
  Choi's theorem \cite{Choi75}.
  A final subsection gives some perspective on connections
  of the Muhly-Solel theory with the time-varying interpolation theory
  of the 1990s. The last section discusses the
  test-function approach and directions for future work.

The notation is mostly standard but we mention here a few conventions for reference.
For $\Omega$ any index set and $B$ a Banach space with norm $\|\ \|_B$, the symbol
$\ell^{2}_B(\Omega)$ denotes the space of $B$-valued norm-squared summable sequences
indexed by $\Omega$:
\[
\ell^{2}_B(\Omega) = \{ \xi \colon \Omega \to B \colon
\sum_{\omega \in \Omega} \| \xi(\omega) \|_B^{2} < \infty\}.
\]
Most often the choice $\Omega = {\mathbb Z}$ (the integers) or $\Omega = {\mathbb Z}_{+}$
(the nonnegative integers) appears. For Hilbert spaces $\cU$ and $\cY$ the symbol
$\cL(\cU,\cY)$ stands for the space of bounded linear operators from $\cU$ into $\cY$.
We write $H^2_\cU(\BD)$ for the Hardy space of analytic functions $f:\BD\to\cU$ that can
be extended to square integrable functions on the unit circle $\BT:=\{\la\in\BC\colon|\la|=1\}$.
With $H^\infty_{\cL(\cU,\cY)}(\BD)$ we denote the Banach space of uniformly bounded
analytic functions on the unit disc $\BD$ with values in $\cL(\cU,\cY)$.

\section{Operator-valued Nevanlinna-Pick interpolation: the one-variable case}
\label{S:1var}

Let $\cU$ and $\cY$ be Hilbert spaces. With $\cS(\cU,\cY)$ we denote the
$\cL(\cU,\cY)$-valued Schur class, i.e., the set of all holomorphic
functions $S$ on the unit disc $\BD$ whose values are contractive operators
in $\cL(\cU,\cY)$:
\begin{equation}  \label{contract-values}
\|S(\la)\|\leq 1 \text{ for all } \lambda \in\BD.
\end{equation}
In this section we consider variations on the classical Nevanlinna-Pick interpolation
problem for functions in the Schur class $\cS(\cU,\cY)$.

\subsection{Standard functional calculus Nevanlinna-Pick interpolation problems}
\label{subS:standard-1var}
The standard operator-valued Nevanlinna-Pick interpolation problem and the left-
and right-tangential versions are the following problems:
\begin{enumerate}
\item[(1)] The Full Operator-Valued Nevanlinna-Pick ({\bf FOV-NP}) interpolation problem:
{\it Given $\la_1,\ldots,\la_N$ in $\BD$ and operators $W_1,\ldots,W_N$ in $\cL(\cU,\cY)$,
determine when there exists a Schur class function $S\in\cS(\cU,\cY)$ such that
$S(\la_i)=W_i$ for $i=1,\ldots,N$.}

\item[(2)] The Left-Tangential Nevanlinna-Pick ({\bf LT-NP}) interpolation problem:
{\it
Given $\la_1,\ldots,\la_N$ in $\BD$, an auxiliary Hilbert space $\cC$ and operators
$X_1,\ldots,X_N$ in $\cL(\cY,\cC)$ and $Y_1,\ldots,Y_N$ in $\cL(\cU,\cC)$, determine
when there exists a Schur class function $S\in\cS(\cU,\cY)$ such that $X_iS(\la_i)=Y_i$
for $i=1,\ldots,N$.}

\item[(3)] The Right-Tangential Nevanlinna-Pick ({\bf RT-NP}) interpolation problem:
{\em Given $\la_1,\ldots,\la_N$ in $\BD$, an auxiliary Hilbert space $\cC$ and operators
$U_1$, $\ldots$, $U_N$ in $\cL(\cC,\cU)$ and $V_1$, $\ldots$, $V_N$ in $\cL(\cC,\cY)$,
determine when there exists a Schur class function $S\in\cS(\cU,\cY)$ such that
$S(\la_i)U_i=V_i$ for $i=1,\ldots,N$.}
\end{enumerate}

Note that the {\bf FOV-NP} interpolation problem is a particular case of the {\bf LT-NP}
interpolation problem, namely with $\cC=\cY$ and $X_i=I_\cY$ for $i=1,\ldots,N$. Moreover,
{\bf LT-NP} reduces to {\bf RT-NP}, and vice versa, as follows.
A function $S:\BD\to\cL(\cU,\cY)$ is in the Schur class $\cS(\cU,\cY)$ if and
only if the function $S^\sharp:\BD\to\cL(\cY,\cU)$ defined by
\begin{equation}\label{Ssharp}
S^\sharp(\la)=S(\ov{\la})^*\quad (\la\in\BD)
\end{equation}
is in $\cS(\cY,\cU)$. In particular, $S$ is a solution to the {\bf LT-NP} interpolation problem,
with data as above, if and only if $S^\sharp:\BD\to\cL(\cY,\cU)$ is a solution to the {\bf RT-NP}
with data $\ov{\la}_1,\ldots,\ov{\la}_N\in\BD$, $X_1^*,\ldots,X_N^*\in\cL(\cC,\cY)$ and
$Y_1^*,\ldots,Y_N^*\in\cL(\cC,\cU)$.

The relations among these variations on the classical Nevanlinna-Pick interpolation problem
are also exhibited in their solutions; see e.g.~\cite{BGR,FF}.

\begin{theorem}\label{T:FOV/LT/RT}
Let the data for the {\bf FOV-NP}, {\bf LT-NP} and the {\bf RT-NP} interpolation problem
be as given in (1), (2) and (3) above.
\begin{enumerate}
\item A solution to the {\bf FOV-NP} interpolation problem exists if and only if the
associated Pick matrix
\begin{equation}\label{Pick-FOV}
\BP_\tu{FOV}:=\mat{c}{\frac{I_\cY-W_iW_j^*}{1-\la_i\ov{\la}_j}}_{i,j=1}^N
\end{equation}
is positive semidefinite.

\item A solution to the {\bf LT-NP} interpolation problem exists if and only if the
associated Pick matrix
\begin{equation}\label{Pick-LT}
\BP_\tu{LT}:=\mat{c}{\frac{X_iX_j^*-Y_iY_j^*}{1-\la_i\ov{\la}_j}}_{i,j=1}^N
\end{equation}
is positive semidefinite.

\item A solution to the {\bf RT-NP} interpolation problem exists if and only if the
associated Pick matrix
\begin{equation}\label{Pick-RT}
\BP_\tu{RT}:=\mat{c}{\frac{U_i^*U_j-V_i^*V_j}{1-\ov{\la}_i\la_j}}_{i,j=1}^N
\end{equation}
is positive semidefinite.
\end{enumerate}
\end{theorem}

\subsection{Operator-argument functional calculus Nevanlinna-Pick interpolation problems}
\label{subS:OA-1var}

Let $S$ be a Schur class function in $\cS(\cU,\cY)$ with Taylor coefficients
$S_0,S_1,\ldots$ in $\cL(\cU,\cY)$, i.e.,
\[
S(\la)=\sum_{n=0}^\infty\la^n S_n\quad (\la\in\BD).
\]
For an auxiliary Hilbert space $\cC$ and operators $X$ in $\cL(\cY,\cC)$ and
$T$ in $\cL(\cC)$ with $\spec(T)<1$ the left-tangential operator-argument point-evaluation
$(XS)^{\wedge L}(T)$ is given by
\begin{equation}\label{lefteval}
(XS)^{\wedge L}(T)=\sum_{n=0}^\infty T^nXS_n.
\end{equation}
Similarly, for an auxiliary Hilbert space $\cC$ and operators $U$ in $\cL(\cC,\cU)$ and
$A$ in $\cL(\cC)$ with $\spec(A)<1$ the right-tangential operator-argument point-evaluation
$(SU)^{\wedge R}(A)$ is given by
\begin{equation}\label{righteval}
(SU)^{\wedge R}(A)=\sum_{n=0}^\infty S_nUA^n.
\end{equation}

With respect to these operator-argument functional calculi we consider the following
tangential Nevanlinna-Pick interpolation problems.
\begin{enumerate}
\item[(1)] The Left-Tangential Nevanlinna-Pick interpolation problem with Operator-Argument
({\bf LTOA-NP}): {\em Given an auxiliary Hilbert space $\cC$ together with operators
$T_1,\ldots,T_N$ in $\cL(\cC)$ with $\spec(T_i)<1$ for $i=1,\ldots,N$ and operators
$X_1,\ldots,X_N$ in $\cL(\cY,\cC)$ and $Y_1,\ldots,Y_N$ in $\cL(\cU,\cC)$, determine when
there exists a Schur class function $S$ in $\cS(\cU,\cY)$ so that $(X_i S)^{\wedge L}(T_i)=Y_i$
for $i=1,\ldots,N$.}

\item[(2)] The Right-Tangential Nevanlinna-Pick interpolation problem with Operator-Argument
({\bf RTOA-NP}): {\em Given an auxiliary Hilbert space $\cC$ together with operators
$A_1,\ldots,A_N$ in $\cL(\cC)$ with $\spec(A_i)<1$ for $i=1,\ldots,N$ and operators
$U_1,\ldots,U_N$ in $\cL(\cC,\cU)$ and $V_1,\ldots,V_N$ in $\cL(\cC,\cY)$, determine when
there exists a Schur class function $S$ in $\cS(\cU,\cY)$ so that $(S U_i)^{\wedge R}(A_i)=V_i$
for $i=1,\ldots,N$.}
\end{enumerate}

The left- and right-tangential point-evaluation with operator-argument
and the corresponding Nevanlinna-Pick interpolation theory also provide
a convenient formalism for encoding tangential interpolation conditions
of Carath\'eodory-Fej\'er type; we refer to \cite[Sections 16.8-9]{BGR}
for the single-variable case and \cite{BB05} for multivariable examples.

The {\bf LT-NP} ({\bf RT-NP}) interpolation problem is the special case of
the {\bf LTOA-NP} ({\bf RTOA-NP}) interpolation problem with $T_i=\la_i I_\cU$
($A_i=\la_i I_\cU$).
As in the case of the standard functional calculus left- and right-tangential
interpolation problems, here also {\bf LTOA-NP} reduces to {\bf RTOA-NP} and
vice versa; again via the transformation $S\mapsto S^\sharp$ in (\ref{Ssharp})
and a similar transformation of the data.
Moreover, it suffices to consider the {\bf LTOA-NP} interpolation problem for the case $N=1$,
since the general case, with data as above, is covered by the
{\bf LTOA-NP} interpolation problem with data
\[
T=\mat{cccc}{T_1&&&\\&T_2&&\\&&\ddots&\\&&&T_N},\quad X=\mat{c}{X_1\\X_2\\\vdots\\X_N}\ands
Y=\mat{c}{Y_1\\Y_2\\\vdots\\Y_N}.
\]
See \cite[Section I.3]{FFGK98} for more details. However, we write out
the results without this reduction for better comparison with the
classical case.

The solutions to the {\bf LTOA-NP} and {\bf RTOA-NP} interpolation problems are given in the
following theorem.

\begin{theorem}\label{T:LTOA/RTOA}
Let the data for the {\bf LTOA-NP} and {\bf RTOA-NP} interpolation problems be as
given in (1) and (2) above.
\begin{enumerate}
\item A solution to the {\bf LTOA-NP} interpolation problem exists if and only if the
associated Pick matrix
\begin{equation}\label{Pick-LTOA}
\BP_\tu{LTOA}:=\mat{c}{\sum_{n=0}^{\infty}T_{i}^{n}(X_i X_j^*- Y_i Y_j^*) T_j^{*n}}_{i,j=1}^N
\end{equation}
is positive semidefinite.

\item A solution to the {\bf RTOA-NP} interpolation problem exists if and only if the
associated Pick matrix
\begin{equation}\label{Pick-RTOA}
\BP_\tu{RTOA}:=\mat{c}{\sum_{n=0}^{\infty} A_i^{*n}(U_i^*U_j-V_i^*V_j)A_j^n}_{i,j=1}^N
\end{equation}
is positive semidefinite.
\end{enumerate}
\end{theorem}

\subsection{Riesz-Dunford functional calculus Nevanlinna-Pick interpolation problems}
\label{subS:RD-1var}
As a third variant we consider Nevanlinna-Pick interpolation problems for Schur class
functions with a Riesz-Dunford functional calculus. In this case the Hilbert spaces
$\cU$ and $\cY$ are both equal to $\BC$, i.e., the Schur class functions are scalar-valued,
but the unit disc is replaced by strict contractions on some fixed Hilbert space $\cZ$. Given a Schur class
function $s$ in $\cS(\BC,\BC)$ and a strict contraction $T$ in $\cL(\cZ)$ we define $s(Z)$
via the Riesz-Dunford functional calculus:
\begin{equation}\label{RD-func}
s(Z)=\frac{1}{2\pi i}\int_{\rho\BT}
(\zeta I_{\cZ}-Z)^{-1}s(\zeta)\, d\zeta=\sum_{n=0}^\infty s_n\cdot Z^n\ \text{ if }\ s(\la)=
\sum_{n=0}^\infty s_n\la^n.
\end{equation}
Here the multiplication $s_{n} \cdot Z^{n}$ is scalar multiplication and $\rho < 1$ is
chosen large enough so that the circle $\rho\BT$ encloses the spectrum of $Z$.

With respect to this functional calculus we consider the following Nevanlinna-Pick
interpolation problems.
\begin{enumerate}
\item[(1)] The Full Riesz-Dunford Nevanlinna-Pick ({\bf FRD-NP}) interpolation problem:
{\em Given a Hilbert space $\cZ$, strict contractions $Z_1,\ldots,Z_N$ in $\cL(\cZ)$ and
operators $W_1,\ldots,W_N$ in $\cL(\cZ)$, determine when there exists a Schur class function $s$
in $\cS(\BC,\BC)$ such that $s(Z_i)=W_i$ for $i=1,\ldots,N$.}

\item[(2)] The Left-Tangential Riesz-Dunford Nevanlinna-Pick ({\bf LTRD-NP}) interpolation problem:
{\em Given Hilbert spaces $\cZ$ and $\cC$, strict contractions $Z_1,\ldots,Z_N$ in $\cL(\cZ)$ and
operators $X_1,\ldots,X_N,Y_1,\ldots,Y_N$ in $\cL(\cZ,\cC)$, determine when there exists a Schur
class function $s$ in $\cS(\BC,\BC)$ such that $X_i s(Z_i)=Y_i$ for $i=1,\ldots,N$.}

\item[(3)] The Right-Tangential Riesz-Dunford Nevanlinna-Pick ({\bf RTRD-NP}) interpolation problem:
{\em Given Hilbert spaces $\cZ$ and $\cC$, strict contractions $Z_1,\ldots,Z_N$ in $\cL(\cZ)$ and
operators $U_1,\ldots,U_N,V_1,\ldots,V_N$ in $\cL(\cC,\cZ)$, determine when there exists a Schur
class function $s$ in $\cS(\BC,\BC)$ such that $s(Z_i)U_i=V_i$ for $i=1,\ldots,N$.}
\end{enumerate}

The {\bf FRD-NP} interpolation problem is the special case of the {\bf RTRD-NP} interpolation
problem with $\cC=\cZ$ and $U_1=\cdots=U_N=I_\cZ$. Let us first consider the {\bf RTRD-NP}
interpolation problem with $\cC=\BC$, that is, the operators $U_1,\ldots,U_N,V_1,\ldots,V_N$
are vectors $u_1,\ldots,u_N,v_1,\ldots,v_N$ in $\cZ$. This type of Nevanlinna-Pick interpolation
problem is studied in the book \cite{RR} (in a slightly  more general form where only a local weak
version of the Riesz-Dunford functional calculus is required). In this case $s(Z_i)u_i$ is equal
to $(u_i s)^{\wedge L}(Z_i)$ as defined in (\ref{lefteval}), and thus, the {\bf RTRD-NP}
interpolation problem is a {\bf LTOA-NP} interpolation problem with the same data set.
If $\cC\not=\BC$ is finite dimensional with orthonormal basis
$\{e_1,\ldots,e_\kappa\}$, then the
{\bf RTRD-NP} interpolation problem can still be seen as a {\bf LTOA-NP} interpolation problem.
In this case the identity $s(Z_i)U_i=V_i$ holds if and only if
\begin{equation*}
s(Z_i)U_ie_j=V_ie_j\text{ for }j=1,\ldots,\kappa.
\end{equation*}
Thus the {\bf RTRD-NP} interpolation problem becomes a {\bf LTOA-NP} interpolation problem with
tangential interpolation conditions indexed by the Cartesian product set
\begin{equation*}
\{1, \dots, N\}\times\{1, \dots, \kappa\}
\end{equation*}
and with interpolation data
\begin{equation*}
Z_{i,j}:=Z_i,\quad x_{i,j}:=U_ie_j,\quad y_{i,j}:=V_i e_j\text{ for
}i=1,\ldots,N\text{ and }j=1,\ldots,\kappa.
\end{equation*}
The {\bf LTRD-NP} interpolation problem reduces to a {\bf RTOA-NP} interpolation
problem in a similar way.

In conclusion, we have the following result.

\begin{theorem}\label{T:RD-NP}
Let the data for the {\bf FRD-NP}, {\bf LTRD-NP} and {\bf RTRD-NP} interpolation problems be as
given in (1) and (2) above with, for the {\bf LTRD-NP} and {\bf RTRD-NP} interpolation problems
$\dim \cC=\kappa$ and $\{e_1,\ldots,e_\kappa\}$
an orthonormal basis for $\cC$, and for the {\bf FRD-NP}
interpolation problem $\dim \cZ=\kappa$ and
$\{e_1,\ldots,e_\kappa\}$ an orthonormal basis for $\cZ$.
\begin{enumerate}
\item A solution to the {\bf FRD-NP} interpolation problem exists if and only if the
associated Pick matrix
\begin{equation}\label{Pick-FRD}
\BP_\tu{FRD}:=\mat{c}{\sum_{n=0}^{\infty} Z_{i}^{n}(e_{i'}e_{j'}^{*}-W_{i} e_{i'}
e_{j'}^{*} W_{j}^{*} )
 Z_{j}^{*n}}_{(i,i'), (j,j') \in \{1, \dots, N\} \times \{1, \dots,
 \kappa \}}
\end{equation}
is positive semidefinite.

\item A solution to the {\bf LTRD-NP} interpolation problem exists if and only if the
associated Pick matrix
\begin{equation}\label{Pick-LTRD}
\BP_\tu{LTRD}:=
\mat{c}{\sum_{n=0}^{\infty} Z_{i}^{*n}(X_i^* e_{i'}e_{j'}^*X_j-Y_{i}^* e_{i'} e_{j'}^*Y_{j})Z_{j}^{n}
}_{(i,i'), (j,j') \in \{1, \dots, N\} \times \{1, \dots, \kappa \}}
\end{equation}
is positive semidefinite.

\item A solution to the {\bf RTRD-NP} interpolation problem exists if and only if the
associated Pick matrix
\begin{equation}\label{Pick-RTRD}
\BP_\tu{RTRD}:=\mat{c}{\sum_{n=0}^{\infty} Z_i^n(U_ie_{i'}e_{j'}^*U_j^*-V_ie_{i'}e_{j'}^*V_j^*)Z_j^{*n}
}_{(i,i'),(j,j')\in\{1,\dots,N\}\times\{1,\dots,\kappa \}}
\end{equation}
is positive semidefinite.
\end{enumerate}
\end{theorem}

The statements in Theorem \ref{T:RD-NP} remain true for the case that $\cC$ is a separable
Hilbert space (i.e., $\kappa=\infty$). The Pick matrices $\BP_\tu{FRD}$, $\BP_\tu{LTRD}$
and $\BP_\tu{RTRD}$ in this case are infinite operator matrices; positivity is then
to be interpreted as positivity of all $M\times M$-finite sections for each $M\in\BZ_+$.

It is possible to study the Full Riesz-Dunford interpolation
 problem for the Nevanlinna  class ${\mathfrak N}$ (holomorphic
 functions $f$ mapping the right half plane into itself) in place of
 the Schur class $\cS$.  For $\cZ$ a coefficient Hilbert space, we
 let ${\mathfrak N}(\cZ)$ be the Nevanlinna class of $\cL(\cZ)$-valued functions
 $F(\lambda)$ holomorphic on the right half plane such that
 $F(\lambda) + F(\lambda)^{*} \geq 0$ for $\lambda +
 \overline{\lambda} \ge 0$.
 If we assume that the coefficient space
 $\cZ$ is finite-dimensional, then the solution of the left-tangential
 interpolation problem with operator-argument for the Nevanlinna
 class is given by Theorem 22.2.2 in \cite{BGR}:  {\em given an
 operator $Z$ in $\cL(\cZ)$ with spectrum in the right half plane
 and given direction
 operators $X$ and $Y$ in $\cL(\cZ, \cC)$, then there exists a
 function $F$ in the operator-valued Nevanlinna class ${\mathfrak N}
 (\cZ)$ satisfying the left-tangential interpolation condition with
 operator-argument}
 $$
 (X F)^{\wedge L}(Z): = \frac{1}{2 \pi i} \int_{C} (\zeta I
 - Z)^{-1} X F(\zeta)\, d\zeta  = Y
 $$
 (where $C$ is a simple closed curve in the right half plane
 with the spectrum of $Z$ in its interior)
 {\em if and only if the unique solution ${\mathbb P}$ of the
 Lyapunov equation
 $$   {\mathbb P} Z^{*} + Z {\mathbb P} = X Y^{*}
 + Y X^{*}
 $$
 is positive semidefinite.}
(This statement is somewhat rough; in
 general one must go to a projective completion of ${\mathfrak
 N}(\cZ)$ and allow functions which are identically equal to $\infty$
 on a subset, but we ignore this technicality for the present
 discussion.)  To solve the full Riesz-Dunford interpolation problem
 for the class ${\mathfrak N}$, we can reduce full Riesz-Dunford interpolation
 to left-tangential operator-argument interpolation by the same trick used
 above for the Schur class. We state the result for the case $N=1$ in the
 following theorem.

 \begin{theorem}  \label{T:RD-PRO}
 Given a finite-dimensional
 coefficient space $\cZ$
 with orthonormal basis $\{e_{1}, \dots, e_{\kappa}\}$ together with
 an operator $Z$ in $\cL(\cZ)$ with spectrum in the right half plane
 and an operator $W$ in $\cL(\cZ)$, there is a scalar
 Nevanlinna class function $f \in {\mathfrak N}$ such that
 $$
    f(Z) = W \quad \text{\rm (Riesz-Dunford functional calculus)}
 $$
 if and only if matrix ${\mathbb P} = \left[ {\mathbb
 P}_{i'j'}\right]_{i',j'=1}^\kappa$ with block entry ${\mathbb
 P}_{i'j'}$ in $\cL(\cZ)$ determined as the unique solution of the Lyapunov equation
 $$
   {\mathbb P}_{i'j'} Z^{*} + Z {\mathbb P}_{i'j'} = e_{i'}e_{j'}^{*}
   W^{*} + W e_{i'} e_{j'}^{*}
 $$
 is positive semidefinite in $\cL(\cZ)^{\kappa \times \kappa}$.
 \end{theorem}

 In \cite{CL} Cohen and Lewkowicz solve a full Riesz-Dunford
 interpolation problem for a somewhat smaller class than the Nevanlinna class
 ${\mathfrak N}$, namely, the class of {\em positive-real-odd} functions
 $\mathcal{PRO}$.  We define the class $\mathcal{PR}$ to be those
 Nevanlinna class functions $f \in {\mathfrak N}$ which are {\em
 real} in the sense that $f(x)$ is real for $x>0$, or, equivalently,
 $\overline{f(\lambda)} = f(\overline{\lambda})$ for $\lambda$ in the
 right half plane.  We say that the function $f$ is in the class
 $\mathcal{PRO}$ of {\em positive-real-odd} functions if in addition
 $f$ has a meromorphic pseudocontinuation to the left half plane which
 satisfies $f(-\overline{\lambda}) = -\overline{f(\lambda)}$; when
 combined with the real property, this then forces $f$ to be odd:
 $f(-\lambda) = -f(\lambda)$.
  If one does a  change of variable  of both the
  independent  and the dependent variables via a
  linear-fractional-transformation mapping the right half plane to the
  unit disk, then the Nevanlinna class transforms into the Schur class $\cS$,
  the positive-real class $\mathcal{PR}$ transforms into the set of
  Schur class functions $f$
  which are real on the unit interval $(-1, 1)$, and the
  positive-real-odd class $\mathcal{PRO}$ transforms into the class of inner functions
  which are real on the unit interval. Using techniques from
  \cite{BGR93} and \cite{ABGR}, it is not difficult to see that, in
  Theorem \ref{T:RD-PRO}, one
  can arrange for the solution $f$ to be in $\mathcal{PRO}$ if $Z$ and $W$ are taken to
  be real matrices.
  The result of Cohen-Lewkowicz in
  \cite{CL} gives a seemingly quite different criterion for the same
  problem which can be stated as follows.

  \begin{theorem}  \label{T:CL} (See \cite{CL}.)
       Given real $\kappa \times
  \kappa$ matrices $Z$ and $W$ with $Z$ ``Lyapunov regular'' {\em (no
  pair $(\lambda, \mu)$ of eigenvalues of $Z$ satisfies $\lambda +
  \overline{\mu} = 0$)},
  then there exists an $f \in \mathcal{PRO}$ such that
  $$
  f(Z) = W \quad \text{\rm (Riesz-Dunford functional calculus)}
  $$
  if and only if either of the following equivalent conditions holds:
  \begin{enumerate}
      \item  there exists a real vector $u$ and a symmetric
  matrix $S$
 (called the {\em Cauchy matrix})
   so that
  \begin{enumerate}
      \item the maximum controllability rank
      $$
      \operatorname{max} \{\operatorname{rank} \begin{bmatrix}
      v & Z^{\top}v & \cdots & (Z^{\top})^{\kappa-1}v \end{bmatrix}
      \colon v \in {\mathbb R}^{\kappa}\}
      $$
   is attained for $v = u$, and
   \item  $SZ + Z^{\top} S = uu^{\top}$ and $SW + W^{\top} S :=
   {\mathbb P}
   \geq 0$ (with ${\mathbb P}$ called the {\em Pick matrix})
   \end{enumerate}
   and in addition  $W$ is in the double commutant $\{Z\}''$ of $Z$, or
   \item $W$ is in the smallest cone containing $Z$ that is closed
   under inversion.
   \end{enumerate}
   \end{theorem}

   It would be of interest to understand more
   directly the equivalence between the criteria of Theorem \ref{T:CL} versus those
   in Theorem \ref{T:RD-PRO}.\smallskip

Besides the standard definition of the Schur class $\cS(\cU,\cY)$ used in this section
there are various other ways to define $\cS(\cU,\cY)$. One of these
definitions is: {\em  A function $S:\BD\to\cL(\cU,\cY)$ is in the Schur class $\cS(\cU,\cY)$
in case $S$ defines a contractive multiplication operator $M_S$ from the Hardy space $H^2_\cU(\BD)$
into the Hardy space $H^2_\cY(\BD)$ via}
\[
M_S: f(\la)=S(\la)f(\la)\quad (\la\in\BD).
\]

We also remark that the Hardy space $H^2_\cU(\BD)$ is a reproducing kernel Hilbert space, whose
reproducing kernel is the classical Szeg\"o kernel
\[
k(\la,\zeta)=\frac{1}{1-\la\ov{\zeta}}\quad(\la,\zeta\in\BD).
\]
Another characterization of Schur class functions in $\cS(\cU,\cY)$ is based on an associated kernel
function:
{\em A function $S:\BD\to\cL(\cU,\cY)$ is in the Schur class $\cS(\cU,\cY)$ in case the function
$k_S:\BD\times\BD\to\cL(\cY)$ given by
\[
k_S(\la,\zeta)=\frac{I_\cY-S(\la)S(\zeta)^*}{1-\la\ov{\zeta}}\quad(\la,\zeta\in\BD)
\]
is a positive kernel in the sense of Aronszajn \cite{Aron} adapted to the operator-valued setting:
For any finite collection of points $\om_1,\ldots,\om_M\in\BD$ and vectors $y_1,\ldots,y_M\in\cY$
it holds that}
\[
\sum_{i,j=1}^M\inn{k_S(\om_i,\om_j)y_j}{y_i}\geq0.
\]
The extensions of Nevanlinna-Pick interpolation theory for functions of several variables
that we consider in the following sections are based on Schur-class functions defined as
``contractive multipliers'' or via ``associated positive kernels'', rather than as functions
having contractive values as in \eqref{contract-values}.

\section{Operator-valued Nevanlinna-Pick interpolation: the multivariable case}
 \label{S:multivar}

In this section we consider Nevanlinna-Pick interpolation problems for Schur-class function of
several variables. The notion of Schur class function in these cases generalizes the definition
via ``contractive multipliers'' or ``positive kernels'' mentioned at
the end of Section \ref{S:1var}, rather that the standard
definition via ``contractive values'' \eqref{contract-values} as
given in Section \ref{S:1var}.

\subsection{The commutative unit ball setting}\label{subS:comball}

 A much studied multivariable analogue of the classical Szeg\"o
 kernel is the Drury-Arveson kernel $k_{d}$ on the unit ball ${\mathbb B}^{d} = \{
 \lambda = (\lambda_{1}, \dots, \lambda_{d}) \in {\mathbb C}^{d}
 \colon \sum_{i=1}^{d} |\lambda_{i}|^{2} < 1\}$ given by
 $$
 k_{d}(\lambda, \zeta) = \frac{1}{1-\langle \lambda, \zeta \rangle}
 = \frac{1}{1 - \lambda_{1} \overline{\zeta}_{1} - \cdots -
 \lambda_{d} \overline{\zeta}_{d}}\quad(\la,\zeta\in\BB^d).
 $$
 The associated reproducing kernel space $\cH(k_{d})$ is called
 the {\em Drury-Arveson space} which is the prototype for a
 reproducing kernel space with a {\em complete Pick kernel};
 some of the seminal references on this topic are
 \cite{Dr, Quiggen, Arv98, Arv00, AgMcC00, BTV, EP, GRS, mt, McC00}.

For a Hilbert space $\cU$ we let $\cH_\cU(k_{d})$ be the space
$\cH(k_{d})\otimes\cU$ of Drury-Arveson space functions with values
in $\cU$. It can be shown that a holomorphic function
$h \colon {\mathbb B}^{d} \to \cU$ with power series representation
$$
 h(\lambda) = \sum_{n \in {\mathbb Z}^{d}_{+}} h_{n} \lambda^{n}\
 \text{ where } h_{n} \in \cU \text{ for } n \in {\mathbb Z}^{d}_{+}\
$$
is in the vector-valued Drury-Arveson space $\cH_{\cU}(k_{d})$ if
and only if
$$
   \|h\|^{2}_{\cH_{\cU}(k_{d})} = \sum_{n \in {\mathbb Z}^{d}}
   \frac{n!}{|n|!} \|h_{n}\|^{2}_{\cU} < \infty.
$$
Here we use standard multivariable notation: For
\[
n = (n_{1},\dots, n_{d}) \in {\mathbb Z}^{d}_{+}\ands\lambda = (\lambda_{1},
 \dots, \lambda_{d}) \in {\mathbb B}^{d} \subset {\mathbb C}^{d},
\]
we set
\[
\lambda^{n} = \lambda_{1}^{n_{1}} \cdots \lambda_{d}^{n_{d}},\quad
n! = n_{1}! \cdots n_{d}!\ands|n| = n_{1} + \cdots + n_{d}.
\]

For coefficient Hilbert spaces $\cU$ and $\cY$ we define the operator-valued
Drury-Arveson Schur-multiplier class $\cS_{d}(\cU, \cY)$ to be the space of holomorphic
functions $S \colon{\mathbb B}^{d} \to \cL(\cU, \cY)$ such that the
multiplication operator
 $$
   M_{S} \colon f(\lambda) \mapsto S(\lambda) f(\lambda)\quad (\la\in\BB)
 $$
 maps $\cH_{\cU}(k_{d})$ contractively into $\cH_{\cY}(k_{d})$.
 We can then pose the Drury-Arveson space versions of the problems
 formulated in Section \ref{S:1var} by simply replacing the unit disk
 ${\mathbb D}$ by the unit ball ${\mathbb B}^{d}$ and the Schur class
 $\cS(\cU, \cY)$ by the Drury-Arveson Schur-multiplier class
 $\cS_{d}(\cU, \cY)$.

All results except the ``right-tangential'' versions extend in a natural way.
The trick to reduce a right-tangential problem to a left-tangential problem
via \eqref{Ssharp} fails to generalize to the Drury-Arveson setting;
if $S \in \cS_{d}(\cU,\cY)$ and we set
$S^{\sharp}(\la):= S(\ov{\la})^{*} =S(\ov{\la_{1}},\dots,\ov{\la_{d}})^{*}$
for each $\la\in\BB^d$, then it is usually not the case that $S^{\sharp}$ is
in $\cS_{d}(\cY,\cU)$. Results on right-tangential Nevanlinna-pick interpolation
in the Drury-Arveson space exist (see \cite{BB04, BB05} for a general
setting), but they are of the flavor of the results
for Nevanlinna-Pick interpolation for the Schur-Agler class to be discussed in
Subsection \ref{subS:polydisk} below: rather than having a test with a single
Pick matrix, the criterion involves being able to solve certain
equations for a family of Pick matrices. This makes the theory for
$\cS_{d}(\cU, \cY)$ completely asymmetric
with respect to left versus right. We summarize the results in the following
theorem, with the more complicated statement for right-tangential
interpolation problems omitted.

\begin{theorem}\label{T:DAint}
\textbf{(1) Drury-Arveson Full Operator-Valued Nevanlinna-Pick interpolation:}
Suppose that we are given points $\lambda^{(1)},\dots,\lambda^{(N)}$ in
${\mathbb B}^{d}$ together with operators $W_{1},\dots,W_{N}$ in
$\cL(\cU,\cY)$.  Then there exists an $S \in \cS_{d}(\cU, \cY)$ with
$S(\lambda^{(i)}) = W_{i}$ for $i=1,\ldots,N$ if and only if the Pick matrix
     $$
     {\mathbb P}_{\tu{FOV}} := \left[ \frac{I_{\cY} - W_{i} W_{j}^{*}}{
     1 - \langle \lambda^{(i)}, \lambda^{(j)} \rangle}
     \right]_{i,j=1}^N
     $$
     is positive semidefinite.

\medskip\noindent
\textbf{(2) Drury-Arveson Left-Tangential Nevanlinna-Pick interpolation:}
Suppose that we are given an auxiliary Hilbert space $\cC$, points
$\lambda^{(1)},\dots,\lambda^{(N)}$ in ${\mathbb B}^{d}$, operators
$X_{1},\dots,X_{N}$ in $\cL(\cY,\cC)$ and operators $Y_{1},\dots,Y_{N}$ in
$\cL(\cU, \cC)$. Then there exists an $S \in \cS_{d}(\cU, \cY)$ such that
$X_{i} S(\lambda^{(i)}) = Y_{i} \text { for } i = 1, \dots, N$ if and only
if the Pick matrix
     $$
     {\mathbb P}_{\tu{LT}}:= \left[\frac{ X_{i}X_{j}^{*} - Y_{i} Y_{j}^{*}}{1
     - \langle \lambda^{(i)}, \lambda^{(j)} \rangle} \right]_{i,j=1}^N
     $$
     is positive semidefinite.

\medskip\noindent
\textbf{(3) Drury-Arveson Left-Tangential Nevanlinna-Pick interpolation with Operator-Argument:}
Suppose that we are given an auxiliary Hilbert space $\cC$ together with commutative
$d$-tuples
     $$
     Z^{(1)} = (Z^{(1)}_{1}, \dots, Z^{(1)}_{d}), \dots,
      Z^{(N)} = (Z^{(N)}_{1}, \dots, Z^{(N)}_{d})\in\cL(\cC)^d,
     $$
i.e., $Z^{(i)}_{k} \in \cL(\cC)$ for $i=1,\dots,N$ and $k=1,\dots,d$ and for
each fixed $i$, the operators $Z^{(i)}_{1},\dots, Z^{(i)}_{d}$ commute pairwise,
with the property that each $d$-tuple $ Z^{(i)}$ has joint spectrum contained in ${\mathbb B}^{d}$.
Assume in addition that we are given operators $X_{1},\dots,X_{N}$ in $\cL(\cY, \cC)$
and operators $Y_{1},\dots,Y_{N}$ in $\cL(\cU, \cC)$. Then there is an
$S \in \cS_{d}(\cU, \cY)$ so that
     $$ (X_{i} S)^{\wedge L}(Z^{(i)}) : = \sum_{n \in {\mathbb
     Z}^{d}_{+}} (Z^{(i)})^{n} X_{i} S_{n} = Y_{i}\quad\text{for}\quad i = 1, \dots, N
     $$
     if and only if the associated Pick matrix
     $$
     {\mathbb P}_{\tu{LTOA}}:= \left[  \sum_{n \in {\mathbb
     Z}^{d}_{+}} (Z^{(i)})^{n} (X_{i} X_{j}^{*} - Y_{i} Y_{j}^{*})
     (Z^{(j)})^{n*} \right]_{i,j=1}^N
     $$
     is positive semidefinite.
     Here we use the multivariable notation:
     $Z^{n}$ $=$ $Z_{1}^{n_{1}} \cdots Z_{d}^{n_{d}}$ if
     $Z =(Z_{1}, \dots, Z_{d}) \in \cL(\cC)^d$
     and $n = (n_{1}, \dots, n_{d}) \in {\mathbb Z}^{d}_{+}$.

\medskip\noindent
\textbf{(4) Full Drury-Arveson Riesz-Dunford Nevanlinna-Pick interpolation:}
Suppose that we are given commutative $d$-tuples $Z^{(1)},\ldots,Z^{(N)}$ as in statement
(3) above, acting on a separable auxiliary Hilbert space $\cZ$ with
orthonormal basis $\{e_{1}, \dots,e_{\kappa}\}$ (with possibly $\kappa=\infty$).
We also assume that we are given operators $W_{1}, \dots,W_{N}$ in $\cL(\cZ)$.
Then there exists a scalar Drury-Arveson Schur class function $s \in \cS_{d}(\BC,\BC)$ so that
   $$
     s(Z^{(i)}): = \sum_{n \in {\mathbb Z}^{d}_{+}} s_{n}
     (Z^{(i)})^{n} = W_{i} \quad\text{for}\quad i = 1, \dots, N
   $$
   if and only if the Pick matrix
   $$
   {\mathbb P}_{\tu{FRD}} :=\left[ \sum_{N \in {\mathbb Z}^{d}_{+}}
   (Z^{(i)})^{n} (e_{i'} e_{j'}^{*} - W_{i} e_{i'} e_{j'}^{*}
   W_{j}^{*}) (Z^{(j)})^{n*} \right]_{(i,i'), (j,j') \in \{1, \dots,
   N\} \times \{1, \dots, \kappa \} }
   $$
   is positive semidefinite.
 \end{theorem}

 \begin{proof}
     Statement (1) \textbf{FOV}-interpolation)
     was obtained in \cite{AP00} and \cite{DP98} via the
     method of descending from the corresponding result for the
     noncommutative case (see Theorem \ref{T:ball-int} below).
     Similarly the second and third statements (\textbf{LT} and
     \text{LTOA} interpolation) can be obtained as a
     consequence of the corresponding noncommutative result in
     \cite{Popescu98} and \cite{Popescu03}. Alternatively, one can
     obtain the results directly without reference to the
     noncommutative theory, as is done in \cite{AgMcC00, BTV, BB02}.
     The fourth statement (\textbf{RD}-interpolation) follows from the result on the
     \textbf{LTOA}-problem in the same way as was done for the
      single-variable case in Section \ref{S:1var}.  We note also that
   \textbf{FOV} is a special case of \textbf{LT} and that  \textbf{LT} is a
   special case of \textbf{LTOA} just as in the single-variable case.
   We mention that an analogue of the Rosenblum-Rovnyak theory for
   the tangential Riesz-Dunford interpolation for this setting
   appears in \cite{Popescu98}, again via the connection with the
   noncommutative theory.
     \end{proof}

\subsection{The noncommutative unit ball setting}\label{subS:noncomball}

There is also a noncommutative version of the Drury-Arveson
Schur class.  To describe this Schur class, let $\{1, \dots, d\}$ be an
alphabet consisting of $d$ letters and let $\cF_{d}$ be the
associated free semigroup generated by the letters $1, \dots, d$
consisting of all words $\gamma$ of the form $\gamma = i_{N}
\cdots i_{1}$, where each $i_{k} \in \{1, \dots, d\}$ and where
$N =1,2,\dots$. For $\gamma = i_{N}\cdots i_{1}\in\cF_{d}$ we
set $|\gamma|:=N$ to be the {\em length} of the word $\gamma$.
Multiplication of two words $\gamma = i_{N} \cdots i_{1}$ and
$\gamma' = j_{N'}\cdots j_{1}$ is defined via concatenation:
$$
\gamma \gamma' = i_{N} \cdots i_{1} j_{N'} \cdots j_{1}.
$$
The empty word $\emptyset$ is included in $\cF_{d}$ and acts as the
unit element for this multiplication; by definition $| \emptyset|= 0$.
We set
\[
\gamma \gamma^{\prime -1} = \begin{cases} \gamma'' &\text{if
   there is a } \gamma'' \in \cF_{d} \text{ so that } \gamma =
   \gamma'' \gamma', \\
   \text{undefined} & \text{otherwise.}
   \end{cases}
\]

For a Hilbert space $\cU$, the associated {\em Fock space} $\ell^2_\cU(\cF_d)$ is
the $\cU$-valued $\ell^{2}$ space indexed by the free semigroup $\cF_{d}$:
$$
\ell^{2}_\cU(\cF_{d}) = \{ u \colon \cF_{d} \to \cU \colon
\sum_{\gamma \in \cF_{d}} \| u(\gamma)\|_\cU^{2} < \infty \}.
$$
Given two coefficient Hilbert spaces $\cU$ and $\cY$ we define the
{\em noncommutative $d$-variable Schur class} $\cS_{nc,d}(\cU,\cY)$
to be the set consisting of all formal power series
$s(z) = \sum_{\gamma \in \cF_{d}} S_{\gamma} z^{\gamma}$
in noncommuting indeterminates $z = (z_{1}, \dots, z_{d})$ (where we think
of $z^{\gamma} = z_{i_{N}} \cdots z_{i_{1}}$ if $\gamma = i_{N} \cdots i_{1}$)
with coefficients $S_{\gamma} \in \cL(\cU, \cY)$ such that the associated
Toeplitz matrix
$$ R = \left[ S_{\gamma \gamma^{\prime -1}} \right]_{\gamma,
\gamma' \in \cF_{d}}, \text{ where we set }S_\tu{undefined}=0,
$$
defines a contraction operator from $\ell^{2}_{\cU}(\cF_{d})$ into
$\ell^{2}_{\cY}(\cF_{d})$.

   There has been a variety of interpolation results
   for this setting, for instance Sarason type, Rosenblum-Rovnyak or
   tangential Riesz-Dunford type and left-tangential operator-argument type;
   we refer to \cite{Popescu98, DP98, AP00, Popescu03, Popescu06}.
   As already mentioned, from results for the noncommutative Schur class
   $\cS_{nc,d}(\cU,\cY)$ one can arrive at interpolation results for the
   Drury-Arveson space as in Theorem \ref{T:DAint} by abelianizing
   the variables (see \cite{Arv98, DP98, AP00, BBF1}).
   We state here only the analogues of
   left-tangential operator-argument Nevanlinna-Pick and full
   Riesz-Dunford Nevanlinna-Pick interpolation for this
   noncommutative setting. The right-tangential versions again do not
   reduce to the left-tangential problems
   due to the same phenomenon discussed for the commutative case.
     In the noncommutative setting we
   distinguish between the {\em Riesz-Dunford functional calculus}
   which uses
   $$  f(Z) := \sum_{\gamma \in \cF_{d}} f_{\gamma} Z^{\gamma}
   $$
   where we set
   $
      Z^{\gamma} = Z_{i_{N}} \cdots Z_{i_{1}} \text{ if } \gamma =
      i_{N} \cdots i_{1}
    $,
   and the {\em transposed Riesz-Dunford functional calculus} which
   uses
   $$
   f^{\top}(Z): = \sum_{\gamma \in \cF_{d}} f_{\gamma} Z^{\gamma^{\top}}
   $$
   where $\gamma^{\top} = i_{1}, \dots, i_{N}$ if
   $\gamma = i_{N} \cdots i_{1}$ and
   $
    Z^{\gamma^{\top}} = Z_{i_{1}}  \cdots
  Z_{i_{N}} \text{ if } Z = (Z_{1}, \dots, Z_{d}).
   $

\begin{theorem}\label{T:ball-int}
\textbf{(1) Free-semigroup algebra Left-Tangential Nevanlinna-Pick
interpolation with Operator-Argument:}
Suppose that we are given a coefficient Hilbert space $\cC$ and a collection
$$
Z^{(1)} = (Z^{(1)}_{1}, \dots, Z^{(1)}_{d}), \dots,Z^{(N)} = (Z^{(N)}_{1}, \dots, Z^{(N)}_{d})
$$
of (not necessarily commutative) $d$-tuples of operators on $\cC$ such that, for each
fixed $i = 1, \dots, N$, the block row matrix
$\sbm{Z^{(i)}_{1} & \cdots &Z^{(i)}_{d}}$
defines a strict contraction operator from $\cC^{d} =\oplus_{k=1}^{d} \cC$ into $\cC$.
In addition, suppose that we are given operators $X_{1},\dots,X_{N}$ in $\cL(\cY,\cC)$ and
operators $Y_{1},\dots,Y_{N}$ in $\cL(\cU, \cC)$.  Then there exists a formal power series
$S(z) = \sum_{\gamma \in\cF_{d}} S_{\gamma} z^{\gamma}$ in the noncommutative Schur class
$\cS_{nc, d}(\cU, \cY)$ such that
       $$
         (X_{i} S)^{\wedge L}(Z^{(i)}): = \sum_{\gamma \in \cF_{d}}
     (Z^{(i)})^{\gamma^{\top}} X_{i} S_{\gamma} = W_{i}\quad\text{for}\quad i = 1, \dots, N
        $$
    if and only if the associated Pick matrix
    $$
    {\mathbb P}_{\tu{LTOA}} := \left[ \sum_{\gamma \in \cF_{d}} (Z^{(i)})^{\gamma} (X_{i}
    X_{j}^{*} - Y_{i} Y_{j}^{*}) (Z^{(j)})^{\gamma *}
    \right]_{i,j = 1}^N
    $$
    is positive semidefinite.

\medskip\noindent
\textbf{(2) Free-semigroup algebra full transposed Riesz-Dunford Nevanlinna-Pick
interpolation:}
Suppose that we are given $d$-tuples of operators $Z^{(1)},\ldots,Z^{(N)}$ as in
statement (1) above, acting on a separable auxiliary Hilbert space
$\cZ$ with orthonormal basis $\{e_{1}, \dots, e_{\kappa}\}$ (with possibly $\kappa=\infty$).
Suppose also that we are given operators $W_{1},\dots,W_{N}$ in $\cL(\cZ)$.
Then there exists a formal power series $s(z) = \sum_{\gamma \in\cF_{d}} s_{\gamma} z^{\gamma}$
in the scalar noncommutative Schur class $\cS_{nc, d}(\BC,\BC)$ satisfying the
transposed Riesz-Dunford interpolation conditions
    \begin{equation}  \label{ball-transRDint}
    s^{\top}(Z^{(i)}) : = \sum_{\gamma \in \cF_{d}} s_{\gamma}
    (Z^{(i)})^{\gamma^{\top}} = W_{i}\quad\text{for}\quad i = 1, \dots N
    \end{equation}
    if and only if the associated Pick matrix
    \begin{equation}   \label{ballP-transRD}
    {\mathbb P}_{\tu{FRD}} := \left[ \sum_{\gamma \in \cF_{d}}
    (Z^{(i)})^{\gamma} (e_{i'} e_{j'}^{*} - W_{i} e_{i'}
    e_{j'}^{*} W_{j}^{*}) (Z^{(j)})^{\gamma *} \right]_{(i,i'),
    (j,j') \in \{1, \dots, N\} \times \{1, \dots,\kappa\}}
    \end{equation}
    is positive semidefinite.

\medskip\noindent
\textbf{(3) Free-semigroup algebra full Riesz-Dunford Nevanlinna-Pick interpolation:}
Suppose that we are given $d$-tuples of operators $Z^{(1)},\ldots,Z^{(N)}$ as in
statement (2) above acting on a finite-dimensional auxiliary Hilbert space $\cZ$
with orthonormal basis $\{e_{1}, \dots, e_{\kappa}\}$ along with operators
$W_{1},\dots,W_{N}$ in $\cL(\cZ)$.  Then there exists a formal power series
$s(z) = \sum_{\gamma \in\cF_{d}} s_{\gamma} z^{\gamma}$ in the scalar noncommutative
Schur class $\cS_{nc, d}(\BC,\BC)$ satisfying the
    Riesz-Dunford-type interpolation conditions
    \begin{equation}   \label{ballRDint}
    s(Z^{(i)}) : = \sum_{\gamma \in \cF_{d}} s_{\gamma}
    (Z^{(i)})^{\gamma} = W_{i} \quad\text{for}\quad i = 1, \dots N
    \end{equation}
    if and only if the associated Pick matrix
    \begin{equation}   \label{ballP-RD}
    {\mathbb P}_{\tu{FRD*}} := \left[ \sum_{\gamma \in \cF_{d}}
    (Z^{(i)})^{\gamma*} (e_{i'} e_{j'}^{*} - W_{i}^{*} e_{i'}
    e_{j'}^{*} W_{j}) (Z^{(j)})^{\gamma} \right]_{(i,i'),
            (j,j') \in \{1, \dots, N\} \times \{1, \dots,\kappa\}}
    \end{equation}
    is positive semidefinite.
    \end{theorem}

    \begin{proof}  The first statement can be found in
        \cite{Popescu03} as well as in \cite{BB07}.  A tangential
        version of the second statement (i.e. tangential
        Riesz-Dunford interpolation) analogous to the
        Rosenblum-Rovnyak theory for the single-variable case is
        given in \cite{Popescu98}.
        The full transpose Riesz-Dunford interpolation result in the
        second statement follows from the first statement on
        \textbf{LTOA} in the same way as in Section \ref{S:1var}
        for the single-variable case.

      We reduce the third statement to the second as follows.
      By taking adjoints, we may rewrite \eqref{ball-transRDint} as
      \begin{equation}   \label{ballRDint1}
      \sum_{\gamma \in \cF_{d}} \overline{s}_{\gamma}
      (Z^{(i)})^{\gamma^{*}} = \sum_{\gamma \in \cF_{d}}
      \overline{s}_{\gamma} (Z^{(i)*})^{\gamma^{\top}} = W_{i}^{*}
      \end{equation}
      where we set
    $$
    Z^{(i)*} : = (Z_{1}^{(i)*}, \dots, Z_{d}^{(i)*}) \text{ for }
    i = 1, \dots, N.
    $$
      Note that
    $$
    \tau \colon \{ h_{\gamma}\}_{\gamma \in \cF_{d}} \mapsto
    \{ \overline{h}_{\gamma}\}_{\gamma \in \cF_{d}}
    $$
    is a conjugation (conjugate-linear norm-preserving
    involution) on $\ell^{2}(\cF_{d})$ and, if $R = \left[
    s_{\gamma \gamma^{\prime -1}}\right]_{\gamma, \gamma' \in \cF_{d}}$
    is the Toeplitz operator associated with the formal power
    series $s(z) = \sum_{\gamma \in \cF_{d}} s_{\gamma}
    z^{\gamma}$ acting on $\ell^{2}(\cF_{d})$, then
    $\tau \circ R \circ \tau = \overline{R}$ where
    $\overline{R} = \left[ \overline{s}_{\gamma \gamma^{\prime -1}}
    \right]_{\gamma, \gamma' \in \cF_{d}}$.
    We conclude that  $s(z) = \sum_{\gamma \in \cF_{d}} s_{\gamma}
    z^{\gamma}$ is in the noncommutative Schur-multiplier class $\cS_{nc,d}$
    if and only if $\overline{s}(z) = \sum_{\gamma \in \cF_{d}}
    \overline{s}_{\gamma} z^{\gamma}$ is in $\cS_{nc,d}$.  Thus
    there is an $s(z) = \sum_{\gamma \in \cF_{d}} s_{\gamma}
    z^{\gamma}$ in $\cS_{nc,d}$ satisfying \eqref{ballRDint1} if and only if
    there is another $s \in \cS_{nc,d}$ (namely, $\overline{s}$) satisfying
    \begin{equation}   \label{ballRDint2}
        s^{\top}(Z^{(i)*}): = \sum_{\gamma \in \cF_{d}}
        s_{\gamma} (Z^{(i)*})^{\gamma^{\top}} = W_{i}^{*} \text{
        for } i = 1, \dots, N.
    \end{equation}
    By part (2) of Theorem \ref{T:ball-int}, this last condition
    is equivalent to the positive-semidefiniteness of the matrix
    ${\mathbb P}_{\tu{FRD*}}$ in Part (3).  This completes the proof of Theorem
    \ref{T:ball-int}.
    \end{proof}

\subsection{Nevanlinna-Pick interpolation for Toeplitz algebras associated with directed graphs}
\label{subS:quiver1}

This example is discussed in \cite[Example 4.3 and pages 52--53]{MS04} and \cite[Section 5]{MSSchur}.
Here we work out the example, in particular, the nature of the point-evaluation and the Nevanlinna-Pick
theorem, in more detail.
The papers \cite{Muhly97, MS99, S04} introduced this class of algebras and studied the uniform
closure of the algebra (called the quiver algebra) generated by creation operators.
The associated weak-$*$ closed Toeplitz algebra obtained here is also known as
a {\em free semigroupoid algebra}; the papers \cite{KP04a, KP04b} give concrete
representations (some as explicit matrix-function algebras)
arising from particular choices of the quiver (i.e., directed
graph).

\paragraph{Quivers}
Formally a {\em quiver} $G$ is a quadruple $\{Q_0,Q_1,s,r\}$ that consists of two
finite sets $Q_0$ and $Q_1$ and two maps $s$ and $r$ from $Q_1$ to $Q_0$.
We think of the elements of $Q_0$ as vertices and those of $Q_1$ as arrows;
for any $\al\in Q_1$ we think of $\al$ as an arrow from $s(\al)$ to $r(\al)$.
It is possible to work with infinite sets $Q_0$ and $Q_1$ in which case they
need to be equipped with a topology \cite{KP04a,KP04b}, but we shall not consider
that case here. With the quiver $G$ we associate the {\em transposed
quiver} $\tilG=\{Q_0,Q_1,r,s\}$ (with respect to $G$), i.e., where the source and range maps are interchanged.
As a particular example, the reader is invited to take $Q_{0} = \{v_{0}\}$
(i.e., the set of vertices is a singleton) and $Q_{1} =\{\alpha_{1}, \dots, \alpha_{d}\}$
with $s(\alpha_{k}) = r(\alpha_{k}) = v_{0}$ for $k = 1, \dots, d$. The
present example then collapses to the free semigroup algebra example
$\cL_{d}$ discussed in Subsection \ref{subS:noncomball}. The difference
between the standard and the transposed Riesz-Dunford functional calculus
depends on whether one starts with the quiver $G$ or its reverse $\tilG$.

\paragraph{Paths}
For each nonnegative integer $n$ we write $Q_n$ for the paths of length $n$ and $\Gamma$
for the collection $\cup_{n=0}^{\infty} Q_{n}$ of all finite paths of whatever length.
Thus a $\gamma \in Q_n$ is an $n$-tuple $(\al_n,\ldots,\al_1)$ consisting of arrows
$\al_k\in Q_1$ with the property that $r(\al_k)=s(\al_{k+1})$ for $k=1,\ldots n-1$.
In that case we write $s_n(\gamma)$ for $s(\al_1)$ and $r_n(\gamma)$ for $r(\al_n)$.
Note that for $n=1$ this definition is consistent with that of $Q_1$ if the elements of
$Q_1$ are seen as paths of length 1. For $n=0$ we can view the elements of $Q_0$ as paths
of length 0, with for any $v\in Q_0$, $r_0(v)=s_0(v)=v$. If the length $n$ of $\gamma$
in $\Gamma$ is not specified, we write simply $s(\gamma)$ and $r(\gamma)$ rather than
$s_{n}(\gamma)$ and $r_{n}(\gamma)$.
The set of paths $\Ga$ associated with the quiver $G$ forms a semigroupoid when
multiplication is defined via concatenation: For
\[
\ga=(\al_n,\ldots,\al_1),\ga'=(\al'_m,\ldots,\al'_1)\in\Gamma\text{ with }
s(\ga)=r(\ga')
\]
we set
\[
\ga\cdot\ga':=(\al_n,\ldots,\al_1,\al'_m,\ldots,\al'_1).
\]
Inversion is then given by
\begin{equation}\label{QuivInv}
\ga\cdot\ga^{\prime -1} = \begin{cases} \gamma'' &\text{if } \gamma'' \in \Gamma
\text{ so that } \gamma =\gamma''\cdot\gamma', \\
   \text{undefined} & \text{otherwise.}
   \end{cases}
\end{equation}

\paragraph{The Fock space and Toeplitz algebra associated with the quiver $G$}

For a Hilbert space $\cU$ with direct sum decomposition $\cU=\oplus_{v\in Q_0}\cU_v$
we write, with some abuse of notation, $\ell^2_\cU(\Ga)$ for the space
$\oplus_{\ga\in\Ga}\cU_{r(\ga)}$; the $\cU$-valued {\em Fock space defined by the
quiver $G$}.

Given two coefficient Hilbert spaces $\cU=\oplus_{v\in Q_0}\cU_v$ and $\cY=\oplus_{v\in Q_0}\cY_v$,
the associated Banach space of Toeplitz operators $\fL_\Gamma(\cU,\cY)$ consists of operators $R$ from
$\ell^{2}_\cU(\Gamma)$ to $\ell^{2}_\cY(\Gamma)$ and hence can be given in terms of an
infinite operator-matrix with rows and columns indexed by $\Gamma$:
$$
  R = [ R_{\gamma, \gamma'}]_{\gamma, \gamma' \in \Gamma}
  \quad\text{with}\quad R_{\ga,\ga'}\in\cL(\cU_{r(\ga')},\cY_{r(\ga)}).
$$
The Toeplitz structure for this setting means that the matrix entries
$R_{\gamma, \gamma'}$ are completely determined from the particular
entries $R_\ga:=R_{\gamma, s(\ga)}$ according to the rule
$$
  R_{\gamma, \gamma'} =   R_{\gamma \gamma^{\prime -1},
  r( \gamma')}
$$
where we take $R_{\text{undefined},v} = 0$ for each $v\in Q_0$.

If $\cU=\cY$, then we write $\fL_\Gamma(\cU)$ instead of $\fL_\Gamma(\cU,\cU)$,
and $\fL_\Gamma(\cU)$ is an algebra that we call the {\em Toeplitz algebra}
associated with $G$ and $\cU$.
In the special case that $\cU_v=\cY_v=\BC$ for all $v\in Q_0$, this Toeplitz algebra
$\fL_\Gamma$ is otherwise known as the (weak-$*$ closed)
{\em path algebra} corresponding to the quiver $G$. The algebra $\fL_\Gamma$ also appears as
the weak-$*$ closed unital subalgebra of $\cL(\ell^2_\BC(\Ga))$ generated by the
creation operators $C_{\al}$, for $\al\in Q_1$, given by
\[
C_{\al}(\oplus_{\ga\in\Ga}f_\ga)=\oplus_{\ga'\in\Ga}\tilde f_{\ga'}\text{ where }
\tilde f_{\ga'}=\begin{cases} f_\ga &\text{if } \ga\in\Ga
\text{ so that } \al =\gamma'\cdot\gamma^{-1}, \\
   0 & \text{otherwise,}
   \end{cases}
\]
or, more succinctly, where
$$
 \tilde f_{\gamma'} = f_{\alpha^{-1} \gamma'}
$$
and we use the analogue of the convention \eqref{QuivInv} for the
case where the inverse path is on the left.

\paragraph{The Schur class associated with the quiver $G$}

Given a quiver $G$ and coefficient spaces $\cU$ and $\cY$ we define
the noncommutative Schur class $\cS_G(\cU,\cY)$  to be the set of
formal power series $S(z)=\sum_{\ga\in\Ga}S_\ga z^\ga$ in noncommutative
indeterminates $z=(z_\al\colon \al\in Q_1)$ (where $z^\ga=z_{\al_n}\cdots z_{\al_1}$
in case $\ga=(\al_n,\ldots,\al_1)\in\Ga$) such that the sequence of Taylor coefficients
$S_\ga\in\cL(\cU_{s(\ga)},\cY_{r(\ga)})$ defines an element $S=\mat{c}{S_{\ga,\ga'}}_{\ga,\ga'\in\Ga}$
(with $S_{\ga,\ga'}=S_{\ga\ga'^{-1}}$ and $S_\tu{undefined}=0$) in  $\fL_\Ga(\cU,\cY)$ of norm at most one.
One can also define a commutative analog of the Schur class
$\cS_G(\cU,\cY)$, but we will not develop this here.

Now assume we are also given an auxiliary Hilbert space $\cZ$ with direct sum decomposition
$\cZ=\oplus_{v\in Q_0}\cZ_v$. We then define the generalized unit disc $\BD_{G,\cZ}$ associated
with $\cZ$ and the quiver $G$ to be the set of tuples of operators of the form
$(Z_\al\in\cL(\cZ_{s(\al)},\cZ_{r(\al)})\colon\al\in Q_1)$ such that  for each $v\in Q_0$ the
row matrix formed by all $Z_\al$ with $r(\al)=v$ is a strict contraction:
\[
Z_v:=\mat{c}{Z_\al}_{\sm{\al\in Q_1\\r(\al)=v}}
:\oplus_{\sm{\al\in Q_1\\r(\al)=v}}\cZ_{s(\al)}\to\cZ_v\text{ satisfies }\|Z_v\|<1.
\]
In other words, the operator matrix
\[
\mat{c}{Z_{v,\al}}_{\sm{v\in Q_0\\\al\in Q_1}}:\oplus_{\al\in Q_1}\cZ_{s(\al)}\to\cZ
\text{ with }
Z_{v,\al}=\begin{cases} Z_\al &\text{if }v=r(\al), \\
   0 & \text{otherwise,}
   \end{cases}
\]
is a strict contraction. For $Z\in\BD_{G,\cZ}$ given by the tuple $(Z_\al\colon\al\in Q_1)$
and $\ga=(\al_n,\ldots,\al_1)\in\Ga$ we introduce the notation $Z^\ga$ for the operator
\begin{equation}\label{quiverpowers}
Z^\ga=Z_{\al_n}\cdots Z_{\al_1}:\cZ_{s_n(\ga)}\to\cZ_{r_n(\ga)}.
\end{equation}
In the sequel we shall use the abbreviations: $\cR_v=\cU_v\otimes\cZ_v$ and $\cQ_v=\cY_v\otimes\cZ_v$
for each $v\in Q_0$, and $\cR=\oplus_{v\in Q_0} \cR_v$ and $\cQ=\oplus_{v\in Q_0} \cQ_v$.

Given a Schur class function $S$ with Taylor coefficients $\{S_\ga\colon \ga\in\Ga\}$
and $Z\in\BD_{G,\cZ}$ we define the value of $S$ at $Z$ to be given by the tensor
functional-calculus:
\begin{equation}\label{Quiv-TFC}
S(Z)=\sum_{\ga\in\Ga} i_{\cQ_{r(\ga)}}(S_\ga\otimes Z^\ga)
i^*_{\cR_{s(\ga)}}
\in\cL(\cR,\cQ).
\end{equation}
Here we use the standard notation that for a subspace $\cV$ of a Hilbert space $\cW$ we write
$i_\cV$ for the canonical embedding of $\cV$ into $\cW$.

With respect to this tensor-product point-evaluation we consider what we will call the
{\em Quiver Left-Tangential Tensor functional calculus Nevanlinna-Pick
interpolation problem} ({\bf QLTT-NP}): {\em Given a data set
\[
\fD:\ Z^{(1)},\ldots,Z^{(N)}\in\BD_{G,\cZ},\
X_1,\ldots,X_N\in\cL(\cQ,\cC),\
Y_1,\ldots,Y_N\in\cL(\cR,\cC),
\]
where $\cC$ is an auxiliary Hilbert space, determine when there exists a Schur class
function $S$ in $\cS_G(\cU,\cY)$ that satisfies
\[
X_i S(Z^{(i)})=Y_i\quad\text{for}\quad i=1,\ldots,N.
\]}

The Riesz-Dunford functional calculus for this setting is then just the
tensor functional calculus for the special case that the coefficient spaces
$\cU$ and $\cY$ are both equal to $\oplus_{v\in Q_0}\BC$. Hence $\cR=\cQ=\cZ$
and for $Z\in\BD_{G,\cZ}$ and a Schur class function $s$ in $\cS_G$ with
Taylor coefficients $\{s_\ga\colon \ga\in\Ga\}$ the value of $s$ at $Z$ is
given by
\[
s(Z)=\sum_{\ga\in\Ga}s_\ga\cdot i_{\cZ_{r(\ga)}}Z^\ga i_{\cZ_{s(\ga)}}\in\cL(\cZ).
\]
The corresponding Nevanlinna-Pick problem is referred to as the {\em Quiver
Left-Tangential Riesz-Dunford Nevanlinna-Pick interpolation problem} ({\bf QLTRD-NP}).

We can also define an operator-argument functional-calculus for the Schur class
$\cS_G(\cU,\cY)$. For this purpose, assume we have another Hilbert space $\cX$,
again admitting an orthogonal direct sum decomposition of the form
$\cX=\oplus_{v\in Q_0}\cX_v$. The points in this case come from the generalized
disk $\BD_{\tilG,\cX}$, where $\tilG$ is the transposed quiver of $G$.
Thus a $T\in\BD_{\tilG,\cX}$ corresponds to a tuple
$(T_\al\in\cL(\cX_{r(\al)},\cX_{s(\al)})\colon\al\in Q_1)$ with the
property that the operator matrix
\[
\mat{c}{T_{v,\al}}_{\sm{v\in Q_0\\\al\in Q_1}}:\oplus_{\al\in Q_1}\cX_{r(\al)}\to\cX
\text{ with }
T_{v,\al}=\begin{cases} T_\al &\text{if }v=s(\al), \\
   0 & \text{otherwise}
   \end{cases}
\]
is a strict contraction. Given a $T=(T_\al\colon\al\in Q_1)\in\BD_{\tilG,\cX}$ and
$\ga=(\al_n,\ldots,\al_1)\in\Ga$, then $\gamma^{top}: = (\alpha_{1},
\dots, \alpha_{n}) \in \widetilde \Gamma$ and we set
\begin{equation}\label{QOAgenpow}
T^{\ga^{\top}}=T_{\al_1}\cdots T_{\al_n}:\cX_{r(\ga)}\to\cX_{s(\ga)}.
\end{equation}
Then for a Schur class function $S\in\cS_G(\cU,\cY)$, a $T\in\BD_{\tilG,\cX}$ and
a block diagonal operator $X=\diag_{v\in Q_0}(X_v)$, with $X_v\in\cL(\cY_v,\cX_v)$,
we define the left-tangential operator-argument point-evaluation $(XS)^{\wedge L}(T)$ by
\[
(XS)^{\wedge L}(T)
=\sum_{\ga\in\Ga} i_{\cX_{s(\ga)}}T^{\ga^{\top}} X_{r(\ga)}S_\ga i^*_{\cU_{s(\ga)}}
 = \sum_{\widetilde \gamma \in \widetilde \Gamma}
i_{\cX_{r(\widetilde \ga)}}T^{\widetilde \ga} X_{s(\widetilde
\ga)}S_{\widetilde \ga^{\top}} i^*_{\cU_{r( \widetilde \ga)}}
\in\cL(\cU,\cX).
\]
Notice that $(XS)^{\wedge L}(T)$ is a block diagonal operator in $\cL(\cU,\cX)$,
that is, the operator $(XS)^{\wedge L}(T)$ maps $\cU_v$ into $\cX_v$ for each $v\in Q_0$.

We then consider the {\em quiver Left-Tangential Nevanlinna-Pick interpolation problem with
Operator Argument} ({\bf QLTOA-NP}): {\em Given a data set
\[
T^{(i)}\in\BD_{\tilG,\cX},\
X^{(i)}=\diag_{v\in Q_0}(X^{(i)}_v),\
Y^{(i)}=\diag_{v\in Q_0}(Y^{(i)}_v),\text{ for }i=1,\ldots,N,
\]
where $X^{(i)}_v \in \cL(\cY_v,\cX_v)$ and $Y^{(i)}_v \in \cL(\cU_v,\cX_v)$, determine when there
exists a Schur class function $S$ in $\cS_G(\cU,\cY)$ that satisfies
\[
(X^{(i)}S)^{\wedge L}(T^{(i)})=Y^{(i)}\quad\text{for}\quad i=1,\ldots,N.
\]}

The solutions to these interpolation problems are given in the following theorem.
The proofs of these statements are given in Subsection \ref{subS:quiver2} below.

\begin{theorem}\label{T:Q-NP}
Let the data for the {\bf QLTT-NP}, {\bf QLTRD-NP} and {\bf QLTOA-NP} problems be as given above.
\begin{itemize}
\item[1.] Assume that $\cZ$ is separable, and that $\{e^{(v)}_1,\ldots,e_{\kappa_v}^{(v)}\}$
is an orthonormal basis for $\cZ_v$ for each $v\in Q_0$ (with possibly $\kappa_v=\infty$).
Then a solution to the {\bf QLTT-NP} interpolation problem exists if and only if for each
$v\in Q_0$ the associated Pick matrix $\BP_{QLTT}^{(v)}\in\cL(\cC)^{\kappa_v N\times\kappa_v N}$
given by
\[
\left[
\sum_{\ga\in\Ga,s(\ga)=v}\!\!\!\!\!\!\!
\begin{array}{l}
X_i i_{\cQ_{r(\ga)}}
\left(I_{\cY_{r(\ga)}}\otimes(Z^{(i)})^\ga e^{(v)}_{i'}e^{(v)*}_{j'}(Z^{(j)})^{\ga*}
\right)
i^*_{\cQ_{r(\ga)}}X_j^*+\\
\hspace*{.6cm}-Y_i i_{\cQ_{r(\ga)}}\left(I_{\cU_{r(\ga)}}\otimes(Z^{(i)})^\ga e^{(v)}_{i'}e^{(v)*}_{j'}(Z^{(j)})^{\ga*}
\right)i^*_{\cQ_{r(\ga)}}Y_j^*
\end{array}
\right]_{(i,i'),(j,j')}
\]
where $(i,i')$ and $(j,j')$ range over $\{1,\ldots,N\}\times\{1,\ldots\kappa_v\}$, is positive
semidefinite.

\item[2.] Assume that $\cC$ is separable, and that $\{e_1,\ldots,e_\kappa\}$ is an orthonormal basis
for $\cC$ (with possibly $\kappa=\infty$). Then a solution to the {\bf QLTRD-NP} interpolation problem exists if and only if
the Pick matrix $\BP_{QLTRD}\in\cL(\cZ)^{\kappa N\times\kappa N}$, for which the entry
corresponding to the pairs $(i,i'),(j,j',)\in\{1,\ldots,N\}\times\{1,\ldots,\kappa\}$ is given by
\[
\begin{array}{l}
\mat{c}{\BP_{QLTRD}}_{(i,i'),(j,j')}=\\[.2cm]
     \hspace*{.6cm}=\displaystyle\sum_{\ga\in\Ga}
i_{\cZ_{s(\ga)}}
(Z^{(i)})^{\ga*} i^*_{\cZ_{r(\ga)}}
(X_i^*e_{i'}e_{j'}^*X_j-Y_i^*e_{i'}e_{j'}^*Y_j)
i_{\cZ_{r(\ga)}}(Z^{(j)})^\ga i^*_{\cZ_{s(\ga)}}
\end{array}
\]
is positive semidefinite.

\item[3.] A solution to the {\bf QLTOA-NP} interpolation problem exists if and only if
the Pick matrix
\[
\BP_{QLTOA}=\mat{c}{\displaystyle\sum_{\ga\in\Ga}
i_{\cX_{s(\ga)}}(T^{(i)})^\ga(X^{(i)}_{r(\ga)}X_{r(\ga)}^{(j)*}
-Y^{(i)}_{r(\ga)}Y_{r(\ga)}^{(j)*})(T^{(j)})^{\ga*}i_{\cX_{s(\ga)}}^*}_{i,j=1}^N
\]
is positive semidefinite.
\end{itemize}
\end{theorem}

We conclude this subsection with an example for a concrete quiver; an example also considered
in \cite{KP04a}. Let $G=\{Q_0,Q_1,s,r\}$ be the quiver with two vertices
$Q_0=\{a,b\}$, two arrows $Q_1=\{\al,\be\}$, and source and range map given by
\[
s(\al)=r(\al)=a,\quad s(\be)=a\ands r(\be)=b.
\]
With the vertices $a$ and $b$ we associate Hilbert spaces $\cA$ and $\cB$, and we consider
the Toeplitz algebra $\mathfrak{L}_\Ga(\cA\oplus\cB)$. Here $\Ga$ is the path
semigroupoid of $G$ which equals
\[
\Ga=\{\al^n,\be\al^n,b\colon n\in\BZ_+\},
\]
where $\al^n$ and $\be\al^n$ are abbreviations for $(\al,\ldots,\al)$ (with length $n$)
and $(\be,\al,\ldots,\al)$ (with length $n+1$), respectively, and $\al^0=a$. The elements
in the Toeplitz algebra are then given by infinite tuples
\begin{equation}\label{QexToep}
R=(V_n,\,W_n,\,B_0\colon n\in\BZ_+, V_n\in\cL(\cA),W_n\in\cL(\cA,\cB),B_0\in\cL(\cB))
\end{equation}
with the property that the infinite operator matrix
\begin{equation}\label{Toep1}
\mat{ccccccc}{
V_0&0&0&0&0&0&\cdots\\
0&B_0&0&0&0&0&\cdots\\
V_1&0&V_0&0&0&0&\cdots\\
W_0&0&0&B_0&0&0&\cdots\\
V_2&0&V_1&0&V_0&0&\cdots\\
W_1&0&W_0&0&0&B_0&\\
\vdots&\vdots&\vdots&\vdots&\vdots&&\ddots}
\end{equation}
defines a bounded operator on $\ell^2_{\cA\oplus\cB}(\BZ_+)$; the norm of $R$ in
$\mathfrak{L}_\Ga(\cA\oplus\cB,\cA\oplus\cB)$ is equal to the operator norm of
the operator matrix in $\cL(\ell^2_{\cA\oplus\cB}(\BZ_+))$. After rearranging
rows and columns it follows that the operator norm of (\ref{Toep1}) is the same
as that of
\[
\mat{ccccccc}{
B_0&0&0&0&0&0&\cdots\\
0&V_0&0&0&0&0&\cdots\\
0&W_0&B_0&0&0&0&\cdots\\
0&V_1&0&V_0&0&0&\cdots\\
0&W_1&0&W_0&B_0&0&\cdots\\
0&V_2&0&V_1&0&V_0&\\
\vdots&\vdots&\vdots&\vdots&\vdots&\vdots&\ddots}
\]
as an element in $\cL(\cB\oplus\ell^2_{\cA\oplus\cB}(\BZ_+))$, and one can even leave out
the first row and column. Thus $R$ can be identified with the functions
$V\in H_{\cL(\cA)}^\infty(\BD)$ and $W\in H_{\cL(\cA,\cB)}^\infty(\BD)$ given by
\begin{equation}\label{VandW}
V(\la)=\sum_{n=0}^\infty \la^nV_n,\quad W(\la)=\sum_{n=0}^\infty \la^nW_n,
\end{equation}
and the constant function with value $B_0$, which we also denote by $B_0$.
The norm of $R$ is then equal to the norm of the multiplication operator
\begin{equation}\label{multop}
\mat{cc}{M_V&0\\M_W&M_{B_0}}\text{ on }\mat{c}{H^2_\cA(\BD)\\H^2_\cB(\BD)},
\end{equation}
where $M_V$, $M_W$ and $M_{B_0}$ denote the multiplication operators for the
functions $V$, $W$ and $B_0$, respectively. The Toeplitz algebra
$\mathfrak{L}_\Ga(\cA\oplus\cB)$ can thus be identified with the algebra
\[
\mat{cc}{H_{\cA}^\infty(\BD)&0\\H_{\cL(\cA,\cB)}^\infty(\BD)&\cL(\cB)},
\]
(with $\cL(\cB)$ identified with the space of constant functions), which
is easily seen to be isometrically isomorphic to the algebra
\[
\mat{cc}{H_{\cA}^\infty(\BD)&0\\H_{\cL(\cA,\cB),0}^\infty(\BD)&\cL(\cB)}.
\]
Here $H_{\cL(\cA,\cB),0}^\infty(\BD)$ is the Banach space of functions
$W\in H_{\cL(\cA,\cB)}^\infty(\BD)$ with $W(0)=0$.
The identification with the latter algebra was already obtained in \cite{KP04a}.

We first consider Nevanlinna-Pick interpolation for the Riesz-Dunford
functional calculus, i.e., when $\cA=\cB=\BC$. Let $\cZ=\cZ_a\oplus\cZ_b$
be an auxiliary Hilbert space. The generalized unit disk $\BD_{G,\cZ}$
then consists of all pairs of operators $(Z_\al,Z_\be)$ with
$Z_\al\in\cL(\cZ_a)$ and $Z_\be\in\cL(\cZ_a,\cZ_b)$ such that $\sbm{Z_a\\Z_b}$
is a strict contraction. Given a point $(Z_\al,Z_\be)$ in $\BD_{G,\cZ}$
and an
\[
R=(v_n,w_n,b_0\in\BC\colon n\in\BZ_+)\in\mathfrak{L}_\Ga,
\]
the point evaluation of $R$ in $(Z_\al,Z_\be)$ is given by
\[
R(Z_\al,Z_\be)=\mat{cc}{\sum_{n=0}^\infty v_nZ_\al^n&0\\[.1cm]
Z_\be\sum_{n=0}^\infty w_nZ_\al^n&b_0I_{\cZ_b}}
\text{ on }\mat{c}{\cZ_a\\\cZ_b}.
\]
Assume we are given data for the {\bf QLTRD-NP} problem:
\[
(Z_{\al}^{(i)},Z_{\be}^{(i)})\in\BD_{G,\cZ},\
X^{(i)}=\mat{cc}{X_{a}^{(i)}&X_{b}^{(i)}},\ Y^{(i)}=\mat{cc}{Y_{a^{(i)}}&Y_{b}^{(i)}}\text{ for }i=1,\ldots,N,
\]
with $X^{(i)},Y^{(i)}\in\cL(\cZ,\cC)$, where $\cC$ is some separable Hilbert space with orthonormal
basis $\{e_1\ldots,e_\kappa\}$. It then follows from Theorem \ref{T:Q-NP} that there
exists an $S\in\mathfrak{L}_\Ga$ with $\|S\|\leq 1$ such that
\[
X^{(i)}S(Z_{\al}^{(i)},Z_{\be}^{(i)})=Y^{(i)}\text{ for }i=1,\ldots,N
\]
if and only if the Pick matrices $\BP_{QLTRD}^{(1)}$ and $\BP_{QLTRD}^{(2)}$,
with $\BP_{QLTRD}^{(1)}$ given by
\[
\mat{c}{\displaystyle\sum_{n=0}^\infty (Z_{\al}^{(i)})^{n*}\left(\begin{array}{l}
X_{a}^{(i)*}e_{i'}e_{j'}^*X_{a}^{(j)}
+Z_{\be}^{(i)*}X_{b}^{(i)*}e_{i'}e_{j'}^*X_{b}^{(j)}Z_{\be}^{(j)}+\\[.1cm]
-Y_{a}^{(i)*}e_{i'}e_{j'}^*Y_{a}^{(j)}
-Z_{\be}^{(i)*}Y_{b}^{(i)*}e_{i'}e_{j'}^*Y_{b}^{(j)}Z_{\be}^{(j)}
\end{array}\right)(Z_{\al}^{(j)})^{n}}_{(i,i'),(j,j')}
\]
and
\[
\BP_{QLTRD}^{(2)}=\mat{c}{X_{b}^{(i)*}e_{i'}e_{j'}^*X_{b}^{(j)}-Y_{b}^{(i)*}e_{i'}e_{j'}^*Y_{b}^{(j)}}_{(i,i'),(j,j')},
\]
are both positive semidefinite. The range of the pairs $(i,i')$ and $(j,j')$ in the definition of
the Pick matrices is $\{1,\ldots,N\}\times\{1,\ldots,\kappa\}$.

Now assume that $S=(v_n,w_n,b_0\in\BC\colon n\in\BZ_+)\in\mathfrak{L}_\Ga$ is a solution. Then $b_0$
necessarily satisfies $|b_0|\leq1$ and $b_0 X_b^{(i)}=Y_b^{(i)}$ for $i=1,\ldots,N$. The existence of
a number $b_0$ with these properties turns out to be equivalent to the positive semidefiniteness of
the pick matrix $\BP_{QLTRD}^{(2)}$, which is the content of the following lemma.

\begin{lemma}\label{L:ConstMult}
For $i=1,\ldots,N$ let $X_i,Y_i\in\cL(\cH,\cK)$, where $\cH$ and $\cK$
are Hilbert spaces and $\cH$ is separable with orthonormal basis
$\{e_1,\ldots,e_\kappa\}$. Then
\begin{equation}\label{XYmat}
\mat{c}{X_ie_{i'}e_{j'}X_j^*-Y_ie_{i'}e_{j'}Y_j^*}_{(i,i'),(j,j')\in\{1,\ldots,N\}\times\{1,\ldots,\kappa\}}
\end{equation}
is a positive semidefinite if and only if there exist a
$\de\in\BC$ with $|\de|\leq1$ such that $\de X_i=Y_i$ for all $i=1,\ldots,N$.
\end{lemma}

\begin{proof}
Set
\[
X=\mat{c}{X_1\\\vdots\\X_N}\ands Y=\mat{c}{Y_1\\\vdots\\Y_N}.
\]
Then (\ref{XYmat}) is unitarily equivalent, via a permutation matrix, to
\begin{equation}\label{rank1s}
\tu{col}_{i'\in\{1,\ldots,\kappa\}}(Xe_{i'})(\tu{col}_{i'\in\{1,\ldots,\kappa\}}(Xe_{i'}))^*
-\tu{col}_{i'\in\{1,\ldots,\kappa\}}(Ye_{i'})(\tu{col}_{i'\in\{1,\ldots,\kappa\}}(Ye_{i'}))^*.
\end{equation}
Thus positive semidefiniteness of (\ref{XYmat}) corresponds to positive semidefiniteness
of (\ref{rank1s}). It follows right away from the Douglas factorization lemma \cite{D66}
that (\ref{rank1s}) being positive semidefinite is equivalent to the existence of a $\de\in\BC$
with $|\de|\leq1$ such that
$\de \tu{col}_{i'\in\{1,\ldots,\kappa\}}(Xe_{i'})=\tu{col}_{i'\in\{1,\ldots,\kappa\}}(Ye_{i'})$,
which is the same as $\de X_i=Y_i$ for all $i=1,\ldots,N$.
\end{proof}

Next assume that both Pick matrices $\BP_{QLTRD}^{(1)}$ and $\BP_{QLTRD}^{(2)}$ are positive
semidefinite. According to Lemma \ref{L:ConstMult} there exists a $b_0$ with $|b_0|\leq1$ such
that $b_0 X_b^{(i)}=Y_b^{(i)}$ for $i=1,\ldots,N$. Using this number $b_0$ we can rewrite
$\BP_{QLTRD}^{(2)}$ as
\[
\mat{c}{\displaystyle\sum_{n=0}^\infty (Z_{\al}^{(i)})^{n*}\left(\begin{array}{l}
X_{a}^{(i)*}e_{i'}e_{j'}^*X_{a}^{(j)}-Y_{a}^{(i)*}e_{i'}e_{j'}^*Y_{a}^{(j)}+\\[.1cm]
\hspace*{1cm}+(1-|b_0|^2)Z_{\be}^{(i)*}X_{b}^{(i)*}e_{i'}e_{j'}^*X_{b}^{(j)}Z_{\be}^{(j)}
\end{array}\right)(Z_{\al}^{(j)})^{n}}_{(i,i'),(j,j')},
\]
where again the range of the pairs $(i,i')$ and $(j,j')$ in the definition of
the Pick matrices is $\{1,\ldots,N\}\times\{1,\ldots,\kappa\}$. We then distinguish between
two case, namely (1) $|b_0|=1$ and (2) $|b_0|<1$. In case $|b_0|=1$, for
$S=(v_n,w_n,b_0\in\BC\colon n\in\BZ_+)\in\mathfrak{L}_\Ga$ to be a solution it is
necessary that $w_n=0$ for all $n\in\BZ_+$ due to the norm constraint $\|S\|\leq 1$.
On the other hand, the Pick matrix $\BP_{QLTRD}^{(1)}$ reduces to
\[
\mat{c}{\displaystyle\sum_{n=0}^\infty (Z_{\al}^{(i)})^{n*}\left(
X_{a}^{(i)*}e_{i'}e_{j'}^*X_{a}^{(j)}-Y_{a}^{(i)*}e_{i'}e_{j'}^*Y_{a}^{(j)}
\right)(Z_{\al}^{(j)})^{n}}_{(i,i'),(j,j')\in\{1,\ldots,N\}\times\{1,\ldots,\kappa\}}
\]
which is a Pick matrix of the form appearing in Part 2 of Theorem \ref{T:RD-NP}. In fact,
it is the pick matrix for  the {\bf LTRD-NP} problem for the scalar-valued Schur class $\cS$
with data
\[
Z_i=Z_\al^{(i)},\ X_i=X_a^{(i)},\ Y_i=Y_a^{(i)},\quad i=1,\ldots,N.
\]
Applying Theorem \ref{T:RD-NP} to this data set we obtain a $v\in\cS$ with
$X_a^{(i)} v(Z_\al^{(i)})=Y_a^{(i)}$ for $i=1,\ldots,N$. Let $v_0,v_1,\ldots$ be the
Taylor coefficients of $v$. It is then not difficult to see
that $S=(v_n,w_n,b_0\in\BC\colon n\in\BZ_+,\, w_n=0)$ is a solution.
The case that $|b_0|<1$ does not reduce to a {\bf LTRD-NP} problem, but rather to
a left-tangential tensor functional-calculus Nevanlinna-Pick ({\bf LTT-NP}) problem
which is a type of problem we discuss in Subsection \ref{subS:MS-1var} below. The problem
can then be solved directly using some of the techniques developed there, but we will
not work out the details here.

Next we specify the left-tangential operator-argument Nevanlinna-Pick result for
our example. Let $\cX=\cX_\al\oplus\cX_\be$ be a given
Hilbert space. in this case, points are elements of the generalized
disk $\BD_{\tilG,\cX}$ (with $\tilG$ the transposed quiver of $G$), which corresponds
to the set of pairs $(T_\al,T_\be)$ with $T_\al\in\cL(\cX_a,\cX_a)$ and
$T_\be\in\cL(\cX_b,\cX_a)$ such that the row operator $\sbm{T_\al&T_\be}$ is a strict
contraction. Now let $R\in\mathfrak{L}_\Ga(\cA\oplus\cB)$ be given by (\ref{QexToep})
and $(T_\al,T_\be)\in\BD_{\tilG,\cX}$. If in addition we are also given a block
diagonal operator $X=\diag(X_a,X_b)$ with $X_a\in\cL(\cA,\cX_a)$ and $X_b\in\cL(\cB,\cX_b)$,
then the left-tangential operator-argument point evaluation $(XR)^{\wedge L}(T_\al,T_\be)$
is given by
\[
(XR)^{\wedge L}(T_\al,T_\be)
=\mat{cc}{\displaystyle\sum_{n=0}^\infty T_\al^n(X_aV_n+T_\be X_bW_n)&0\\0&X_bB_0}.
\]
The data for the {\bf QLTOA-NP} problem
in then given by
\[
(T_{\al}^{(i)},T_{\be}^{(i)})\in\BD_{\tilG,\cX},\ X^{(i)}=\diag(X_{a}^{(i)},X_{b}^{(i)}),\
Y^{(i)}=\diag(Y_{a}^{(i)},Y_{b}^{(i)})
\text{ for }i=1,\ldots,N
\]
with $X_{a}^{(i)},Y_{a}^{(i)}\in\cL(\cA,\cX_a)$ and $X_{b}^{(i)},Y_{b}^{(i)}\in\cL(\cB,\cX_b)$,
and it follows from Part 3 of Theorem \ref{T:Q-NP} that there exists an
$S\in\mathfrak{L}_\Ga(\cA\oplus\cB)$ with $\|S\|\leq1$ such that
\begin{equation}\label{QexLOA}
(X^{(i)}S)^{\wedge L}(T_{\al}^{(i)},T_{\be}^{(i)})=Y^{(i)}\quad\text{for}\quad i=1,\ldots,N
\end{equation}
if and only if the Pick matrix $\BP_{QLTOA}$ in Part 3 of Theorem \ref{T:Q-NP} is positive
semidefinite. In this case, after rearranging columns and rows, $\BP_{QLTOA}$ can be identified
with
\[
\BP_{QLTOA}=\sbm{\BP_{QLTOA}^{(1)}&0\\0&\BP_{QLTOA}^{(2)}}
\]
where $\BP_{QLTOA}^{(1)}$ is the Pick matrix given by
\[\small
\mat{c}{\displaystyle\sum_{n=0}^\infty (T_{\al}^{(i)})^n
(X_{a}^{(i)}X_{a}^{(j)*}
+T_{\be}^{(i)}X_{b}^{(i)}X_{b}^{(j)*}T_{\be}^{(j)*}
-Y_{a}^{(i)}Y_{a}^{(j)*}
-T_{\be}^{(i)}Y_{b}^{(i)}Y_{b}^{(j)*}T_{\be}^{(j)*})(T_{\al}^{(j)})^{n*}}_{i,j=1}^N
\]
and
\[
\BP_{QLTOA}^{(2)}=\mat{c}{X_{b}^{(i)}X_{b}^{(j)*}-Y_{b}^{(i)}Y_{b}^{(j)*}}_{i,j=1}^N.
\]

To see the sufficiency of this Pick matrix criterion, assume that $\BP_{QLTOA}$ is
positive semidefinite, and thus, equivalently, that $\BP_{QLTOA}^{(1)}$ and $\BP_{QLTOA}^{(2)}$
are positive semidefinite. Notice that $\BP_{QLTOA}^{(2)}$ can also be written as
\[
\mat{c}{X^{(1)}_b\\\vdots\\X^{(N)}_b}\mat{ccc}{X^{(1)*}_b&\cdots&X^{(N)*}_b}
-\mat{c}{Y^{(1)}_b\\\vdots\\Y^{(N)}_b}\mat{ccc}{Y^{(1)*}_b&\cdots&Y^{(N)*}_b}.
\]
Hence the positive semidefiniteness of $\BP_{QLTOA}^{(2)}$, again using Douglas factorization
lemma, corresponds to the existence of a contraction $B_0\in\cL(\cB)$ with
\[
X^{(i)}_bB_0=Y^{(i)}_b\quad\text{for}\quad i=1,\ldots,N.
\]
Let $D_{B_0^*}$ denote the defect operator of $B_0^*$, that is, $D_{B_0^*}$ is the positive
quare root of $I_\cB-B_0B_0^*$. We can then rewrite the first Pick
matrix $\BP_{QLTOA}^{(1)}$ as
\[
\mat{c}{\displaystyle\sum_{n=0}^\infty (T_{\al}^{(i)})^n
(X_{a}^{(i)}X_{a}^{(j)*}
+T_{\be}^{(i)}X_{b}^{(i)}D_{B_0^*}^2X_{b}^{(j)*}T_{\be}^{(j)*}
-Y_{a}^{(i)}Y_{a}^{(j)*})(T_{\al}^{(j)})^{n*}}_{i,j=1}^N.
\]
In this form $\BP_{QLTOA}^{(1)}$ is a Pick matrix of the type appearing in Part 1 of
Theorem~\ref{T:LTOA/RTOA}. In fact, it is the Pick matrix for the {\bf LTOA-NP} problem for functions
from the Schur class $\cS(\cA,\cA\oplus\cB)$ with data
\begin{equation}\label{QexData}
T_i=T_\al^{(i)},\ X_i=\mat{cc}{X_a^{(i)}&T_\be^{(i)}X_b^{(i)}D_{B_0^*}},\
Y_i=Y_a^{(i)},\ \text{ for }i=1,\ldots,N.
\end{equation}
Applying Theorem \ref{T:LTOA/RTOA} to this data set, we obtain a function
\begin{equation}\label{H}
H=\mat{c}{V\\\tilW}\in\cS(\cA,\cA\oplus\cB)
\end{equation}
with
\[
(X_iH)^{\wedge L}(T_i)=Y_i\quad\text{for}\quad i=1,\ldots,N.
\]
In other words, the Taylor coefficients $V_0,V_1,\ldots$ of $V$ and $\tilW_0,\tilW_1,\ldots$
of $\tilW$ satisfy
\[
\sum_{n=0}^\infty (T_\al^{(i)})^n(X_a^{(i)}V_n+T_\be^{(i)} X_b^{(i)}D_{B_0^*}\tilW_n)=Y_a^{(i)}\text{ for }i=1,\ldots,N.
\]
Now set $W_n=D_{B_0^*}\tilW_n$ for each $n\in\BZ_+$.
It follows that $S=(V_n,W_n,B_0\colon n\in\BZ_+)$ is in $\mathfrak{L}_\Ga(\cA\oplus\cB)$
and satisfies the interpolation conditions (\ref{QexLOA}), and it is not difficult to see
that $\|S\|\leq 1$. Thus $S$ is a solution to the {\bf QLTOA-NP} problem for the quiver
$G$ considered in this example. It is also possible to provide a direct proof of the
necessity of the Pick matrix condition; we leave the details as an exercise for the
interested reader.

\subsection{The polydisk setting:~commutative and noncommutative}
\label{subS:polydisk}

For the setting of the polydisk ${\mathbb D}^{d} = \{ \lambda =
(\lambda_{1}, \dots, \lambda_{d}) \colon |\lambda_{k}| < 1 \text{ for
} k=1, \dots, d\}$, the results concerning Nevanlinna-Pick-like
interpolation are of a different flavor.  We define what is now
called the $d$-variable {\em Schur-Agler class} $\mathcal{SA}_{d}$ to
consist of those holomorphic complex-valued functions $s(\lambda) =
\sum_{n \in {\mathbb Z}^{d}_{+}} s_{n} \lambda^{n}$ on ${\mathbb
D}^{d}$ with the property that, for every commutative $d$-tuple $Z =
(Z_{1}, \dots, Z_{d})$ of strict contraction operators ($\|Z_{k}\| <
1$ for $k = 1, \dots, d$) on a Hilbert space $\cK$, it happens that
the resulting operator
$$
  s(Z) = \sum_{n \in {\mathbb Z}^{d}_{+}} s_{n} Z^{n}
$$
($d$-variable Riesz-Dunford functional calculus) has $\| s(Z)\| \le 1$.
The special choice $Z=(\lambda_{1}I_{\cK}, \dots, \lambda_{d}I_{\cK})$
with $\lambda = (\lambda_{1}, \dots, \lambda_{d}) \in {\mathbb
D}^{d}$ shows that the Schur-Agler class is a subset of the Schur
class (defined to be the set of holomorphic functions mapping
${\mathbb D}^{d}$ into the closed unit disk $\overline{\mathbb D}$).
The converse holds for the cases $d=1$ and $d=2$ as a consequence of
the von Neumann inequality holding for these two cases, as can be
seen from the Sz.-Nagy dilation theorem for the case $d=1$ and from
the And\^o dilation theorem for the case $d=2$---see
e.g.~\cite{AgMcC}.

In fact, as a consequence of the Drury-von
Neumann inequality/dilation
theorem for commutative row contractions $Z = (Z_{1}, \dots, Z_{d})$
\cite{Dr}
and the Popescu-von Neumann inequality/dilation theorem
\cite{Popescu91}, a Schur-Agler type characterization also holds for
the Drury-Arveson Schur-multiplier class $\cS_{d}$ and the
free-semigroup algebra $\cL_{d}$:  {\em a holomorphic function $s$ on
the ball ${\mathbb B}^{d}$ is in the Drury-Arveson Schur-multiplier
class $\cS_{d}$ if and only if $\|s(Z) \| \le 1$ for all commutative
$d$-tuples $Z = (Z_{1}, \dots, Z_{d})$ for which the block row matrix
${\mathbf Z} = \begin{bmatrix} Z_{1} & \cdots & Z_{d} \end{bmatrix}$ is a
strict contraction}, and, similarly, {\em a formal power series $s(z)
= \sum_{\gamma \in \cF_{d}} s_{\gamma} z^{\gamma}$ is in the
noncommutative Schur-multiplier class $\cS_{nc, d}$ if and only if
$s(Z) = \sum_{\gamma \in \cF_{d}} s_{\gamma} Z^{\gamma}$ has
$\|s(Z)\| \le 1$ for any (not necessarily commutative) $d$-tuple
$Z = (Z_{1}, \dots, Z_{d})$ for which the row matrix ${\mathbf Z}: =
\begin{bmatrix} Z_{1} & \cdots & Z_{d} \end{bmatrix}$ is a strict
contraction.}

The main result on interpolation for the Schur-Agler
class is the following result of Agler.

\begin{theorem}  \label{T:Agler} (See \cite{Ag88, Ag90, AgMcC99, AgMcC}.)
    Suppose that we are given a subset $X$ of ${\mathbb D}^{d}$ and a
    function $f \colon X \to {\mathbb C}$.  Then there exists a
    function $s \colon {\mathbb D}^{d} \to {\mathbb C}$ in the
    Schur-Agler class $\mathcal{SA}_{d}$ such that
    $$
      s|_{X} = f
    $$
    if and only if there exist $d$ positive kernels $K_{1}, \dots,
    K_{d}$ on $X \times X$ so that
    \begin{equation}  \label{Agler-decom}
    1-f(\lambda) \overline{f(\zeta)} =
    \sum_{k=1}^{d} (1 - \lambda_{k} \overline{\zeta_{k}})
    K_{k}(\lambda, \zeta) \text{ for all } \lambda, \zeta \in X.
    \end{equation}
    \end{theorem}

    Given two Hilbert spaces $\cU$ and $\cY$, the operator-valued
    version $\mathcal{SA}_{d}(\cU, \cY)$ of the Schur-Agler class can
    be defined via tensor functional calculus as follows.  Given a
    holomorphic function $S(\lambda) = \sum_{n \in {\mathbb Z}^{d}_{+}}
    S_{n} \lambda^{n}$ with coefficients (and hence also values) in
    $\cL(\cU, \cY)$ and given a commutative $d$-tuple $Z: =
    (Z_{1}, \dots, Z_{d})$ of operators on another auxiliary Hilbert
    space $\cK$  such that each $Z_{k}$ is a strict contraction for
    $k = 1, \dots, d$, we define $S(Z)$ via
    $$
      S(Z) = \sum_{n \in {\mathbb Z}^{d}_{+}} S_{n} \otimes Z^{n} \in
      \cL(\cU \otimes \cK, \cY \otimes \cK),
    $$
    using the (commutative) multivariable notation for $Z^n$ defined in
    Part 3 of Theorem \ref{T:DAint}.
    We then say that $S \in \mathcal{SA}_{d}(\cU, \cY)$ if $\|S(Z)\|
    \le 1$ whenever $Z = (Z_{1}, \dots, Z_{d})$ is a commutative
    $d$-tuple of strict contraction operators on $\cK$.  We will not
    state the results precisely here, but rather merely mention that
    the extensions to left- and right-tangential interpolation have
    been given in  \cite{BT}.  The theory can be generalized to an
    arbitrary domain $\cD$ with polynomial-matrix defining function
    ($\cD = \cD_{Q} = \{ \lambda \in {\mathbb C}^{d} \colon \| Q(\lambda) \| <
    1\}$ for a fixed matrix polynomial $Q(\lambda)$)---see \cite{AT,
    BB04}; note that the polydisk corresponds to the case
    $$
     Q(\lambda) = \begin{bmatrix}  \lambda_{1} & & \\ & \ddots & \\ &
     & \lambda_{d} \end{bmatrix}.
     $$
    In this general setting the analysis of the left- and right-tangential
    Nevanlinna-Pick problem with operator-argument (using
    an adaptation of the Taylor-Vasilescu functional calculus) has
    been worked out in \cite{BB05}.

    A noncommutative version of the Schur-Agler class can be defined
    as follows.  Given coefficient Hilbert spaces $\cU$ and $\cY$ and
    a formal power series $S$ of the form $S(z) = \sum_{\gamma \in
    \cF_{d}} S_{\gamma} z^{\gamma}$, where for each $\ga$ in the free
    semigroup $\cF_d$ defined in Subsection \ref{subS:noncomball} the coefficient
    $S_{\gamma}$ is in $\cL(\cU, \cY)$, we say that $S$ is in the
    noncommutative operator-valued $d$-variable Schur-Agler class
    $\mathcal{SA}_{nc,d}(\cU, \cY)$ if, for any (not necessarily
    commutative) $d$-tuple $Z =
    (Z_{1}, \dots, Z_{d})$ of strict contraction operators, it is the
    case that
    $$
    S(Z) = \sum_{\gamma \in \cF_{d}} S_{\gamma} \otimes Z^{\gamma}
    $$
    has $\|S(Z)\| \le 1$. Here we use the noncommutative multivariable notation
    of Subsection \ref{subS:comball}: $Z^\ga=Z_{i_N}\cdots Z_{i_1}$ if $\ga=i_N\ldots i_1$.
It is one of the results of \cite{A-KV} that
    in fact one need only check $\|S(Z)\| \le 1$ for $Z = (Z_{1},
    \dots, Z_{d})$ a $d$-tuple of matrices of finite size $\kappa\times\kappa$ for
    arbitrary $\kappa = 1,2,\dots$.

    We then have the following result for
    left-tangential Nevanlinna-Pick interpolation with operator-argument
    on the noncommutative polydisk.

    \begin{theorem}  \label{T:NCpolydisk-int}
   {\em \textbf{(1) Left-Tangential interpolation
    with Operator-Argument for the noncommutative polydisk:}}
    Suppose that we are given $N$ $d$-tuples of (not necessarily
    commutative) strict contraction operators
    \[
    T^{(1)}=(T^{(1)}_{1}, \dots, T^{(1)}_{d}),\ldots,
    T^{(N)}=(T^{(N)}_{1}, \dots, T^{(N)}_{d})
    \]
    on a Hilbert space $\cC$
    together with operators $X_{1}$, $\dots$, $X_{N}$ in
    $\cL(\cY,\cC)$ and $Y_{1}$, $\dots$, $Y_{N}$ in $\cL(\cU,
    \cC)$.  Then there exists a formal power series $S(z) =
    \sum_{\gamma \in \cF_{d}} S_{\gamma} z^{\gamma}$ in the
    noncommutative $d$-variable operator-valued Schur class $\cS_{nc,d}(\cU,
    \cY)$ such that
    $$
    (X_{i}S)^{\wedge L}(T^{(i)}) :=
    \sum_{\gamma \in \cF_{d}} (T^{(i)})^{\gamma^{\top}} X_{i}
    S_{\gamma} = Y_{i} \text{ for } i=1, \dots, N
    $$
    if and only if there exist $d$ positive semidefinite block
    matrices  $K_{1}, \dots, K_{d}$ with entries in $\cL(\cC)$
    of the form
    $$
      K_{k} = \left[ K_{k}(i,j) \right]_{i,j=1, \dots,
      N}
    $$
    so that the noncommutative Agler decomposition
    $$
      X_{i} X_{j}^{*} - Y_{i} Y_{j}^{*} =
      \sum_{k=1}^{d} \left( K_{k}(i,j) - T_{k}^{(i)} K_{k}(i,j)
      (T_{k}^{(j)})^{*} \right)
    $$
    holds.

    \medskip
    \noindent
    {\em \textbf{(2) Riesz-Dunford interpolation
    for the noncommutative polydisk.}}  Suppose that $\cZ$ is a
    separable Hilbert space with orthonormal basis
    $\{e_{1}, \dots, e_{\kappa}\}$ (with possibly $\kappa=\infty$) and that we are given $N$ (not
    necessarily commutative) $d$-tuples of strictly contractive
    operators $Z^{(1)}$ $=$
    $(Z^{(1)}_{1}, \dots, Z^{(1)}_{d})$, $\dots$, $Z^{(N)}$ $=$
    $(Z^{(N)}_{1}, \dots, Z^{(N)}_{d})$ in $\cL(\cZ)$.  Assume
    also that we are given operators $W_{1}$, $\dots$, $W_{N}$ in
    $\cL(\cZ)$.  Then there exists a formal power series $s(z) =
    \sum_{\gamma \in \cF_{d}} s_{\gamma} z^{\gamma}$ in the
    scalar noncommutative Schur-Agler class $\mathcal{SA}_{nc,d}$
    such that
    $$
    s(Z^{(i)}) : = \sum_{\gamma \in \cF_{d}}
    s_{\gamma}(Z^{(i)})^{\gamma^{\top}} = W_{i} \text{ for } i = 1, \dots, N
    $$
    if and only if there exist $d$ positive semidefinite operator
    matrices $K_{1}, \dots, K_{d}$ of size $(N \cdot \kappa) \times
    (N \cdot \kappa)$ with entries in $\cL(\cZ)$ written in the form
    $$
    K_{k} = \left[ K_{k}( (i,i'),\, (j,j') ) \right]_{(i,i'),
    (j,j') \in \{1, \dots, N\} \times \{1, \dots, \kappa\}}
    $$
    so that the following noncommutative Agler decomposition
    holds:
    $$
       e_{i'} e_{j'}^{*}  - W_{i} e_{i'} e_{j'}^{*}
    W_{j}^{*} = \sum_{k=1}^{d} \left( K_{k} \left((i,i'), (j,j')\right)
    - Z_{k}^{(i)} K_{k}\left( (i,i'), (j,j') \right)
    (Z_{k}^{(j)})^{*} \right).
    $$
   \end{theorem}

   \begin{proof}  Statement (1) is a particular case of \cite[Theorem
       7.9]{BB07}.  Statement (2) then follows from statement (1) in
       the same way as was done for the single-variable case in
       Section \ref{S:1var}.
    \end{proof}

    As an example we next  discuss how the Schur interpolation problem for the noncommutative polydisk
    can be handled as an application of Theorem \ref{T:NCpolydisk-int}.
     We first say that a subset $\Gamma$
    of $\cF_{d}$ is {\em lower inclusive} if, whenever $\gamma \in
    \Gamma$ and $\gamma$ factors as $\gamma = \gamma' \gamma''$, then
    also $\gamma' \in \Gamma$.  The Carath\'eodory-Fej\'er
    interpolation problem for the noncommutative polydisk can be
    formulated as follows:  {\em Given a lower inclusive subset
    $\Gamma$ of $\cF_{d}$ and given a collection of operators
    $\{F_{\gamma} \colon \gamma \in \Gamma\} \subset \cL(\cU, \cY)$,
    find a formal power series $S(z) = \sum_{\gamma \in \cF_{d}}
    S_{\gamma} z^{\gamma}$ in the noncommutative operator-valued
    Schur class $\mathcal{SA}_{nc,d}(\cU, \cY)$ such that}
    $$
      S_{\gamma} = F_{\gamma} \text{ for } \gamma \in \Gamma.
    $$
    In \cite[Section 7.4]{BB07} it is shown how to choose the data set
    $$
    {\mathfrak D}: \quad T^{(1)}, \dots, T^{(N)}, \quad
      X_{1}, \dots, X_{d}, \quad Y_{1}, \dots, Y_{d}
     $$
     so that the associated left-tangential interpolation problem
     with  operator-argu\-ment handled by statement (1) in Theorem
     \ref{T:NCpolydisk-int} is equivalent to the
     Carath\'eodory-Fej\'er interpolation problem for the
     noncommutative polydisk.  We note that this
     Carath\'eodory-Fej\'er problem was handled directly earlier in \cite{KV}.

     Just as was the case for the commutative case,  the
     theory for the noncommutative polydisk can be extended to more
     general noncommutative domains.  This is done in \cite{BGM2}
     for noncommutative operator domains with  a certain type of linear defining
     function $Q$: $Z = (Z_{1}, \dots, Z_{d}) \in \cD_{Q}$ if $\|
     Q(Z) \| < 1$ where $Q(z) = Q_{1} z_{1} + \cdots Q_{d} z_{d}$ and
     $ Q(Z) = Q_{1} \otimes Z_{1} + \cdots + Q_{d} \otimes Z_{d}$.
     The results on Left (and/or Right) Tangential interpolation with Operator-Argument
     in \cite{BB07} are actually given for this level of
     generality.

\section{$W^*$-correspondence Nevanlinna-Pick theorems} \label{S:C*-NP}

\subsection{Preliminaries}

Let $\cA$ and $\cB$ be $C^*$-algebras and $E$ a linear space. We say
that $E$ is an {\em $(\cA,\cB)$-correspondence} when $E$ is a bi-module
with respect to a given right $\cB$-action and a left $\cA$-action,
and $E$ is endowed with a $\cB$-valued inner product $\inn{\ }{\ }_E$
satisfying the following axioms: For any $\la,\mu\in\C$,
$\xi,\eta,\ze\in E$, $a\in\cA$ and $b\in\cB$
\begin{enumerate}
\item $\inn{\la\xi+\mu\ze}{\eta}_E =\la\inn{\xi}{\eta}_E+\mu\inn{\ze}{\eta}_E;$
\item $\inn{\xi\cdot b}{\eta}_E=\inn{\xi}{\eta}_E b;$
\item $\inn{a\cdot\xi}{\eta}_E=\inn{\xi}{a^*\cdot\eta}_E;$
\item $\inn{\xi}{\eta}_E^*=\inn{\eta}{\xi}_E;$
\item $\inn{\xi}{\xi}_E\geq 0 \ (\text{in }\cB);$
\item $\inn{\xi}{\xi}_E=0\mbox{ implies that }\xi=0$;
\end{enumerate}
and such that $E$ is a Banach space with respect to the norm $\|\ \|_E$
defined by
\[
\|\xi\|_E=\|\inn{\xi}{\xi}_E\|^\half_\cB\quad(\xi\in E),
\]
where $\|\ \|_\cB$ denotes the norm of $\cB$. We also impose that
\begin{equation*}
(\la\xi)\cdot b=\xi\cdot(\la b)\ands (\la a)\cdot\xi=a\cdot(\la\xi)\quad
(\la\in\C,a\in\cA,b\in\cB,\xi\in E).
\end{equation*}
In practice we usually write $\inn{\ }{\ }$ and $\|\ \|$
for the inner product and norm on $E$, and in case $\cA=\cB$ we say that
$E$ is an $\cA$-correspondence.

Given two $(\cA,\cB)$-correspondences $E$ and $F$,
the set of bounded linear operators from $E$ to $F$
is denoted by $\cL(E,F)$.
It may happen that a $T\in\cL(E,F)$ is not adjointable, i.e.,
it is not necessarily the case that
\[
\inn{T\xi}{\ga}=\inn{\xi}{T^*\ga}\quad(\xi\in E,\ga\in F)
\]
for some $T^*\in\cL(F,E)$. We will write $\cL^a(E,F)$ for the
set of adjointable operators in $\cL(E,F)$. As usual we have the
abbreviations $\cL(E)$ and $\cL^a(E)$ in case $F=E$.

The third inner-product axiom implies that the left $\cA$-action can
be identified with a $*$-homomorphism $\vph$ of $\cA$ into the
$C^*$-algebra $\cL^a(E)$. In case this $*$-homomorphism $\vph$ is
specified we will occasionally write $\vph(a) \xi$ instead of
$a\cdot \xi$.

Furthermore, an operator $T\in\cL(E,F)$ is said to be a {\em right $\cB$-module
map} if
\[
T(\xi\cdot b)=T(\xi)\cdot b\quad(\xi\in E, b\in\cB),
\]
and a {\em left $\cA$-module map} whenever
\[
T(a\cdot\xi)=a\cdot T(\xi)\quad(\xi\in E, a\in\cA).
\]
It is easily checked that an adjointable map $T \in \cL^{a}(E,F)$ is
automatically a right module map.
Occasionally we leave out the $\cB$ and $\cA$, and just say left or right
module map.
In case $T$ is both a left and right module map we also say that $T$
is a {\em bi-module map}. Notice that the product of two left (right) module
maps is again a left (right) module map, and the adjoint of an adjointable
left (right) module map, as also a left (right) module map.

We will have a need for various constructions which create new correspondences
out of given correspondences.

Given two $(\cA, \cB)$-correspondences $E$ and $F$, we define the
{\em direct-sum correspondence} $E \oplus F$ to be the direct-sum vector
space $E \oplus F$ together with the diagonal left $\cA$-action and
right $\cB$-action and the direct-sum $\cB$-valued inner-product
defined by setting for each $\xi,\xi'\in E$, $\ga,\ga'\in F$,
$a\in\cA$ and $b\in\cB$:
\[
\begin{array}{c}
a\cdot(\xi\oplus\ga)=(a\cdot\xi) \oplus (a\cdot\ga),\qquad
(\xi\oplus\ga)\cdot b=(\xi\cdot b)\oplus(\ga\cdot b),\\[.1cm]
\inn{\xi\oplus\ga}{\xi'\oplus\ga'}_{E\oplus F}=
\inn{\xi}{\xi'}_E+\inn{\ga}{\ga'}_F.
\end{array}
\]
Bounded linear operators between direct-sum correspondences admit
operator matrix decompositions in precisely the same way as in the
Hilbert space case ($\cB=\C$), while adjointability and the left and
right module map property of such an operator corresponds to the
operators in the decomposition being adjointable, or left or
right module maps, respectively.

Now suppose that we are given three $C^{*}$-algebras $\cA, \cB$ and
${\mathcal C}$ together with an $(\cA, \cB)$-correspondence $E$ and
a $(\cB,\cC)$-correspondence $F$. Then we define the
{\em tensor-product correspondence} $E \otimes F$ to be the completion of
the linear span of all tensors $\xi\otimes\ga$ (with $\xi\in E$ and
$\ga\in F$) subject to the identification
\begin{equation}\label{balance}
 (\xi\cdot b)\otimes\ga=\xi\otimes(b\cdot\ga)\quad(\xi\in E,\ga\in F,b\in\cB),
\end{equation}
with left $\cA$-action, right $\cC$-action and the $\cC$-valued
inner-product defined by setting for each $\xi,\xi'\in E$,
$\ga,\ga'\in F$, $a\in\cA$ and $c\in\cC$:
\[
\begin{array}{c}
a\cdot(\xi\otimes\ga)=(a\cdot\xi)\otimes\ga,\qquad
(\xi\otimes\ga)\cdot c=\xi\otimes(\ga\cdot c),\\[.1cm]
\inn{\xi\otimes\ga}{\xi'\otimes\ga'}_{E\otimes F}=
\inn{\inn{\xi}{\xi'}_E\cdot\ga}{\ga'}_F.
\end{array}
\]
In case the left action on $F$ is given by the $*$-homomorphism
$\vph$ we occasionally emphasize this by writing $E\otimes_\vph F$
for $E\otimes F$.

It is more complicated to characterize the bounded linear operators
between tensor-product correspondences. One way to construct such
operators is as follows.  Let $E$ and $E'$ be
$(\cA,\cB)$-correspondences and $F$ and $F'$
$(\cB,\cC)$-correspondences, for $C^*$-algebras $\cA$, $\cB$ and
$\cC$. Furthermore, let $X\in\cL(E,E')$ be a right module map and
$Y\in\cL(F,F')$ a left module map. Then we write $X\otimes Y$ for
the operator in $\cL(E\otimes F,E'\otimes F')$ which is determined
by
\begin{equation}\label{tensoroperator}
X\otimes Y (\xi\otimes\ga)=(X\xi)\otimes(Y\ga)\quad
(\xi\otimes\ga\in E \otimes F).
\end{equation}
The module map properties are needed to guarantee that the balancing
in the tensor-product (see \eqref{balance}) is respected by the
operator $X\otimes Y$.

If, in addition, $X$ is also a left module map, then $X\otimes Y$ is
a left module map, while $Y$ also being a right module map
guarantees that $X\otimes Y$ is a right module map. Moreover,
$X\otimes Y$ is adjointable in case $X$ and $Y$ are both adjointable
operators, with $(X\otimes Y)^*=X^*\otimes Y^*$.

Notice that the left action on $E\otimes F$ can now be written as
$a\mapsto\varphi(a)\otimes I_F\in\cL^a(E\otimes F)$,
where $I_F\in\cL^{a}(F)$ is the identity operator on $F$.

\subsection{The Fock space $\cF^{2}(E)$ and the Toeplitz algebra
$\cF^{\infty}(E)$}  \label{subS:Fock}

In this section we shall consider the situation where $\cA
= \cB$, i.e, $E$ is an $\cA$-correspondence.
We also restrict our attention to the case where $\cA$ is a
von Neumann algebra and let $E$ be an $\cA$-$W^{*}$-correspondence.
This means that $E$ is an $\cA$-correspondence which is also {\em
self-dual} in the sense that any right $\cA$-module map $\rho \colon
E \to \cA$ is given by taking the inner-product against some element
$e_{\rho}$ of $E$:
\begin{equation}  \label{rho}
\rho(e) = \langle e, e_{\rho} \rangle_{E} \in \cA.
\end{equation}
It is easily seen that such maps are adjointable with adjoint
$\rho^{*} \colon \cA \to E$  given by
\begin{equation}  \label{rho*}
\rho^{*}(a) = e_{\rho} \cdot a
\end{equation}
and hence also {\em any} right module map $\nu \colon \cA \to E$ has
the form $\nu = \rho^{*}$ as in \eqref{rho*}.
Moreover, the space $\cL^{a}(E)$ of adjointable operators on
the $W^{*}$-correspondence $E$ is in fact a
$W^{*}$-algebra, i.e., it is the
abstract version of a von Neumann algebra with an
ultra-weak topology
(see \cite{MS04}).

Since $E$ is an $\cA$-correspondence, we may define the self-tensor-product
$E^{\otimes 2} = E \otimes E$
to get another $\cA$-correspondence,
and, inductively, an $\cA$-correspondence
$E^{\otimes n} = E \otimes (E^{\otimes (n-1)})$ for each $n =
1,2, \dots$.  If we use $a \mapsto \varphi(a)$ to denote the left
$\cA$-action $\varphi(a) e = a \cdot e$ on $E$, we denote the left
$\cA$-action on $E^{\otimes n}$ by $\varphi^{(n)}$:
$$
  \varphi^{(n)}(a) \colon \xi_{n} \otimes \xi_{n-1} \otimes \cdots
  \otimes \xi_{1}\mapsto (\varphi(a) \xi_{n}) \otimes
  \xi_{n-1} \otimes \cdots \otimes \xi_{1}.
$$
Note that, using the notation in
\eqref{tensoroperator}, we may write $\varphi^{(n)}(a)=\varphi(a)\otimes
I_{E^{\otimes n-1}}$. We formally set $E^{\otimes 0} = \cA$. Then
the {\em Fock space} $\cF^{2}(E)$ is defined to be
\begin{equation}  \label{Fock}
   \cF^{2}(E) = \oplus_{n=0}^{\infty} E^{\otimes n}
\end{equation}
and is also an $\cA$-$W^{*}$-correspondence.
The left $\cA$-action on $\cF^2(E)$ is given by $\varphi_{\infty}$:
\begin{equation}  \label{ArepFock}
  \varphi_{\infty}(a) \colon \oplus_{n=0}^{\infty} \xi^{(n)} \mapsto
  \oplus_{n=0}^{\infty}( \varphi^{(n)}(a) \xi^{(n)}) \text{ for }
  \oplus_{n=0}^{\infty} \xi^{(n)} \in
  \oplus_{n=0}^{\infty} E^{\otimes n},
\end{equation}
or, more succinctly,
$$
\varphi_\infty(a)=\textup{diag}(a,\varphi^{(1)}(a),\varphi^{(2)}(a),\ldots).
$$

 In addition to the operators $\varphi_{\infty}(a) \in
       \cL^{a}(\cF^2(E))$, we introduce the so-called {\em
       creation operators} on $\cF^{2}(E)$ given, for each $\xi \in E$,
       by the subdiagonal (or shift) block matrix
       $$
      T_{\xi} = \begin{bmatrix} 0 & 0 & 0 & \cdots \\
                      T_{\xi}^{(0)} & 0 & 0  & \cdots \\
      0 & T_{\xi}^{(1)} & 0 & \cdots & \\
      \vdots & \ddots & \ddots  & \ddots  \end{bmatrix}
       $$
       where the block entry $T_{\xi}^{(n)} \colon E^{\otimes n} \to
       E^{\otimes n+1}$ is given by
       \begin{equation*}
    T_{\xi}^{(n)} \colon \xi_{n}\otimes \cdots \otimes \xi_{
    1} \mapsto
     \xi \otimes \xi_{n} \otimes \cdots \otimes \xi_{1}.
       \end{equation*}
       The operator $T_\xi$ is also in $\cL^{a}(\cF^2(E))$.  In summary,
       both $T_\xi$ and $\varphi_{\infty}(a)$ are right $\cA$-module maps
       with respect to the right $\cA$-action on $\cF^2(E)$ for each
       $\xi\in E$ and $a\in\cA$.  Moreover, one easily
       checks that
       \[
\varphi_{\infty}(a)T_\xi=T_{a\xi}=T_{\varphi(a)\xi}\quad\mbox{and}\quad
       T_\xi\varphi_{\infty}(a)=T_{\xi a}\quad\mbox{for each
}a\in\cA\mbox{ and }
       \xi\in E.
       \]

       We let $\cF^{\infty}(E)$ denote the {\em Toeplitz algebra}
       equal to the weak-$*$ closed algebra generated by the linear span of the
       collection of operators
       $$
     \{ \varphi_{\infty}(a), \, T_{\xi} \colon a \in \cA
      \text{ and } \xi
     \in E \}
       $$
       in the $W^{*}$-algebra $\cL^{a}(\cF(E))$.
       The justification for the term ``Toeplitz algebra'' comes from
       the following proposition, which is a variation on Proposition
       4.2 in \cite{BBFtH}.

       \begin{proposition} \label{P:Toeplitz-recipe}
       If $R \in \cL^{a}(\cF^{2}(E))$ is in the
       Toeplitz algebra $\cF^{\infty}(E)$ with matrix
       representation
       \begin{equation}\label{Rdec}
       R = [R_{i,j}]_{i,j=0,1,2,\dots}
       \end{equation}
       where $R_{i,j}\in\cL^a(E^{\otimes j},E^{\otimes i})$,
       then there exist a sequence $\xi^{(0)},\xi^{(1)},\xi^{(2)},\ldots$ with
       $\xi^{(n)}\in E^{\otimes n}$ such that
       \begin{equation}   \label{Toeplitz-recipe}
       R_{i,j} = \begin{cases} 0 &\text{if } i<j, \\
       T^{(0)}_{\xi^{(i-j)}} \otimes I_{E^{\otimes j}} &\text{if } i \ge j,
       \end{cases}
       \end{equation}
       where for $n=0,1,2,\ldots$ the operator $T^{(0)}_{\xi^{(n)}}\in\cL^a(\cA,E^{\otimes n})$
       is given by
       \begin{equation}\label{crea-n}
       T^{(0)}_{\xi^{(n)}} a=\xi^{(n)}a\quad (a\in\cA).
       \end{equation}
       In particular, $R$ is completely determined by the entries of its first column.
       Conversely, if $\xi^{(0)},\xi^{(1)},\xi^{(2)},\ldots$ is a sequence with
       $\xi^{(n)}\in E^{\otimes n}$ such that the infinite operator matrix $R$ given by
       (\ref{Rdec}), (\ref{Toeplitz-recipe}) and (\ref{crea-n}) induces a bounded operator
       on $\cF^2(E)$, then $R$ is in $\cF^\infty(E)$.
\end{proposition}

    \begin{proof} Let $\cF^{\infty \prime}(E)$ be the set of all
        adjointable operators on $\cF^{2}(E)$ with matrix
        representation $R = [R_{i,j}]_{i,j=0,1, \dots}$ given by
        (\ref{Toeplitz-recipe}) and (\ref{crea-n}).
        Note first that $\cF^{\infty \prime}(E)$ is an algebra;
        indeed if $S$ and $R$ in $\cF^{\infty \prime}(E)$ are given
        by the sequences $\xi^{(0)},\xi^{(1)},\xi^{(2)},\ldots$ and
        $\ze^{(0)},\ze^{(1)},\ze^{(2)},\ldots$ with
        $\xi^{(n)},\ze^{(n)}\in E^{\otimes n}$, respectively, then it
        is straightforward to check that $SR$ is the element of
        $\cF^{\infty \prime}(E)$ given by the sequence
        $\rho^{(0)},\rho^{(1)},\rho^{(2)},\ldots$ with $\rho^{(n)}\in E^{\otimes n}$
        equal to
        \[
        \rho^{(n)}=\sum_{k=0}^n\xi^{(k)}\otimes\ze^{(n-k)}.
        \]

To show that $\cF^{\infty}(E) \subset \cF^{\infty\prime}(E)$, it therefore
suffices to check that
(1) each of the generators $\varphi_{\infty}(a)$ and $T_{\xi}$ (for $a \in \cA$ and
$\xi \in E$) is in $\cF^{\infty \prime}(E)$,
(2) that $\cF^{\infty \prime}(E)$ is closed under addition,
and that (3) $\cF^{\infty \prime}(E)$ is weak-$*$  closed.
These verifications are straightforward and are left to the reader.

Conversely, to show that $\cF^{\infty \prime}(E) \subset\cF^\infty(E)$, since
$\cF^{\infty}(E)$ is weak-$*$ closed by definition, it suffices to consider
that case of an operator $R \in \cF^{\infty \prime}(E)$ with support only on
a subdiagonal:  $R_{i,j} = 0$ unless $i-j = k$ for some $k\ge 0$. Moreover,
we can restrict to the case that $R_{k,0}\in\cL^a(\cA,E^{\otimes k})$ is
defined by some pure tensor $\xi_k\otimes\cdots\otimes\xi_1$ in $E^{\otimes k}$ via
$R_{k,0}a=\xi_k\otimes\cdots\otimes\xi_2\otimes (\xi_1 a)$ for $a\in\cA$. To see that
this is the case, first note that, since $E^{\otimes k}$ is an
$\cA$-$W^*$-correspondence, $R_{k,0}$ is of the form $R_{k,0}a=\xi^{(k)} a$ for
some $\xi^{(k)}\in E^{\otimes k}$. The claim then follows since any element
$\xi^{(k)}\in E^{\otimes k}$ can be approximated by linear combinations of pure
tensors in $E^{\otimes k}$, and because $\cF^{\infty}(E)$ is weak-$*$ closed.
Finally, assume that $R$ is only supported on the $k^\tu{th}$ diagonal and that
$R_{k,0}$ is defined by the pure tensor $\xi_k\otimes\cdots\otimes\xi_1\in E^{\otimes k}$
as described above. It is then straightforward to check that $R=T_{\xi_k}\cdots T_{\xi_1}$.
Thus $R$ is in $\cF^\infty(E)$.
\end{proof}

\subsection{Correspondence-representation pairs and their dual}\label{subS:dual}

In addition to the von Neumann algebra $\cA$ and the
$\cA$-correspondence $E$,
   suppose that we are also given an auxiliary Hilbert space
$\cE$ and a representation (meaning a
   nondegenerate $*$-homomorphism) $\sigma \colon \cA \to \cL(\cE)$; as
   this will be the setting for much of the analysis to
follow, we refer
   to such a pair $(E, \sigma)$ as a {\em correspondence-representation
   pair}. We further assume that $\si$ is faithful (injective) and
   normal ($\si$-weakly continuous).
   Then the Hilbert space $\cE$ equipped with $\sigma$ becomes an
   $(\cA,\C)$-correspondence with left $\cA$-action given by $\sigma$:
   $$
      a \cdot y = \sigma(a) y \text{ for all } a \in \cA \text{ and } y
      \in \cE.
   $$
Thus we can form the tensor-product $(\cA, {\mathbb C})$-correspondence
$E \otimes_{\sigma} \cE$. With $E^{\sigma}$ we denote the set of all
bounded linear operators $\mu\colon \cE\to E \otimes_{\sigma} \cE$ which
are also left $\cA$-module maps:
       \begin{equation} \label{Esigma0}
       E^{\sigma} = \{ \mu\colon \cE\to E \otimes_{\sigma} \cE
       \colon
           \mu \sigma(a)  = ( \varphi(a) \otimes I_{\cE}) \mu \}.
       \end{equation}
       It turns out that $E^{\sigma}$ is itself a $W^*$-correspondence (the
       correspondence {\em dual} to $E^{\sigma}$; see \cite[Section 3]{MS04}),
       not over $\cA$ but over the $W^*$-algebra
       $$
       \sigma(\cA)' = \{b \in \cL(\cE) \colon b \sigma(a) =
       \sigma(a) b \text{ for all } a \in \cA\} \subset \cL(\cE)
       $$
       (the {\em commutant} of the image $\sigma(\cA)$ of the
       representation $\sigma$ in $\cL(\cE)$) with left and right $\sigma(\cA)'$-action and
       $\sigma(\cA)'$-valued inner-product $\langle \cdot, \cdot
       \rangle_{E^{\sigma}}$ given by
\[
\begin{array}{c}
\mu\cdot b=\mu b,\quad b\cdot\mu=(I_\cE\otimes b)\mu,\\[.1cm]
\langle \mu, \nu \rangle_{E^{\sigma}} = \nu^{*}\mu.
\end{array}
\]
       The intertwining relations of elements $\mu$ and $\nu$ in $E^\si$ with
       $\si(a)$ and $a\otimes I_\cE$ for $a\in\cA$ imply that $\nu^*\mu$ is indeed
       in $\si(\cA)'$.

Notice that elements $b\in\si(\cA)'$ and $\mu\in E^\si$ define operators on the
$(\cA,\BC)$-correspondence
\[
\cF^2(E,\si):=\cF^2(E)\otimes_\si \cE=\oplus_{n=0}^\infty E^{\otimes n}\otimes_\si \cE
\]
via the operator matrices
\[
I_{\cF^2(E)}\otimes b=\tu{diag}(b,I_E \otimes b,I_{E^{\otimes 2}}\otimes b,\ldots)
\]
and
\[
I_{\cF^2(E)}\otimes \mu=
\mat{ccccc}{0&0&0&0&\cdots\\
\mu&0&0&0&\cdots\\
0&I_E\otimes\mu&0&0&\cdots\\
0&0&I_{E^{\otimes 2}}\otimes\mu&0&\cdots\\
\vdots&\vdots&\ddots&\ddots&\ddots}.
\]
Here, in interpreting $I_{E^{\otimes n}}\otimes\mu$, we use the identification
\begin{equation}\label{ident}
E^{\otimes n} \otimes_{\sigma} \cE = E^{\otimes n-1}
      \otimes (E \otimes_{\sigma} \cE).
\end{equation}
The operators $I_{E^{\otimes n}}\otimes b$ and $I_{E^{\otimes n}}\otimes\mu$ are well defined because
$b$ and $\mu$ are left $\cA$-module maps.

The following theorem provides another way of characterizing the elements of $\cF^\infty(E)$ using
the operators $I_{\cF^2(E)}\otimes b$ and $I_{\cF^2(E)}\otimes\mu$. The result follows directly from
the combination of Theorems 3.9 and 3.10 in \cite{MS04}.

\begin{proposition}\label{P:dual}
Given an operator $X\in\cL(\cF^2(E,\si))$, there exists an $R\in\cF^\infty(E)$ so that
$X=R\otimes I_\cE$ if and only if $X$ commutes with $I_{\cF^2(E)}\otimes b$ and
$I_{\cF^2(E)}\otimes\mu$ for each $b\in\si(\cA)'$ and $\mu\in E^\si$. Moreover, if
$X=R\otimes I_\cE$ for some $R\in\cF^\infty(E)$, then $\|X\|=\|R\|$.
\end{proposition}

We can then prove the following concrete version of the $C^*$-correspondence commutant lifting
theorem \cite[Theorem 4.4]{MS98}.

\begin{theorem}\label{T:CLT}
Given subspaces $\cM$ and $\cN$ of $\cF^2(E,\si)$ that are both invariant under
$I_{\cF^2(E)}\otimes b$ and $I_{\cF^2(E)}\otimes\mu^*$ for all $b\in\si(\cA)'$ and
$\mu\in E^\si$, and a contractive operator $X$ from $\cM$ into $\cN$ such that
\[
X^*(I_{\cF^2(E)}\otimes b)|_{\cN}=(I_{\cF^2(E)}\otimes b)X^*\ands
X^*(I_{\cF^2(E)}\otimes\mu^*)|_{\cN}=(I_{\cF^2(E)}\otimes\mu^*)X^*,
\]
for all $b\in\si(\cA)'$ and $\mu\in E^\si$, there exists an $S\in\cF^\infty(E)$
with $\|S\|\leq 1$ so that $(S^*\otimes I_\cE)\cN \subset \cM$ and $X= P_\cN(S\otimes I_\cE)|_\cM$.
\end{theorem}

\begin{proof}
Using the $C^*$-correspondence commutant lifting theorem \cite[Theorem 4.4]{MS98} and
the intertwining relations of $X$ it follows that $X$ can be lifted to a
contractive operator $Y$ on $\cF^2(E,\si)$ that commutes with $I_{\cF^2(E)}\otimes b$
and $I_{\cF^2(E)}\otimes\mu$ for each $b\in\si(\cA)'$ and $\mu\in E^\si$, with the
property that $Y^*\cN\subset\cM$ and  $X=P_\cN Y|_\cM$. The claim then follows
immediately with Proposition \ref{P:dual}.
\end{proof}

  \subsection{The generalized disk ${\mathbb D}((E^{\sigma})^{*})$ and
  the first Muhly-Solel point-evaluation}   \label{subS:points}

  For the definition of
       point-evaluations to follow, however, the important object is
       $(E^{\sigma})^*$,  the set of adjoints of elements of $E^\si$
       (which are also left $\cA$-module maps):
       \begin{equation} \label{Esigma*}
       (E^{\sigma})^*=\{\eta:E \otimes_{\sigma} \cE\to\cE
       \colon \eta^*\in E^{\sigma}\}.
       \end{equation}

       For a given $\eta \in (E^{\sigma})^{*}$ and a positive integer $n$,
       we may define the generalized power
       $\eta^{n} \colon E^{\otimes n} \otimes_{\sigma} \cE \to \cE$ by
       \begin{equation*}  
    \eta^{n} = \eta (I_{E} \otimes \eta) \cdots
(I_{E^{\otimes n-1}}
       \otimes \eta)
       \end{equation*}
       where we again use the identification (\ref{ident})
       in this definition. We also set
       $\eta^0=I_\cE\in\cL(\cE)$. Again the fact that $\eta$
       is a left $\cA$-module map ensures that $I_{E^{\otimes k}}\otimes\eta$
       is a well defined operator in $\cL(E^{\otimes k+1}\otimes_\sigma
       \cE, E^{\otimes k}\otimes_\sigma \cE)$. The defining $\cA$-module
       property of $\eta$ in \eqref{Esigma*} then extends to the
       generalized powers $\eta^{n}$ in the form
       \begin{equation}  \label{mod-prop-ext}
    \eta^{n} (\varphi^{(n)}(a) \otimes I_{\cE})
    = \sigma(a) \eta^{n},
       \end{equation}
       i.e., $\eta^{n}$ is also an $\cA$-module map.

       Denote by ${\mathbb D}((E^{\sigma})^{*})$ the set of strictly
       contractive elements of $(E^{\sigma})^{*}$:
       $$
     {\mathbb D}((E^{\sigma})^{*}) = \{ \eta \in
(E^{\sigma})^{*} \colon
     \| \eta \| < 1 \}.
       $$

     We then consider elements $R \in \cF^{\infty}(E)$ as functions
       $\widehat R$
       on ${\mathbb D}((E^{\sigma})^{*})$
       with values in $\cL(\cE)$
       according to the formula for the {\em first Muhly-Solel point
       evaluation}
        \begin{equation}  \label{MSpoint-eval}
       \widehat R(\eta) = \sum_{n=0}^{\infty} \eta^{n} (R_{n,0}
       \otimes I_{\cE})
       \text{ where } R = [R_{i,j}] \in \cF^{\infty}(E).
\end{equation}
{}From the facts that $\| \eta \| < 1$ and $\|R_{n,0}\| \le M < \infty$
(since $R$ is bounded on $\cF^{\infty}(E)$), one can see that the
series in \eqref{MSpoint-eval} converges in operator norm.

\begin{remark} {\em The definition of the first Muhly-Solel
    point-evaluation in \cite{MS04, MSSchur, MSPoisson} is actually more
    elaborate.  There it is observed that an element $\eta \in
    \overline{\mathbb D}((E^{\sigma})^{*})$ induces an ultraweakly
    continuous completely contractive bimodule map $\widehat \eta
    \colon E \to \cL(\cE)$ via the formula
    \begin{equation}  \label{etahat}
    \widehat \eta(\xi)e = \eta(\xi \otimes e) \text{ for } \xi \in E,
    \, e \in \cE.
    \end{equation}
    The bimodule property means that
    $$
    \widehat \eta (\varphi(a) \xi b) = \sigma(a) \widehat \eta(\xi)
    \sigma(b) \text{ for all } a,b \in \cA, \, \xi \in E.
    $$
    Conversely, if $S \colon E \to \cL(\cE)$ is an ultraweakly
    continuous completely contractive bimodule map, then the same
    formula \eqref{etahat} turned around
    \begin{equation}  \label{tildeS}
    \widetilde S(\xi \otimes e) = S(\xi) e
   \end{equation}
   can be used to define an element $\widetilde S$ of the closed generalized disk
   $\overline{\mathbb D}((E^{\sigma})^{*})$.  The pair $(\widehat
   \eta, \sigma)$ (or $(S, \sigma)$) is said to be a {\em
   covariant representation} of the correspondence $E$.   Given a
   covariant representation $(S, \sigma)$, in case $\|S\| < 1$ (or,
   equivalently, $\| \widetilde S\| < 1$), there is an associated
   completely contractive representation $S \times \sigma$ of the
   Toeplitz algebra $\cF^{\infty}(E)$ defined on generators via
   \begin{align}
      & (S \times \sigma) (\varphi_{\infty}(a)) = \sigma(a),  \notag \\
      & (S \times \sigma)(T_{\xi}) = S(\xi).
      \label{Sxsigma}
   \end{align}
   Conversely, if $\rho$ is any ultraweakly continuous completely
   contractive  $\cL(\cE)$-valued representation of
   $\cF^{\infty}(E)$, the same formulas \eqref{Sxsigma} can be turned
   around
   \begin{align}
       & \sigma(a) = \rho(\varphi_{\infty}(a)),\notag  \\
       & S(\xi) = \rho(T_{\xi})
       \label{rho-covrep}
   \end{align}
   to define a covariant representation $(S, \sigma)$ of $E$.  Given
   any covariant representation $(S, \sigma)$ of $E$, the formulas
   \eqref{Sxsigma} can always
   be extended to define a representation of the {\em uniform closure}
   $\cT^{\infty}(E)$
   of the algebra generated by $\varphi_{\infty}(a)$ and $T_{\xi}$
   ($a \in \cA$ and $\xi \in E$).  It is an unsolved problem in the
   theory to identify which covariant representations
   $(S,\sigma)$  have the property that \eqref{Sxsigma} can be
   extended to define an ultraweakly continuous completely
   contractive representation of the weak-$*$ closed Toeplitz algebra
   $\cF^{\infty}(E)$.

   If we fix a representation $\sigma \colon \cA \to \cL(\cE)$, and
   let $\eta \in {\mathbb D}((E^{\sigma})^{*})$ and $R \in \cF^{\infty}(E)$,
   then the value $\widehat R(\eta)$ of $R$ at $\eta$ is defined to be the
   element of $\cL(\cE)$ assigned to $R$ by the representation
   $\widehat \eta \times \sigma$:
   $$
   \widehat R(\eta) = (\widehat \eta \times \sigma)(R).
   $$
   In the end this definition agrees with the definition
   \eqref{MSpoint-eval}.   The formula \eqref{MSpoint-eval} for the
   point-evaluation is called the {\em Cauchy transform} in \cite{MS04}.
   } \end{remark}

  Following \cite{MSSchur} (see also \cite{BBFtH}),
   given a correspondence-representation
   pair $(E, \sigma)$, we define the {\em Schur class} $\cS_{E,
   \sigma}$ to be the class of all functions $S \colon {\mathbb
   D}((E^{\sigma})^{*}) \to \cL(\cE)$ which can be expressed in the
   form $S(\eta) = \widehat R(\eta)$ for some $R \in \cF^{\infty}(E)$
   with $\|R\| \le 1$.  The associated left-tangential Nevanlinna-Pick
   interpolation problem is:   {\em Given a subset $\Om$ of ${\mathbb
   D}((E^{\sigma})^{*})$ and two functions $F \colon \Om \to \cL(\cE)$
   and $G \colon \Om \to \cL(\cE)$,
   determine when there exists a function $S$ in the Schur class
   $\cS_{E,\sigma}$ so that $G S|_{\Om} = F$.}
   We shall explain the solution to this problem due to Muhly-Solel
   \cite{MS04} in Subsection \ref{subS:MS-NP}.  For the moment we note the
   following two extreme cases:
   \begin{enumerate}
       \item[\textbf{Case 1:}] $\Om = {\mathbb D}((E^{\sigma})^{*})$ \textbf{and}
       $G(\eta)=I_\cE$.
       In this case the function $F$ is defined on all of ${\mathbb
       D}((E^{\sigma})^{*})$ and we seek a test to decide if $F \in
       \cS_{E, \sigma}$.

       \item[\textbf{Case 2:}] $\Om$ \textbf{finite.}  In this case
       $\Om$ consists of finitely many points, say $\eta_{1}, \dots,
       \eta_{N} \in {\mathbb D}((E^{\sigma})^{*})$, and we are given
       operators $Y_1=F(\eta_1),\ldots,Y_N=F(\eta_N)$ and
       $Y_{1} = F(\eta_{1}),\dots,Y_{N}=F(\eta_{N})$
       in $\cL(\cE)$.  We seek a test to decide if there
       is a function $S$ (or, more ambitiously, a description of all
       such functions $S$) in the Schur class $\cS_{E, \sigma}$
       satisfying the interpolation conditions
       $$
        X_j S(\eta_{j}) = Y_{j} \text{ for } j = 1, \dots, N.
       $$
   \end{enumerate}

  \subsection{Positive and completely positive kernels/maps}
  \label{subS:cp}

  The solution of the Nevanlinna-Pick interpolation problem in
  \cite{MS04} involves the notion of a completely positive map while
  the characterization of the Schur class $\cS_{E, \sigma}$ in
  \cite{MSSchur} involves the notion of {\em completely positive
  kernel} introduced in \cite{BBLS}.

  In the following discussion $\cA$, $\cB$ and $\cC$ are
  $C^{*}$-algebras (in particular, possibly $W^{*}$-algebras) and $\Om$
  is a set.  We begin with the notion of {\em positive kernel} which
  goes back at least to Aronszajn \cite{Aron};
  a function $K \colon \Om \times \Om \to {\mathcal C}$ is said to be
  a {\em positive kernel} if
  $$
  \sum_{i,j=1}^{n} c_{i}^{*} K(\om_{i},\om_{j}) c_{j}  \geq 0
  $$
  for all choices of $\om_{1}, \dots, \om_{n} \in \Om$ and $c_{1}, \dots,
  c_{n} \in \cC$ for $n = 1, 2, \dots$.  Equivalently, for every
  choice of $n$ points $\om_{1}, \dots, \om_{n} \in \Om$, the matrix
  $$
   [K(\om_{i},\om_{j})]_{i,j=1, \dots, n}
  $$
  is a positive element of $\cC^{n \times n}$ where $n=1,2, \dots$.
  Given two $C^{*}$-algebras $\cA$ and $\cB$, a map $\varphi \colon
  \cA \to \cB$ is said to  be a {\em positive map} if
  $$
  a \geq 0 \text{ in } \cA \Longrightarrow \varphi(a) \geq 0
  \text{ in } \cB.
  $$
  Such a map $\varphi \colon \cA \to \cB$ is said to be {\em
  completely positive} if
\begin{equation}\label{n-positive}
\begin{bmatrix} a_{11} & \dots &  a_{1n} \\ \vdots & & \vdots \\
      a_{n1} & \dots & a_{nn} \end{bmatrix} \succeq 0 \text{ in }
      \cA^{n \times n} \Longrightarrow \begin{bmatrix}
      \varphi(a_{11}) & \dots & \varphi(a_{1n}) \\ \vdots & & \vdots
      \\ \varphi(a_{n1}) & \dots & \varphi(a_{nn}) \end{bmatrix}
      \succeq 0 \text{ in } \cB^{n \times n}
\end{equation}
  for every $n = 1, 2, \dots$. If the implication (\ref{n-positive}) holds
  for a fixed $n$ in $\BZ_+$, then we say that $\vph$ is {\em $n$-positive}.
  In case $\cA = {\mathbb C}^{N \times
  N}$, a result of Choi \cite{Choi75} (see also \cite[Theorem
  3.14]{Paulsen-book}) says that {\em $\varphi \colon \cA \to \cB$ is
  completely positive if and only if the single block matrix
  $$
    \begin{bmatrix} \varphi(e_{11}) & \dots & \varphi(e_{1N}) \\
    \vdots & & \vdots \\
    \varphi(e_{N1}) & \dots & \varphi(e_{NN}) \end{bmatrix}
   $$
   is positive in $\cB^{N \times N}$}, where $e_{ij}$ are the standard
   matrix units
   $$
   [e_{ij}]_{\alpha, \beta} = \delta_{i,\alpha} \delta_{j, \beta}
   $$
   in ${\mathbb C}^{N \times N}$
   (here we make use of the Kronecker delta---$\delta_{i,j} = 1 $
   for $i=j$ and $\delta_{i,j} = 0$ for $i \ne j$). There is an
   alternative characterization in terms of positive kernels
   (see \cite{Stinespring}):
   {\em A map $\varphi \colon \cA \to \cB$ is completely positive if
   and only if the kernel $k_{\varphi} \colon \cA \times \cA \to \cB$
   defined by
   $$
    k_{\varphi}(a, a') = \varphi(a a^{\prime *})
    $$
   is a positive kernel.}  Barreto-Bhat-Liebscher-Skeide in
   \cite{BBLS} (with motivation from quantum physics which need not
   concern us here) combined these notions as follows:  we say that a
   kernel ${\mathbb K} \colon \Om \times \Om\to \cL(\cA, \cB)$ is {\em
   completely positive} if the associated kernel $k \colon (\Om\times
   \cA) \times (\Om \times \cA) \to \cB$ given by
   $$
   k((\om,a), (\om',a')) = {\mathbb K}(\om,\om')[a a^{\prime *}]
   $$
   is a positive kernel, i.e., if, for all choices $(\om_{1},a_{1})$,
   $\dots$, $(\om_{n},a_{n})$ in $\Om \times \cA$ and for all choices
   $b_{1}$, $\dots$, $b_{n}$ in $\cB$ it is the case that
   \begin{equation}  \label{cp}
   \sum_{i,j=1}^{n} b_{i}^{*} {\mathbb K}(\om_{i},\om_{j})[a_{i}
   a_{j}^{*}] b_{j} \geq 0 \text{ in } \cB
   \end{equation}
   for $n=1, 2, \dots$.  Note that the notion of completely positive
   kernel contains the notions of positive kernel and of completely
   positive map as special cases: in case the set $\Om$ consists of a
   single point, then ${\mathbb K}$ can be considered simply as a linear map
   from $\cA$ to $\cB$, and the condition that ${\mathbb K}$ be a
   completely positive kernel collapses to the condition that
   ${\mathbb K}$ be a completely positive map from $\cA$ to $\cB$ by
   the positive-kernel formulation of completely positive map
   mentioned above.  On the other extreme, if $\cA$ is just the
   complex numbers ${\mathbb C}$, then the condition that ${\mathbb
   K}$ be completely positive collapses to the condition that the
   kernel $k(\cdot, \cdot):= {\mathbb K}(\cdot, \cdot)[1]$ be a
   positive kernel.

   A number of equivalent characterizations of complete positivity
   for a kernel ${\mathbb K} \colon \Om \times \Om \to \cL(\cA, \cB)$ is
   given in \cite{BBLS}.  Let us mention some of these
   which will be convenient for our analysis of various generalized
   Nevanlinna-Pick interpolation problems.

   \begin{proposition}  \label{P:cp}
       Suppose that ${\mathbb
       K}$ is a function from $\Om \times \Om$ into $\cL(\cA, \cB)$.
       Then the following are equivalent:
       \begin{enumerate}
       \item ${\mathbb K}$ is a completely positive kernel, i.e.,
       the kernel $k \colon (\Om \times \cA) \times (\Om \times\cA) \to \cB$
       given by
       $$
       k((\om,a), (\om',a')) = {\mathbb K}(\om,\om')[a a^{\prime *}]
       $$
       is a positive kernel.

       \item For every choice of $n$ points $\om_{1}, \dots, \om_{n}$
       in $\Om$, the map $\varphi_{\om_{1}, \dots, \om_{n}} \colon
       \cA^{n \times n}
       \to \cB^{n \times n}$ given by
\begin{equation}   \label{varphix's}
\varphi_{\om_{1}, \dots, \om_{n}}([a_{ij}]_{i,j=1, \dots n}) =
\left[ {\mathbb K}(\om_{i},\om_{j})[a_{ij}]
\right]_{i,j=1, \dots, n}
\end{equation}
is a completely positive map for any $n=1,2, \dots$.

\item For every choice of $n$ points $\om_{1}, \dots, \om_{n}$ in $\Om$,
the map $\varphi_{\om_{1}, \dots, \om_{n}} \colon \cA^{n \times n} \to
\cB^{n \times n}$ given by \eqref{varphix's} is a positive map for
any $n=1,2, \dots$.
\end{enumerate}
\end{proposition}

\begin{proof}
The equivalence of (1), (2) and (3) correspond to the equivalence of
(1), (4) and (5) in Lemma 3.2.1 from \cite{BBLS}.
\end{proof}


 An important example of a completely positive kernel for our
 purposes here appearing implicitly in the work of Muhly-Solel
 \cite{MS04, MSSchur} and developed explicitly in \cite{BBFtH} is the
 Szeg\"o kernel associated with a correspondence-representation
 pair $(E,\sigma)$ and defined as follows. We suppose that we are
 given an $\cA$-$W^{*}$-correspondence $E$ together with a
 representation $\sigma \colon \cA \to \cL(\cE)$ for a Hilbert space
 $\cE$.  We then define the associated {\em Szeg\"o kernel}
 $$
 {\mathbb K}_{E, \sigma} \colon {\mathbb D}((E^{\sigma})^{*}) \times
 {\mathbb D}((E^{\sigma})^{*}) \to \cL(\sigma(\cA)', \cL(\cE))
 $$
 by
 \begin{equation}  \label{Szego-ker}
   {\mathbb K}_{E, \sigma}(\eta, \zeta)[b] = \sum_{n=0}^{\infty}
 \eta^{n} (I_{E^{\otimes n}} \otimes b) \zeta^{n *}.
 \end{equation}
 Then it can be seen that ${\mathbb K}_{E, \sigma}$ is a completely
 positive kernel and in fact is the reproducing kernel for a
 reproducing kernel correspondence $H^{2}(E, \sigma)$ whose elements
 are $\cE$-valued functions on the generalized disk ${\mathbb
 D}((E^{\sigma})^{*})$ which also carries a representation
 $\iota_{\infty}$ of $\sigma(\cA)'$.  This construction is based on yet another and
 perhaps the most fundamental characterization of completely positive
 kernel which is not mentioned in Proposition \ref{P:cp}, namely:
 {\em ${\mathbb K}$ is completely positive if and only if ${\mathbb K}$ has a {\em
 Kolmogorov decomposition} in the sense of \cite{BBLS}}; the
 reproducing-kernel correspondence associated with the completely
 positive kernel can then be taken to be the middle space in its
 Kolmogorov decomposition
 and the Schur class $\cS_{E,\sigma}$ can
 alternatively be characterized as the space of contractive-multiplier
 $\sigma(\cA)'$-module maps acting on $H^{2}(E, \sigma)$. For further
 details on these ideas, we refer to \cite{BBLS}.

 In short there are at least two points of view to the Schur class
 $\cS_{E,\sigma}$.  The first is the connection with the
 representation theory for $\cF^{\infty}(E)$ as developed in the work
 of Muhly-Solel \cite{MS04, MSSchur, MSPoisson}, while the second is
 the view of the Schur class as contractive-multiplier module maps
 from \cite{BBFtH}.  Our purpose here is to ignore all this finer
 structure and use the most direct definition \eqref{MSpoint-eval}
 of the point-evaluation to make explicit how the Muhly-Solel
 Nevanlinna-Pick interpolation theorem hooks up with the existing
 literature on assorted   generalizations (in particular, to
 multivariable contexts) of Nevanlinna-Pick interpolation.

  \subsection{A $W^{*}$-correspondence Nevanlinna-Pick interpolation theorem}
  \label{subS:MS-NP}

  The solution to the $W^{*}$-correspondence version of the Nevanlinna-Pick
  interpolation problem is as follows.  In the statement of the result we make use of
  the Szeg\"o kernel ${\mathbb K}_{E,\sigma}$ defined as in \eqref{Szego-ker}.

 \begin{theorem} \label{T:MS-NP}
     Suppose that we are given a correspondence-representation pair
     $(E, \sigma)$ together with a subset $\Om$ of $\BD((E^{\sigma})^{*})$
     and two functions $F \colon \Om \to \cL(\cE)$ and $G \colon \Om \to \cL(\cE)$.
     Then the following are equivalent.
     \begin{enumerate}
     \item[(1)] There exists a function $S \colon {\mathbb
     D}((E^{\sigma})^{*}) \to \cL(\cE)$ in the Schur class
     $\cS_{E, \sigma}$ such that
     $$
       G S|_{\Om} = F.
     $$

     \item[(2)] The kernel ${\mathbb K}_{F} \colon \Om \times \Om \to
     \cL(\sigma(\cA)', \cL(\cE))$ given by
     $$
     {\mathbb K}_{F}(\eta, \eta')[b] = G(\eta){\mathbb
     K}_{E,\sigma}(\eta, \eta')[b]G(\zeta)^{*} - F(\eta) {\mathbb K}_{E,
     \sigma}(\eta, \zeta)[b] F(\zeta)^{*}
     $$
    is completely positive.

    \item[(3)] For each choice of $\eta_{1}, \dots, \eta_{n} \in \om$,
    the map $\varphi_{\eta_{1}, \dots, \eta_{n}}$ from
    $(\sigma(\cA)')^{n \times n}$ to $\cL(\cE)^{n \times n}$ given by
    $$
    \varphi_{\eta_{1}, \dots, \eta_{n}}\left([b_{ij}]_{i,j=1}^n\right) =
    \left[ G(\eta_{i}){\mathbb
    K}_{E,\sigma}(\eta_{i,}\eta_{j})[b_{ij}]G(\eta_{j})^{*} - F(\eta_{i}) {\mathbb
    K}_{E,\sigma}(\eta_{i}, \eta_{j})[b_{ij}] F(\eta_{j})^{*}
    \right]_{i,j=1}^n
    $$
    is a completely positive (or even just a positive) map for any $n=1,2,3,\ldots$.
  \end{enumerate}
  \end{theorem}

\begin{remark}\label{R:Xfin}
{\em In case $\Om$ is finite, say $\Om=\{\eta_1,\ldots,\eta_N\}$, and $X_1=G(\eta_1),\ldots,X_N=G(\eta_N)$
and $Y_1=F(\eta_1),\ldots,Y_N=F(\eta_N)$, we can always assume $n$ in condition (3) to be of the form
$n=kN$ for some positive integer $k$ and the sequence of points from $X$ to be the sequence
$\eta_1,\ldots,\eta_N$ repeated $k$ times. It then follows that condition (3) can be written as:} The map
$\vph:\si(\cA)'^{N\times N}\to\cL(\cE)^{N\times N}$ given by
\begin{equation}\label{XfinCriterion}
\vph\left(\mat{c}{b_{i,j}}_{i,j=1}^N\right)=\left[ X_i{\mathbb
    K}_{E,\sigma}(\eta_{i,}\eta_{j})[b_{ij}]X_j^* - Y_i {\mathbb
    K}_{E,\sigma}(\eta_{i}, \eta_{j})[b_{ij}] Y_j^*
    \right]_{i,j=1}^N
\end{equation}
is a completely positive map.
\end{remark}

  \begin{proof}[Proof of Theorem \ref{T:MS-NP}]  Note that the equivalence of (2) and (3) is just a
      particular case of Proposition \ref{P:cp}.
      The equivalence of (1) and (3) for the case that $\Om$
      is finite is given in \cite{MS04} (see also Remark \ref{R:Xfin}).
      For the case where $\Om = {\mathbb D}((E^{\sigma})^{*})$ and
      $G(\eta)=I_\cE$ for each $\eta\in \Om$ the equivalence of (1) and (2) is given in
      \cite{MSSchur} (see also \cite{BBFtH} for the
      reproducing-kernel point of view).  The case of a general $\Om$ can be
      seen to follow from the case of a finite $\Om$ via a standard weak-$*$
      compactness argument; one way to organize this argument is as an
      application of  Kurosh's Theorem (see \cite[page 30]{AgMcC}).
   \end{proof}


\subsection{The second Muhly-Solel point-evaluation and associated Nevanlinna-Pick theorem}
\label{subS:MS-alt}

In the recent paper \cite{MSPoisson} Muhly and Solel introduced a
Poisson kernel, and applied
this object to define a second point-evaluation for elements of the Toeplitz algebra
$H^\infty (E)$. As pointed out in \cite{MSPoisson}, this point-evaluation has certain
characteristics that resemble those of  the point-evaluation used in discrete-time time-varying
interpolation and system theory, as developed in the 1990s; cf. \cite{BGK92,ADD90,FFGK98}.
In \cite{FFGK96} it was observed that many time-varying interpolation problems can be recast as
classical interpolation problems with an operator argument. In this section we prove a
Nevanlinna-Pick interpolation theorem for the alternative point-evaluation of \cite{MSPoisson};
in the examples to follow (see Subsections \ref{subS:MS-1var} and \ref{subS:quiver2} below) we
show that this Nevanlinna-Pick theorem indeed corresponds to the operator-argument versions in
the settings considered there.

For the second point-evaluation, the points are formed by pairs $(\ze,a)$,
where $a$ is from the $W^*$-algebra $\cA$ and $\ze$ is from the set
\[
\BD(E^*):=\{\ze \colon\ze^*\in E,\ \|\ze^*\|_E<1\}.
\]
Given such a pair $(\ze,a)$ we set
\[
\ze_{(n)}^*:=\ze^*\otimes\cdots\otimes\ze^*\in E^{\otimes n}\text{ for }n=1,2,\ldots,
\]
and define
\[
W_{\ze^*,a^*}=\vph_{\infty}(a^*)\mat{c}{1_\cA\\T^{(0)}_{\ze_{(1)}^*}\\T^{(0)}_{\ze_{(2)}^*}\\\vdots}
\in\cL^a(\cA,\cF^2(E)),
\]
where $\vph_\infty(a^*)$ and $T^{(0)}_{\ze_{(n)}^*}$ are given by (\ref{ArepFock}) and
(\ref{crea-n}), respectively. The Poisson kernel defined in \cite{MSPoisson} is the special
case of $W_{\ze^*,a^*}$ with $a^*=T^{(0)*}_{\ze^*} T^{(0)}_{\ze^*}$. In that case $W_{\ze^*,a^*}$
is well-defined even if $\|\ze^*\|=1$, and $W_{\ze^*,a^*}$ is a contractive operator. It is easy to
see that for each $\ze\in\BD(E^*)$ and $a\in\cA$ the operator $W_{\ze^*,a^*}$ still defines a bounded
operator, but not necessarily contractive.

For an $R\in\cF^\infty(E)$ we define the evaluation of $R$ in a point $(\ze,a)$ from $\BD(E^*)\times \cA$
by
\[
\widehat{R}(\ze,a)=W_{\ze^*,a^*}^* R|_\cA\in \cL^a(\cA)=\cA.
\]
Here we used the fact that for any $W^*$-algebra $\cA$, considered in the standard way as a
$W^*$-correspondence over itself, we can identify $\cL^a(\cA)$ with the
$C^*$-algebra $\cA$ itself.
If $R\in\cF^\infty(E)$ is given by a sequence of elements of $E^{\otimes n}$, $n=0,1,\ldots$,
as in Proposition \ref{P:Toeplitz-recipe}, then $\widehat{R}(\ze,a)$ can be written more concretely.

\begin{proposition}\label{P:altpoint}
Let $R\in\cF^\infty(E)$ be given by (\ref{Rdec})-(\ref{crea-n}) with $\xi^{(n)}\in E^{\otimes n}$
for $n=0,1,2,\ldots$. Then for any $(\ze,a)\in\BD(E^*)\times\cA$ we have
\[
\widehat{R}(\ze,a)=\sum_{n=0}^\infty\inn{a\xi^{(n)}}{\ze_{(n)}^*}.
\]
\end{proposition}

\begin{proof}
The statement follows directly from the fact that for any $a'\in\cA$ we have
\begin{eqnarray*}
\widehat{R}(\ze,a)a'&=&\sum_{n=0}^\infty T^{(0)*}_{\ze_{(n)}^*}\vph_{n}(a)T^{(0)}_{\xi^{(n)}}a'
=\sum_{n=0}^\infty T^{(0)*}_{\ze_{(n)}^*} a\cdot\xi^{(n)}\cdot a'\\
&=&\sum_{n=0}^\infty\inn{a\cdot\xi^{(n)}\cdot a'}{\ze_{(n)}^*}
=\sum_{n=0}^\infty\inn{a\cdot\xi^{(n)}}{\ze_{(n)}^*}a'.
\end{eqnarray*}
\end{proof}

The following lemma will be useful in the sequel.

\begin{lemma}\label{L:inter}
For any $(\ze,a)\in\BD(E^*)\times \cA$ and $R\in\cF^\infty(E)$ we have
\begin{equation}\label{indent}
W_{\ze^*,a^*}^*R=W_{\ze^*,\widehat{R}(\ze,a)^*}^*.
\end{equation}
\end{lemma}

\begin{proof}
Let $(\ze,a)\in\BD(E^*)\times \cA$ and $R\in\cF^\infty(E)$. Let $\xi^{(n)}\in E^{\otimes n}$,
for $n=0,1,2,\ldots$, such that $R$ is given by (\ref{Rdec})-(\ref{crea-n}). First observe that for
$n,k=0,1,2,\ldots$ and for any $\rho\in E^{\otimes k}$ we have
\begin{eqnarray*}
T^{(0)*}_{\ze^*_{(n+k)}}\vph_{n+k}(a)(T^{(0)}_{\xi^{(n)}}\otimes I_{E^{\otimes k}})\rho
&=&T^{(0)*}_{\ze^*_{(n+k)}} (a\cdot\xi^{(n)})\otimes \rho\\
&=&\inn{(a\cdot\xi^{(n)})\otimes \rho}{\ze^*_{(n+k)}}\\
&=&\inn{(a\cdot\xi^{(n)})\otimes \rho}{\ze^*_{(n)}\otimes\ze^*_{(k)}}\\
&=&\inn{\inn{a\cdot\xi^{(n)}}{\ze^*_{(n)}}\cdot\rho}{\ze^*_{(k)}}\\
&=&T^{(0)*}_{\ze^*_{(k)}}\vph_{k}(\inn{a\cdot\xi^{(n)}}{\ze^*_{(n)}})\rho.
\end{eqnarray*}
It then follows that the $k^\tu{th}$ entry in the infinite bock row matrix $W_{\ze^*,a^*}^*R$
is equal to
\begin{eqnarray*}
\sum_{n=0}^\infty T^{(0)*}_{\ze^*_{(n+k)}}\vph_{n+k}(a)(T^{(0)}_{\xi^{(n)}}\otimes I_{E^{\otimes k}})
&=&\sum_{n=0}^\infty T^{(0)*}_{\ze^*_{(k)}}\vph_{k}(\inn{a\cdot\xi^{(n)}}{\ze^*_{(n)}})\\
&=&T^{(0)*}_{\ze^*_{(k)}}\vph_{k}\left(\sum_{n=0}^\infty\inn{a\cdot\xi^{(n)}}{\ze^*_{(n)}}\right)\\
&=&T^{(0)*}_{\ze^*_{(k)}}\vph_{k}(\widehat{R}(\ze,a)).
\end{eqnarray*}
This proves our claim.
\end{proof}

One application of Lemma \ref{L:inter} is the following multiplicative law for elements in
the Toeplitz algebra $\cF^\infty(E)$ with respect to the point-evaluation considered in the
present subsection.

\begin{proposition}\label{P:multilaw}
Let $R,S\in\cF^\infty(E)$ and $(\ze,a)\in\BD(E^*)\times \cA$. Then
\[
(\widehat{RS})(\ze,a)=\widehat{S}(\ze,\widehat{R}(\ze,a)).
\]
\end{proposition}

\begin{proof}
It follows from Lemma \ref{L:inter} that
\[
(\widehat{RS})(\ze,a)=W_{\ze,a^*}^*R S|_\cA=W_{\ze,\widehat{R}(\ze,a)^*}^*S|_\cA
=\widehat{S}(\ze,\widehat{R}(\ze,a)).
\]
\end{proof}

It follows, in particular, from this proposition that the alternative point-evaluation is not
multiplicative; unlike the first Muhly-Solel point-evaluation discussed in
Subsection \ref{subS:points}; cf., \cite[Proposition 4.4]{BBFtH}. The multiplicative law
obtained in this case shows more resemblance to that appearing in the operator-argument
functional calculus; cf., formula I.2.7 in \cite{FFGK98}.

We now prove the following Nevanlinna-Pick interpolation theorem.

%

\begin{theorem}\label{T:MSOA-NP}
Suppose that we are given $(\ze_1,a_1),\ldots,(\ze_N,a_N)\in\BD(E^*)\times \cA$ and
$w_1,\ldots,w_N\in\cA$. Then there exists an $S\in\cF^\infty(E)$ with $\|S\|\leq1$ and
$\widehat{S}(\ze_k,a_k)=w_k$ for $k=1,\ldots,N$ if and only if the operator matrix
\begin{equation}\label{MSOA-Pick}
\left[\sum_{n=0}^\infty\inn{(a_ia_j^*-w_iw_j^*)\ze_{j\,(n)}^*}{\ze_{i\,(n)}^*}_{E^{\otimes n}} \right]_{i,j=0}^N
\end{equation}
is a positive element of $\cA^{N\times N}$.
\end{theorem}

\begin{proof}
Assume we have an $S\in\cF^\infty(E)$ with $\|S\|\leq1$ and $\widehat{S}(\ze_k,a_k)=w_k$ for
$k=1,\ldots,N$. Then, by Lemma \ref{L:inter},
$W_{\ze_k^*,a_k^*}^*S=W_{\ze_k^*,\widehat{S}(\ze_k,a_k)^*}^*=W_{\ze_k^*,w_k^*}^*$ for $k=1,\ldots,N$,
and thus
\[
S^*\mat{ccc}{W_{\ze_1^*,a_1^*}&\cdots&W_{\ze_N^*,a_N^*}}
=\mat{ccc}{W_{\ze_1^*,w_1^*}&\cdots&W_{\ze_N^*,w_N^*}}.
\]
Since $\|S\|\leq 1$, this implies that
\begin{equation}\label{eqn1}
\begin{array}{l}
\mat{c}{W_{\ze_1^*,a_1^*}^*\\\vdots\\W_{\ze_N^*,a_N^*}^*}
\mat{ccc}{W_{\ze_1^*,a_1^*}\!\! &\!\! \cdots\!\! &\!\! W_{\ze_N^*,a_N^*}}
-\mat{c}{W_{\ze_1^*,w_1^*}^*\\\vdots\\W_{\ze_N^*,w_N^*}^*}
\mat{ccc}{W_{\ze_1^*,w_1^*}\!\! &\!\! \cdots\!\! &\!\! W_{\ze_N^*,w_N^*}}\\[.2cm]
\hspace*{1.5cm}=\mat{c}{W_{\ze_1^*,a_1^*}^*\\\vdots\\W_{\ze_N^*,a_N^*}^*}(I-SS^*)
\mat{ccc}{W_{\ze_1^*,a_1^*}\!\! &\!\! \cdots\!\! &\!\!
W_{\ze_N^*,a_N^*}} \succeq 0\quad \text{(in $\cA^{N\times N}$)}.
\end{array}
\end{equation}
Next observe that for any $(\ze,a),(\ze',a')\in\BD(E^*)\times \cA$ we have
\[
W^*_{\ze^*,a^*}W_{\ze'^*,a'^*}=\sum_{n=0}^\infty\inn{aa'^* \ze'^{*}_{(n)}}{\ze_{(n)}^*}\in\cA=\cL^a(\cA).
\]
{}From this computation it follows that the left hand side of the identity in (\ref{eqn1})
is equal to the Pick matrix (\ref{MSOA-Pick}). Thus (\ref{MSOA-Pick}) is a positive element
of $\cA^{N\times N}$.

Conversely, assume that (\ref{MSOA-Pick}) is positive in $\cA^{N\times N}$.
Let $\si:\cA\to\cL(\cE)$ be a faithful (i.e., injective) representation of $\cA$
into $\cL(\cE)$ for some Hilbert space $\cE$. Fix $(\ze,a)\in\BD(E^*)\times \cA$.
For notational convenience we set
$\widetilde W_{\ze^*,a^*}=W_{\ze^*,a^*}\otimes I_\cE\in\cL^a(\cE,\cF^2(E,\si))$, where
$\cF^2(E,\si):=\cF^2(E)\otimes_\si\cE$.
We claim that for each $b\in\si(\cA)'$ and $\mu\in E^\si$ we have
\begin{equation}\label{eqs}
\begin{array}{c}
\widetilde W_{\ze^*,a^*}^*(I_{\cF^2(E)}\otimes b)=b\widetilde W_{\ze^*,a^*}^*,\\[.2cm]
\widetilde W_{\ze^*,a^*}^*(I_{\cF^2(E)}\otimes \mu)
=((T^{(0)*}_{\ze^*}\otimes I_\cE)\mu)\widetilde W_{\ze^*,a^*}^*.
\end{array}
\end{equation}
The first identity follows directly from the computation
\[
\widetilde W_{\ze^*,a^*}^*(I_{\cF^2(E)}\otimes b)
=(W_{\ze^*,a^*}^*\otimes I_\cE)(I_{\cF^2(E)}\otimes b)
=(I_\cA\otimes b)(W_{\ze^*,a^*}^*\otimes I_\cE)=b \widetilde W_{\ze^*,a^*}^*.
\]
Here we used the fact that $\cA\otimes_\si\cE$ can be identified with $\cE$.
To see that the second identity in (\ref{eqs}) holds we show that the $n^\tu{th}$ entry
\[
(T^{(0)*}_{\ze_{(n+1)}^*}\vph_{n+1}(a)\otimes I_\cE)(I_{E^{\otimes n}}\otimes \mu)
:E^{\otimes n}\otimes_\si\cE\to\cA
\]
in the infinite block row matrix $\widetilde W_{\ze^*,a^*}^*(I_{\cF^2(E)}\otimes \mu)$ is
equal to
\[
((T^{(0)*}_{\ze^*}\otimes I_\cE)\mu)(T_{\ze_{(n)}^*}^{(0)*}\vph_{n}(a)\otimes I_\cE).
\]
It suffices to check the equality for elements from $E^{\otimes n}\otimes_\si\cE$
of the form $\xi^{(n)}\otimes e$ with $\xi^{(n)}\in E^{\otimes n}$ and $e\in\cE$.
For such elements $\xi^{(n)}\otimes e\in E^{\otimes n}\otimes_\si\cE$ and for
each $e'\in\cE$ we obtain that the
\begin{equation*}
\begin{array}{l}
\inn{(T^{(0)*}_{\ze_{(n+1)}^*}\vph_{n+1}(a)\otimes I_\cE)(I_{E^{\otimes n}}\otimes \mu)
(\xi^{(n)}\otimes e)}{e'}_{\cE}=\\[.3cm]
\hspace*{1.5cm}=\inn{(a\cdot\xi^{(n)})\otimes\mu e}{\ze_{(n+1)}^*\otimes e'}
=\inn{\inn{a\cdot\xi^{(n)}}{\ze_{(n)}^*}\cdot\mu e}{\ze^*\otimes e'}\\[.3cm]
\hspace*{1.5cm}=\inn{\mu \si(\inn{a\cdot\xi^{(n)}}{\ze_{(n)}^*})e}{\ze^*\otimes e'}
=\inn{((T_{\ze_{(1)}^*}^{(0)*}\otimes I_\cE)\mu)\si(T_{\ze_{(n)}^*}^{(0)*}\vph_{n}(a)\xi^{(n)})e}{e'}\\[.3cm]
\hspace*{1.5cm}=\inn{((T_{\ze_{(1)}^*}^{(0)*}\otimes I_\cE)\mu)
(T_{\ze_{(n)}^*}^{(0)*}\vph_{n}(a)\otimes I_\cE)(\xi^{(n)}\otimes e)}{e'}.
\end{array}
\end{equation*}
Thus our claim follows.
Now set
\[
\cH_{\ze,a}:=\ov{\im \widetilde W_{\ze,a}}\subset \cF^2(E,\si):=\cF^2(E)\otimes_\si\cE.
\]
Using the intertwining relations (\ref{eqs}) we see that $\cH_{\ze,a}$ is an invariant subspace
for the operators $I_{\cF^2(E)}\otimes \mu^*$ and $I_{\cF^2(E)}\otimes b$ for each $\mu\in E^\si$ and
$b\in\si(\cA)'$.

Identifying $\cA$ with its image in $\cL(\cE)$ under the representation $\si$,
we see that
the Pick matrix (\ref{MSOA-Pick}) defines a positive semidefinite element in $\cL(\cE)^{N\times N}$.
Moreover, from the first part of the proof we see that, equivalently, the left hand side of the
identity in (\ref{eqn1}), with $W_{\ze_k^*,a_k^*}^*$ replaced by $\widetilde W_{\ze_k^*,a_k^*}^*$,
defines positive semidefinite element in $\cL(\cE)^{N\times N}$. Thus, by the Douglas
factorization lemma \cite{D66}, we can define a contraction
\[
V:\cM\to\cN
\quad\text{by}\quad\widetilde{W}_{\ze_k^*,a_k^*}V=\widetilde{W}_{\ze_k^*,w_k^*}^*\text{ for }k=1,\ldots,N,
\]
where we set
\[
\cM:=\bigvee_{k=1}^N \im \widetilde{W}_{\ze_k^*,w_k^*}\ands
\cN:=\bigvee_{k=1}^N \im \widetilde{W}_{\ze_k^*,a_k^*}.
\]
Now, for $b\in\si(\cA)'$ and $\mu\in E^\si$, let $T_b$, $\widetilde{T}_b$ and
$T_\mu$, $\widetilde{T}_\mu$ be the compressions of $I_{\cF^2(E)}\otimes b$ and
$I_{\cF^2(E)}\otimes \mu$ to $\cM$ and $\cN$, respectively. Then, from the
intertwining relations (\ref{eqs}) it follows for $k=1,\ldots,N$ that
\[
\widetilde W_{\ze_k^*,a_k^*}^* T_b V=b\widetilde W_{\ze_k^*,a_k^*}^* V=b\widetilde W_{\ze_k^*,w_k^*}^*
=\widetilde W_{\ze_k^*,w_k^*}^*\widetilde{T}_b=\widetilde W_{\ze_k^*,a_k^*}^*V\widetilde{T}_b,
\]
and with a similar computation $\widetilde W_{\ze_k^*,a_k^*}^* T_\mu V
=\widetilde W_{\ze_k^*,a_k^*}^*V\widetilde{T}_\mu$. Since these identities hold for all $k=1,\ldots,N$,
we obtain that $V$ intertwines $T_b$ with $\widetilde{T}_b$ and $T_\mu$ with $\widetilde{T}_\mu$
for each $b\in\si(\cA)'$ and $\mu\in E^\si$. It then follows from Theorem \ref{T:CLT} that there
exists an $S\in\cF^\infty(E)$ with $\|S\|\leq1$, such that $(S^*\otimes I_\cE)\cN\subset\cM$
and $V=P_\cN(S\otimes I_\cE)|_\cM$. Hence for $k=1,\ldots,N$ we have
\begin{eqnarray*}
\si(\widehat S(\ze_k,a_k))&=&\widehat S(\ze_k^*,a_k)\otimes I_\cE
=(W_{\ze_k^*,a_k^*}^*S|_\cA)\otimes I_\cE=\widetilde W_{\ze_k^*,a_k^*}^*(S\otimes I_\cE)|_\cE\\
&=&\widetilde W_{\ze_k^*,w_k^*}^*|_\cE=\si(w_k).
\end{eqnarray*}
Thus $\widehat S(\ze_k,a_k)=w_k$ for $k=1,\ldots,N$, because $\si$ is injective.
\end{proof}

\subsection{Example:~Operator-valued Nevanlinna-Pick interpolation on the unit disk ${\mathbb D}$.}
\label{subS:MS-1var}

As a first example we consider the case where $E=\cA=\cL(\cV)$ for $\cV$ some Hilbert space.
Then $E$ is a $W^*$-$\cA$-correspondence with left and right actions given by the usual left and
right multiplication in $\cL(\cV)$ and with with $\cL(\cV)$-valued inner product given by
\[
\inn{R}{Q}=Q^*R\text{ for all }R,Q\in\cL(\cV).
\]
The preliminaries of this case have been spelled out in Subsection 6.1 of \cite{BBFtH}. We
recall here the results needed in the sequel. The balancing (\ref{balance}) in the tensor
products causes the tensor spaces $E^{\otimes n}$ to reduce to $E=\cL(\cV)$ via the
identification
\begin{equation}\label{1var-ident}
R_n\otimes\ldots\otimes R_1= I_\cV\otimes\cdots\otimes I_\cV\otimes (R_n\cdots R_1)
\equiv R_n\cdots R_1.
\end{equation}
Then the Fock space $\cF^2(E)$ is equal to $\cL(\cV,\ell^2_\cV(\BZ_+))$, and the Toeplitz
algebra $\cF^\infty(E)$ is the algebra of block Toeplitz matrices
\begin{equation}  \label{Rmatrix}
R = \begin{bmatrix} R_{0} & 0& 0&\cdots \\[.05cm]
R_{1} & R_{0} &0&\cdots \\[-.1cm]
R_{2} & R_{1} & R_{0} &\ddots \\[-.1cm]
\vdots & \ddots & \ddots & \ddots \end{bmatrix}, \quad  R_{j} \in
\cL(\cV)
\end{equation}
that induce bounded operators on $\cL(\cV,\ell^2_\cV(\BZ_+))$ via
matrix multiplication:  if we write $X \in \cL(\cV,
\ell^{2}_{\cV}(\BZ_{+}))$ in block column-matrix form
$$
  X = \begin{bmatrix} X_{1} \\ X_{2} \\  \vdots \end{bmatrix} \text{
  with } X_{k} \in \cL(\cV) \text{ for } k = 1,2, \dots,
$$
then  $R \colon X \mapsto  R \cdot X$ where the matrix for $R$ is given by
\eqref{Rmatrix}.
By the closed-graph theorem, for any such fixed
$X$, $R \cdot X \in \cL(\cV, \ell^{2}_{\cV}(\BZ_{+}))$ if and only if
$(R \cdot X) \cdot v = R \cdot (X \cdot v) \in \ell^2_\cV(\BZ_+)$ for each $v \in
\cV$.  As $X$ and $v$ are arbitrary, an equivalent condition is that
$R \cdot w \in \ell^{2}_{\cV}(\BZ_{+})$ for each $w \in
\ell^{2}_{\cV}(\BZ_{+})$, i.e. (again by the closed-graph theorem)
$R$ is bounded when viewed as an ordinary Toeplitz matrix acting on
$\ell^{2}_{\cV}(\bbZ_{+})$.
Furthermore one can check that the norm
of $R$ viewed as an element of $\cL(\cL(\cV, \ell^{2}_{\cV}(\BZ_{+})))$ is
equal to its norm when viewed as an element of
$\cL(\ell^{2}_{\cV}(\BZ_{+}))$.
Thus a sequence of coefficients
$\{R_{n}\}_{n \in {\mathbb Z}}$ with $R_{n} \in \cL(\cV)$ and $R_{n}=0$ for $n < 0$
gives rise to an element $[R_{i-j}]_{i,j \in {\mathbb Z}_{+}}$ in $\cF^{\infty}(E)$
exactly when $R(\lambda): = \sum_{n=0}^{\infty} R_{n} \lambda^{n}$ is in the
operator-valued Schur class $\cS(\cV)$.

Let $\cG$ be another Hilbert space. We take $\cE$ to be the tensor product
Hilbert space $\cV\otimes\cG$, and define a representation $\si:\cA\to\cL(\cE)$ by
\[
\si(R)=R\otimes I_\cG\text{ for }R\in\cL(\cV)=\cA.
\]
Then $(E \otimes_{\sigma} \cE)^{*}$ can be identified with $\cE$ via the identification
\[
R \otimes_{\sigma} (v \otimes g) = I_{\cU} \otimes_{\sigma}((Rv)\otimes g)\equiv (Rv)\otimes g
\text{ for } R \in \cL(\cV),\, v \in \cV, \, g \in\cG.
\]
We identify $E^{\sigma}$ as
\[
E^{\sigma} = \{ \eta \in \cL(\cE) \colon \eta (R \otimes
I_{\cG}) = (R \otimes I_{\cG})\eta \text{ for all } R \in \cL(\cV)\} =
I_{\cV} \otimes \cL(\cG),
\]
and hence the generalized disk ${\mathbb D}((E^{\sigma})^{*})$ can
be identified with the set of strict contraction operators on $\cG$.
For $\eta \in {\mathbb D}((E^{\sigma})^{*})$ identified with a
 strict contraction operator $Z \in \cL(\cG)$ and for
 $R = [R_{i-j}]_{i,j \in\BZ_+} \in \cF^{\infty}(E)$, one can work out
 that the associated first point-evaluation $\widehat R(Z)$ is given by a
 tensor functional-calculus
 $$
   \widehat R(Z) = \sum_{n=0}^{\infty}R_{n} \otimes Z^{n}.
 $$

Thus, the correspondence Nevanlinna-Pick interpolation problem for
this setting becomes:  {\em Given strict contraction operators
$Z_{1}$, $\dots$, $Z_{N}$ in $\cL(\cG)$ and operators $X_{1},\dots,X_{N}$
and $Y_{1},\dots,Y_{N}$ in $\cL(\cV \otimes \cG)$, find a Schur class function
$S(\lambda) = \sum_{n=0}^{\infty} S_{n} \lambda^{n} \in \cS(\cV)$
 so that}
\begin{equation}\label{LTT}
   X_i S(Z_{i}):= X_i\sum_{n=0}^{\infty} S_{n} \otimes Z_{i}^{n} = Y_{i} \text{
   for } i = 1, \dots, n.
\end{equation}
  We call this the {\em left-tangential tensor functional-calculus Nevanlinna-Pick
 interpolation problem} (\textbf{LTT-NP}).  Note that this contains the
 \textbf{LT-NP} (take $\cG = {\mathbb C}$) and \textbf{LTRD-NP} (take $\cV = {\mathbb C}$)
 problems discussed in Subsections \ref{subS:standard-1var} and
 \ref{subS:RD-1var}, respectively, as special cases.   The solution to the
 \textbf{LTT-NP} problem is readily obtained as a direct application of
 Theorem \ref{T:MS-NP}.

\begin{theorem} \label{T:TFC-NP}
     Suppose that we are given the data set
  $${\mathfrak D}: \quad  Z_{1}, \dots, Z_{N} \in \cL(\cG), \quad
  X_{1}, \dots, X_{N},Y_1,\ldots,Y_N \in \cL(\cV \otimes \cG)
  $$
  for an {\em \textbf{LTT-NP}} problem.  Then a solution $S \in
  \cS(\cV)$ exists if and only if any of the following equivalent
  conditions holds:
  \begin{enumerate}
      \item[(1)] The kernel ${\mathbb K}_{{\mathfrak D}}$ mapping $ \{1,
      \dots, N\} \times \{1, \dots, N\}$ into $\cL^a(\cL(\cG), \cL(\cV
      \otimes \cG))$ given by
      $$
      {\mathbb K}_{{\mathfrak D}}(i,j)[B] =\sum_{n=0}^{\infty}
      X_i\left(I_{\cV} \otimes \left(  Z_{i}^{n} B
      Z_{j}^{*n} \right)\right)X_j^* - Y_{i} \left(
      I_{\cV} \otimes \left( Z_{i}^{n} B
        Z_{j}^{*n}\right) \right) Y_{j}^{*}
      $$
      is completely positive.

      \item[(2)] The map $\varphi \colon \cL(\cG)^{N \times N} \to \cL(\cV
      \otimes \cG)^{N \times N}$ given by
 \begin{align}\label{vph}
  & \varphi \left( \left[B_{ij}\right]_{i,j=1, \dots, N}\right)=  \\
  &\nonumber \quad =
   \left[\sum_{n=0}^{\infty} X_i\left(I_{\cV} \otimes \left(  Z_{i}^{n}
   B_{ij}Z_{j}^{*n} \right)\right)X_j^* - Y_{i} \left(
   I_{\cV} \otimes \left(  Z_{i}^{n}
      B_{ij}Z_{j}^{*n} \right) \right) Y_{j}^{*} \right]_{i,j=1}^N
   \end{align}
   is a completely positive map.
   \end{enumerate}
      \end{theorem}

Note that criteria of Theorem \ref{T:TFC-NP} give seemingly
different criteria than the ones obtained in Section \ref{S:1var}.
We now show how, after some reductions, one can see directly the
equivalence of the criteria in Theorem \ref{T:TFC-NP} with the
Pick-matrix criteria of Section \ref{S:1var}.

First assume that $\cV$ and $\cG$ admit direct sum decompositions
\[
\cV=\cU\oplus\cY\oplus\cC\quad\text{and}\quad \cG=\cZ\oplus\cD,
\]
and that the operators in the data set ${\mathfrak D}$ are actually operators
\[
Z_i\in\cL(\cZ),\quad X_i\in\cL(\cY\otimes\cZ,\cC\otimes \cD),\quad
Y_i\in\cL(\cU\otimes\cZ,\cC\otimes \cD)\text{ for }i=1,\ldots N,
\]
identified with operators in $\cL(\cG)$ and $\cL(\cE)$, respectively, by adding
zero-operators in the operator matrix decompositions. A solution to the ``non-square''
\textbf{LTT-NP} problem: {\em Given strict contraction operators $Z_{1},\dots,Z_{N}$
in $\cL(\cZ)$, and operators $X_{1},\dots,X_{N}\in\cL(\cY\otimes\cZ,\cC\otimes \cD)$
and $Y_{1},\dots,Y_{N}\in\cL(\cU\otimes\cZ,\cC\otimes \cD)$, find a Schur class function
$S(\lambda) = \sum_{n=0}^{\infty} S_{n} \lambda^{n} \in \cS(\cU,\cY)$
satisfying (\ref{LTT})}, can then be obtained from a solution $\tilde S\in\cS(\cE)$ just by
assigning $S(\la)$ to be the compression of $\tilde  S(\la)$ to $\cL(\cU,\cY)$.
Conversely, a solution $S\in\cL(\cU,\cY)$ to the ``non-square'' problem defines a solution
$\tilde S\in\cL(\cE)$ to the ``square'' problem by embedding the values $S(\la)\in\cL(\cU,\cY)$
into $\cL(\cE)$. It thus follows that a solution to the ``non-square'' problem exists if and only
if the map $\vph$ in criterion 2 of Theorem \ref{T:TFC-NP} (which now can also be seen as a map
from $\cL(\cZ)^{N\times N}$ into $\cL(\cZ\otimes\cD)^{N\times N}$) is completely positive.

Transforming a result from a ``square'' Nevanlinna-Pick
problem to one for Schur class functions that take ``non-square'' values can even be
performed on the level of $W^*$-correspondence Nevanlinna-Pick interpolation using
techniques developed in \cite{MSSchur}.

\paragraph{Standard functional calculus Nevanlinna-Pick interpolation}
Assume that $\cG=\BC$. Hence the data for our Nevanlinna-Pick interpolation is of the
{\bf LT-NP} form:
\[
\fD:\ \la_1,\ldots,\la_N\in\BD,\ X_1,\ldots,X_N\in\cL(\cY,\cC),\ Y_1,\ldots,Y_N\in\cL(\cU,\cC).
\]
In that case, the kernel $\BK_\fD:\{1,\ldots,N\}\times\{1,\ldots,N\}\to\cL(\BC,\cL(\cL(\cV)))$
reduces to
\[
\BK_\fD(i,j)[c]=c\cdot\frac{X_iX_j^*-Y_iY_i^*}{i-\la_i\ov{\la}_j}.
\]
It is then straightforward to see that complete positivity of the kernel $\BK_\fD$
collapses to positive semidefiniteness of the standard Pick matrix
\[
\mat{c}{\frac{X_iX_j^*-Y_iY_i^*}{i-\la_i\ov{\la}_j}}_{j,i=1}^N.
\]
Thus we recover the criterion of Theorem \ref{T:FOV/LT/RT} (Part 2).

\paragraph{Riesz-Dunford functional calculus Nevanlinna-Pick interpolation}
Next we consider the case that $\cV=\BC$. In that case the data for our
Nevanlinna-Pick interpolation problem takes the {\bf LTRD-NP} form:
\[
\fD:\ Z_1,\ldots,Z_N\in\cL(\cZ),\ X_1,\ldots,X_N,Y_1,\ldots,Y_N\in\cL(\cZ,\cC).
\]
To see that Theorem \ref{T:TFC-NP} also contains the result of Theorem \ref{T:RD-NP} (Part 2), it is
convenient to introduce a third criterion.

 \begin{theorem}  \label{T:TFC-NP2}
 In case $\cV=\BC$,
     in addition to the two conditions (1) and (2) in Theorem \ref{T:TFC-NP},
     a third condition equivalent to the existence of an
     $S\in\cS(\BC,\BC)$ that satisfies (\ref{LTT}) is that the map $\varphi_{*}$ from
     $\cL(\cC)^{N \times N}$ to $\cL(\cC)^{N \times N}$ given by
 \begin{equation}  \label{clasA-Pick3}
     \varphi_{*}\left( \left[ C_{ij}
     \right]_{i,j=1}^N \right) =
     \left[ \sum_{n=0}^{\infty}
 Z_{i}^{*n}( X_i^* C_{ij} X_j-Y_{i}^{*} C_{ij} Y_{j})
 Z_{j}^{n}
     \right]_{i,j=1}^N
  \end{equation}
  be a completely positive map.
 \end{theorem}

\begin{proof} We prove that positivity of $\vph$ is equivalent to positivity of $\vph_*$.
The proof of the equivalence for $k$-positivity with $k\in\BZ_+$ arbitrary
(and thus for complete positivity) goes analogously.
The map $\varphi$ in \eqref{vph} being positive implies that
 \begin{equation}   \label{trace-pos1}
 \sum_{i,j=1}^{N} \operatorname{trace}\left(C_{ji} \left(
\sum_{n=0}^{\infty} X_i Z_{i}^{n}B_{ij}Z_{j}^{*n}X_j^*
- Y_{i}Z_{i}^{n}B_{ij}Z_{j}^{*n}Y_{j}^{*}\right)\right)\geq0
 \end{equation}
for all ${\bf B}=\left[B_{ij}\right]_{i,j=1}^N \succeq 0$ and
${\bf C}=\left[C_{ij} \right]_{i,j=1}^N \succeq 0$ in $\cL(\cG)^{N\times N}$
with ${\bf B}$ also in the trace class. Using the invariance of the trace under cyclic
permutations, we may rewrite \eqref{trace-pos1} as
\begin{equation}   \label{trace-pos2}
  \sum_{i,j=1}^{N} \operatorname{trace}  \left(
  \left(\sum_{n=0}^{\infty}
Z_{j}^{*n}(X_j^* C_{ji} X_i-Y_{j}^{*} C_{ji} Y_{i})Z_{i}^{n}\right)B_{ij} \right)\geq0.
  \end{equation}
From this it follows that $\vph_*$ is a positive map. To see the converse implication
one just reverses the above computation but now with ${\bf C}$ in the trace class.
\end{proof}

In order to transform the completely positive map criterion of Theorem \ref{T:TFC-NP2} into one
of checking positivity of a (possibly infinite) operator matrix we need the following
extension of Choi's theorem \cite{Choi75}.

\begin{theorem}\label{T:ChoiExt}
A weak-$*$ continuous map $\psi:\cL(\cH)\to\cA$, where $\cH$ a separable Hilbert space with
orthonormal basis $\{e_1,\ldots,e_\kappa\}$ (with possibly $\kappa=\infty$) and $\cA$ is a
$C^*$-algebra, is completely positive if and only if the (possibly infinite) block matrix
\begin{equation}\label{compos}
\mat{c}{\psi(e_ie_j^*)}_{i,j=1}^\kappa
\end{equation}
is a positive element of $\cA^{\kappa\times\kappa}$ (i.e., if $\kappa=\infty$, then all
$M$-finite sections are positive in $\cA^{M\times M}$ for each $M\in\BZ_+$).
\end{theorem}

\begin{proof}
If $\kappa<\infty$, then the statement is just Choi's theorem \cite{Choi75}. So assume that
$\kappa=\infty$.
Complete positivity of $\psi$ means that for each $N\in\BZ_+$ the map
\[
\mat{c}{B_{i,j}}_{i,j=1}^N\mapsto \mat{c}{\psi(B_{i,j})}_{i,j=1}^N
\]
from $\cL(\cH)^{N\times N}$ into $\cA^{N\times N}$ is positive. In particular, using
$B_{i,j}=e_ie_j^*$ it follows directly that $\psi$ being completely positive implies that
the block matrix (\ref{compos}) is a positive element of $\cA^{\kappa\times\kappa}$.

Now assume that the block matrix (\ref{compos}) is positive.
For each $M\in\BZ_+$, let $P_M$ denote the projection onto
$\cH_M:=\tu{span}\{e_1,\ldots,e_M\}$. Using Choi's theorem and the fact that the $M$-finite
section $\sbm{\vph(e_ie_j^*)}_{i,j=1}^M$ of (\ref{compos}) is a positive element of $\cA^{M\times M}$ we obtain that the map
$\tilde B\mapsto\psi(P_M \tilde BP_M)$ from $\cH_M$ into $\cA$ is completely positive.
Fix a $N\in\BZ_+$ and a positive $\sbm{B_{i,j}}_{i,j=1}^N$ in $\cL(\cH)^{N\times N}$.
Then for each $M\in\BZ_+$ we have $\sbm{\psi(P_MB_{i,j}P_M)}_{i,j=1}^N\succeq0$.
Moreover, $P_M B_{i,j}P_M$ converges to $B_{i,j}$ in the weak-$*$ topology as $M\to\infty$,
and thus, by hypothesis $\psi(P_M B_{i,j}P_M)\to\psi(B_{i,j})$ in the weak-$*$ topology as
$M\to\infty$ for all $i,j=1,\ldots,N$. This implies that
\[
\mat{c}{\psi(P_M B_{i,j}P_M)}_{i,j=1}^N\to\mat{c}{\psi(B_{i,j})}_{i,j=1}^N\text{ weak-$*$ as }M\to\infty.
\]
Since positivity is preserved under weak-$*$ convergence it follows that
$\sbm{\psi(B_{i,j})}_{i,j=1}^N$ is a positive element of $\cA^{N\times N}$, and thus, because $N$ was
chosen arbitrarily, that $\psi$ is a completely positive map.
\end{proof}

Now assume that $\cC$ is a separable Hilbert space with orthonormal basis $\{e_1,\ldots,e_\kappa\}$.
It then follows from Theorem \ref{T:ChoiExt} that complete positivity of the maps $\vph_*$ given by
(\ref{clasA-Pick3}) reduces to positivity of the Pick matrix
\[
\mat{c}{\sum_{n=0}^{\infty}
 Z_{i}^{*n}( X_i^* e_{i'}e_{j'}^* X_j-Y_{i}^{*} e_{i'}e_{j'}^* Y_{j})
 Z_{j}^{n}}_{(i,i'),(j,j')\in\{1,\ldots,N\}\times\{1,\ldots,\kappa\}}.
\]
Hence we recover Part 2 of Theorem \ref{T:RD-NP}. Indeed, the one thing left to verify is that
the map $\vph_*$ is weak-$*$ continuous, which we leave as an exercise to the reader.

\paragraph{Operator-argument functional calculus Nevanlinna-Pick interpolation}
Finally, we show that the second point-evaluation for $\cF^\infty(E)$ discussed in
Subsection $\ref{subS:MS-alt}$ gives us the operator-argument Nevanlinna-Pick theorem
for the operator-valued Schur class. To see that this is the case, note that
the points are elements of
\[
\BD(E^*)\times \cA=\{T\in\cL(\cV) \colon \|T\|<1 \}\times \cL(\cV).
\]
For each $\ze\in\BD(E^*)$ identified with a strict contraction operator $T\in\cL(\cV)$
the element $\ze_{(n)}^*=\ze^*\otimes\cdots\otimes\ze^*$ of $E^{\otimes n}$  corresponds to
$T^{*n}$ (via the identification (\ref{1var-ident})), and the operator $T^{(0)}_{\ze_{(n)}^*}$
in (\ref{crea-n})
is just multiplication with $T^{*n}$. Thus, we obtain that for $(T,X)\in \BD(E^*)\times \cA$
and $R=\mat{c}{R_{i-j}}_{i,j\in\BZ_+}\in \cF^\infty(E)$ the value of $R$ in $(T,X)$ is equal
to the left-tangential point-evaluation:
\[
\widehat R(T,X)=\sum_{n=0}^\infty T^nXR_n=(X R)^{\wedge L}(T),
\]
where we also use $R$ to indicate the Schur-class function in $\cS(\cV)$ corresponding to
$R\in\cF^\infty(E)$. So, in this case, the Nevanlinna-Pick problem with data
$\ze_k=T_k\in\BD(E^*)$, $a_k=X_k,w_k=Y_k\in\cA$, for $k=1,\ldots,N$, considered in Subsection
\ref{subS:MS-alt} is the {\bf LTOA} problem of Subsection \ref{subS:OA-1var}. One easily
computes that the operator matrix (\ref{MSOA-Pick}) reduces to the Pick matrix $\BP_\text{LTOA}$
in (\ref{Pick-LTOA}). Thus, Theorem \ref{T:MSOA-NP} here gives us precisely Part 1 of
Theorem \ref{T:LTOA/RTOA} (for the case $\cU=\cY=\cC$; the non-square case can be obtained
as explained above).

\subsection{Example:~Nevanlinna-Pick interpolation for Toeplitz
algebras associated with directed graphs}\label{subS:quiver2}

Next, following Section 5 in \cite{MSSchur}, we show how the
Toeplitz algebra associated with a directed graph/quiver can be seen
as an example of the $W^*$-correspondence formalism, and we prove
Theorem \ref{T:Q-NP}. We follow the notation and terminology used in
Subsection \ref{subS:quiver1}. Let $G=\{Q_0,Q_1,s,r\}$ be a quiver
and $\cV$ a given Hilbert space that admits an orthogonal
decomposition $\cV=\oplus_{v\in Q_0}\cV_v$. For each nonnegative
integer $n$ we associate with $Q_n$, the set of paths of length $n$,
the space $C_{G,\cV}(Q_n)$ of continuous $\cL(\cV)$-valued functions
$f$ on $Q_n$, where $f(\ga)$ maps $\cV_{s(\ga)}$ into $\cV_{r(\ga)}$
and $f(\ga)|_{\cV_{s(\ga)}^\perp}=0$ for each $\ga\in Q_n$; of
course the continuity is automatic as $Q_{n}$ is a finite set with
the discrete topology---this notation is used for consistency with
more general settings where $Q_{n}$ is a more general topological
space as in \cite{KP04a,KP04b}. Usually we will just write
$C_G(Q_n)$, rather than $C_{G,\cV}(Q_n)$, for notational
convenience.

The space $C_G(Q_n)$ can be seen as a
$W^*$-$C_G(Q_0)$-correspondence with left and right multiplication and
$C_G(Q_0)$-valued inner-product given by:
\begin{equation*}
\begin{array}{rcl}
(f\cdot\xi)(\gamma)&=&f(r_n(\gamma))\xi(\gamma)\\[.2cm]
(\xi\cdot f)(\gamma)&=&\xi(\gamma)f(s_n(\gamma))\\[.2cm]
\inn{\xi}{\eta}(v)&=&\displaystyle\sum_{s_n(\gamma')=v}\ov{\eta(\gamma')}\xi(\gamma')
\end{array}\quad(\xi,\eta\in C_G(Q_n),f\in C_G(Q_0),\gamma\in Q_n,v\in Q_0).
\end{equation*}
Tensoring $C_G(Q_n)$ with $C_G(Q_m)$ gives $C_G(Q_{n+m})$, hence, in
particular, we obtain for each $n\in\BZ_+$ that
$C_G(Q_n)=C_G(Q_1)^{\otimes n}$. More precisely, for
$\xi_n,\ldots,\xi_1\in C_G(Q_1)$ we can identify
$\xi_1\otimes\ldots\otimes \xi_n\in C_G(Q_1)^{\otimes n}$ with the
element $\xi^{(n)}\in C_G(Q_n)$ given by
\begin{equation}\label{quivIden}
\xi^{(n)}(\ga):=\xi_n(\al_n)\cdots\xi_1(\al_1)\quad\text{for}\quad\ga=(\al_n,\ldots,\al_1)\in Q_n.
\end{equation}
To see that this is the case, note that, rather than viewing elements of $C_G(Q_n)$ as functions,
they are also given by tuples $F=(F_\gamma\in\cL(\cV_{s(\ga)},\cV_{r(\ga)})\colon \gamma\in Q_n)$
that we usually identify with operator matrices
\begin{equation}\label{QuivE}
{\bf F}=\mat{c}{F_{\ga,v}}_{\ga\in Q_n,v\in Q_0}:\bigoplus_{v\in Q_0}\cV_v\to\bigoplus_{\ga\in Q_n}\cV_{r(\ga)},
\text{ where }
F_{\gamma,v}=
   \begin{cases}  F_{\gamma} &\text{if }
       s_n(\gamma)=v, \\
       0 &\text{otherwise}.
       \end{cases}
\end{equation}
The advantage of the operator matrix representation is that for $F,F'\in C_G(Q_n)$ the norm
$\|F\|_{C_G(Q_n)}$ is equal to the operator norm of ${\bf F}$, and $\inn{F}{F'}$ can be
identified with ${\bf F'^*F}$; the operator matrix ${\bf F'^*F}\in\cL(\cV)$ is block diagonal
and, for each $v\in Q_0$, the diagonal entry from $\cL(\cV_v)$ corresponds to the value of
$\inn{F}{F'}(v)$. Note that $C_G(Q_0)$ corresponds to the $W^*$-algebra of block
diagonal operators $\diag_{v\in Q_0}(A_v)$ on $\cV=\oplus_{v\in Q_0}\cV_v$. Moreover, the
operator $T_F^{(0)}$ defined by (\ref{crea-n}) mapping $C_G(Q_0)$ (block diagonal operators) into
$C_G(Q_n)$ can be identified with multiplication with ${\bf F}$.
Subject to this identification, the operator ${\bf F}\otimes I_{C_G(Q_1)}$ from $C_G(Q_1)$ to
$C_G(Q_{n+1})$ corresponds to multiplication with the block operator matrix
\[
{\bf F}\otimes I_{C_G(Q_1)}=\mat{c}{F_{\ga,\al}}_{\ga\in Q_{n+1},\al\in Q_1},\text{ where }
F_{\gamma,\al}=
   \begin{cases}  F_{\gamma'} &\text{if }
       \ga=(\ga',\al), \\
       0 &\text{otherwise}.
       \end{cases}
\]
Analogously, one obtains formulas for ${\bf F}\otimes I_{C_G(Q_1)^{\otimes n}}={\bf F}\otimes I_{C_G(Q_n)}$
for each $n\in\BZ_+$.
Taking a product of such operators
${\bf F_1},{\bf F_2}\otimes I_{C_G(Q_1)},\ldots,{\bf F_n}\otimes I_{C_G(Q_1)^{\otimes n}}$,
with $F_1,F_2,\ldots, F_n\in C_G(Q_1)$, we finally arrive at (\ref{quivIden}).

We now take $E$ to be the $W^*$-$C_G(Q_0)$-correspondence $C_G(Q_1)$.
Elements of the Fock space $\cF^2(C_G(Q_1))=\oplus_{n\in\BZ_+} C_G(Q_n)$ are then given by
tuples $F=(F_\gamma\in\cL(\cV_{s(\ga)},\cV_{r(\ga)})\colon \gamma\in\Ga)$, where $\Ga$ is
the collection of all finite paths of whatever length, such that the operator matrix
\[
{\bf F}=\mat{c}{F_{\gamma,v}}_{v\in Q_0,\gamma\in\Gamma}
\]
is in $\cL(\cV,\oplus_{\ga\in\Ga}\cV_{r(\ga)})$ (i.e., bounded). Here $F_{\gamma,v}$ is
as defined in (\ref{QuivE}). As in Subsection \ref{subS:quiver1} we denote the Hilbert space
$\oplus_{\ga\in\Ga}\cV_{r(\ga)}$ by $\ell^2_\cV(\Ga)$.
The Toeplitz algebra $\cF^\infty(E)$, seen as a subalgebra of $\cL(\ell^2_\cV(\Gamma))$,
is then precisely the Toeplitz algebra $\fL_\Gamma(\cV,\cV)$ defined in Subsection \ref{subS:quiver1}.
That is, an element $R$ in $\cF^\infty(E)$ is given by a tuple
$R=(R_\gamma\in\cL(\cV_{s(\ga)},\cV_{r(\ga)})\colon \gamma\in\Ga)$ with the property that
the operator matrix
\[
{\bf R}=\mat{c}{R_{\ga,\ga'}}_{\ga,\ga'\in\Ga},\text{ where }R_{\ga,\ga'}=R_{\ga\ga'^{-1}}
\text{ (with $R_\tu{undefined}=0$)},
\]
induces a bounded operator on $\ell^2_\cV(\Gamma)$. Here $\ga\ga'^{-1}$ is defined by (\ref{QuivInv}).

In addition to $\cG$ we assume that we are given a Hilbert space $\cG$ that also has an
orthogonal-sum decomposition of the form $\cG=\oplus_{v\in Q_0}\cG_v$. We then set
$\cE_v=\cV_v\otimes\cG_v$ for each $v\in Q_0$ and $\cE=\oplus_{v\in Q_0}\cE_v$. Now let
$\si:C_G(Q_0)\to\cL(\cE)$ be the representation of $C_G(Q_0)$ given by
\[
\si(\diag_{v\in Q_0}(A_v))=\diag_{v\in Q_0}(A_v\otimes I_{\cG_v}).
\]
Then, for each $n\in\BZ_+$, we have
$C_G(Q_n)\otimes_\si\cE=\oplus_{\ga\in Q_n}(\cV_{r_n(\ga)}\otimes \cG_{s_n(\ga)})$ and thus
\[
\cF^2(E,\si):=\cF^2(E)\otimes_\si\cE=\oplus_{\gamma\in\Gamma}(\cV_{r(\ga)}\otimes \cG_{s(\ga)}).
\]
It is straightforward that
\[
\si(C_G(Q_0))'=\left\{ \diag_{v\in Q_0} (I_{\cV_v}\otimes B_v)\colon B_v\in\cL(\cG_v)\right\}.
\]
Moreover, the space $C_G(Q_1)^\si$ consists of those operator matrices mapping
$\cE=\oplus_{v\in Q_0}(\cV_v\otimes\cG_v)$ into
$C_G(Q_1)\otimes\cE=\oplus_{\al\in Q_1}(\cV_{r(\al)}\otimes\cG_{s(\al)})$
that are of the form
\[
\mat{c}{I_{\cV_{r(\al)}}\otimes K_{\al,v}}_{\al\in Q_1,v\in Q_0}
\]
with $K_{\al,v}\in\cL(\cG_{r(\al)},\cG_{s(\al)})$ and $K_{\al,v}=0$ in case $r(\al)\not=v$.
Leaving out the identity operators $I_{\cV_v}$, we can identify $\si(C_G(Q_0))'$ with $C_{\tilG,\cG}(Q_0)$
and $C_G(Q_1)^\si$ with $C_{\tilG,\cG}(Q_1)$, where
$\tilG=\{Q_0,Q_1,r,s\}$ is the transposed quiver of $G$
(i.e., with the source and range maps interchanged).
Note that the generalized disk $\BD((C_G(Q_1)^\si)^*)$ of strictly
contractive operators with adjoint in $C_G(Q_1)^\si$ can be identified with the set
$\BD_{G,\cG}$ defined in Subsection \ref{subS:quiver1}.
Next observe that for
 $Z=(Z_\al\in\cL(\cG_{s(\al)},\cG_{r(\al)})\colon \al\in Q_1)\in \BD_{G,\cG}$, the $n^\tu{th}$
generalized power $Z^n$ of $Z$ then corresponds to the tuple
$(Z_\ga\in\cL(\cG_{s_n(\ga)},\cG_{r_n(\ga)})\colon \ga\in Q_n)$ with $Z_\ga=Z^\ga$, where
$Z^\ga$ is defined by (\ref{quiverpowers}).

Now let $R=(R_\ga\colon \ga\in\Ga)\in\cF^\infty(E)=\fL_\Ga(\cV,\cV)$ and
$Z=(Z_\al\colon\al\in Q_1)\in\BD_{G,\cG}$.
It then follows that the first Muhly-Solel point-evaluation $\widehat R(Z)$ of $R$ in $Z$ is given
by the tensor-product functional-calculus for $\fL_\Ga(\cV,\cV)$:
\[
\widehat R(Z)=\sum_{\ga\in\Ga}
i_{\cV_{r(\ga)}\otimes\cG_{r(\ga)}}(R_\ga\otimes Z^\ga)i^*_{\cV_{s(\ga)}\otimes\cG_{s(\ga)}},
\]
where, as in Subsection \ref{subS:quiver1}, we use the general notation: for a subspace $\cH$ of
a Hilbert space $\cK$ we write $i_\cH$ for the canonical embedding of $\cH$ into $\cK$. Thus the
$W^*$-correspondence Schur class $\cS(E,\si)$ corresponds to the free semigroupoid algebra
Schur class $\cS_G(\cV,\cV)$ in combination with the point-evaluation in the generalized
disk $\BD_{G,\cG}$.

The $W^*$-correspondence Nevanlinna-pick problem in this case thus turns out to be the
left-tangential tensor-product functional-calculus
free semigroupoid algebra Nevanlinna-Pick problem ({\bf QLTT-NP}): {\it Given a data set
\begin{equation}\label{QuivData1}
\fD: Z^{(1)},\ldots,Z^{(N)}\in\BD_{G,\cG},\ X_1,\ldots,X_N,Y_1,\ldots,Y_N\in\cL(\cE),
\end{equation}
determine when there exists a Schur class function $S\in\cS_G(\cV,\cV)$ such that
\begin{equation}\label{QLTT}
X_iS(Z^{(i)})=Y_i\quad\text{for}\quad i=1,\ldots,N.
\end{equation}
}

As in the ``unit disk'' example of Subsection \ref{subS:MS-1var}, one can deduce from the
solution to this problem the analogous result for the ``non-square'' case considered in
Subsection \ref{subS:quiver1}, but we will not work out those details here.

In order to state the solution we remark that
 the Szeg\"o kernel (\ref{Szego-ker}) specified for this setting has the form
  $$
    {\BK}_{C_G(Q_{1}), \sigma} \colon
    \BD_{G,\cG} \times \BD_{G,\cG} \to \cL^a(C_\tilG(Q_0), \cL(\cE))
  $$
  and can be written as
 \begin{equation}   \label{Szego-quiver}
\BK_{C_G(Q_{1}), \sigma}(Z,Z')[B]=
\sum_{\ga\in\Ga}
i_{\cE_{r(\ga)}}\left(I_{\cV_{r(\ga)}}\otimes Z^\ga i^*_{\cG_{s(\ga)}}
B i_{\cG_{s(\ga)}}(Z'^{\ga})^*\right)i^*_{\cG_{r(\ga)}}.
\end{equation}

An application of the general Theorem \ref{T:MS-NP} then leads to the
following solution of the problem.

\begin{theorem}\label{T:MS-Quiv1}
Suppose we are given data as in \tu(\ref{QuivData1}). Then there exists a
solution $S\in\cS_G(\cV,\cV)$ to the {\bf QLTT-NP} interpolation problem
if and only if one of the following equivalent conditions holds:
 \begin{enumerate}
      \item[(1)]
   the kernel ${\mathbb K}_{{\mathfrak D}} \colon \{1, \dots, N\} \times
  \{1, \dots, N\} \to \cL^a(C_\tilG(Q_0), \cL(\cE))$ given by
  \begin{align*}
      & {\mathbb K}_{{\mathfrak D}}(i, j)[B]
   = X_i{\mathbb K}_{C_G(Q_{1}), \sigma}(Z^{(i)},
    Z^{(j)})[B]X_j^* -
    Y_{i} K_{C_G(Q_{1}), \sigma} (Z^{(i)}, Z^{(j)}) [B]
    Y_{j}^{*}
    \end{align*}
    is a completely positive kernel.

    \item[(2)]
    the map $\varphi$ from $C_\tilG(Q_0)^{N \times N}$ to $\cL(\cE)^{N \times N}$ given by
\[
\varphi\left([B_{ij}]_{i,j=1}^N \right)=\mat{c}{\BK_\fD(i,j)[B_{i,j}]}_{i,j=1}^N
\]
    is completely positive.
  \end{enumerate}
\end{theorem}

As a consequence of the extension of Choi's theorem (see Theorem \ref{T:ChoiExt})
we obtain the following third condition.

\begin{theorem}\label{T:MS-Quiv2}
Assume that $\cV$ is separable, and that for each $v\in Q_0$ we have an
orthonormal basis $\{e_1^{(v)},\ldots,e_{\kappa_v}^{(v)}\}$ for $\cV_v$.
Then, in addition to the two conditions (1) and (2) in Theorem
\ref{T:MS-Quiv1}, a third condition equivalent to the existence of an
$S\in\cS_G(\cV,\cV)$ that satisfies (\ref{QLTT}) is that for each $v\in Q_0$
the operator matrix
\begin{equation}\label{QLTT-Pick}
\mat{c}{\BK_\fD(i,j)[e_{i'}^{(v)}e_{j'}^{(v)*}]}_{(i,i'),(j,j')\in\{1,\ldots,N\}\times\{1,\ldots,\kappa_v\}}
\in\cL(\cE)^{\kappa_v N\times\kappa_v N}
\end{equation}
is positive. Here $\BK_\fD$ is the kernel defined in Part 1 of Theorem \ref{T:MS-Quiv1}.
\end{theorem}

If one writes out the definition of the kernel $\BK_\fD$ (and that of $\BK_{C_G(Q_{1}), \sigma}$)
the operator matrix (\ref{QLTT-Pick}) turns out to be exactly the Pick matrix $\BP_{QLTT}$ of
Part 1 of Theorem \ref{T:Q-NP}. Thus we obtain the first solution
criterion of Theorem \ref{T:Q-NP}.

In order to prove Theorem \ref{T:MS-Quiv2} it is convenient to first prove two lemmas.

\begin{lemma}\label{L:composmap}
Let $\cH$ be a Hilbert space with orthogonal sum decomposition
$\cH=\oplus_{i=1}^N\cH_i$. Then the map $\psi:\cL(\cH)\to\cL(\cH)$ defined by
\[
\psi(\mat{c}{B_{i,j}}_{i,j=1}^N)=\diag_{i=1}^N (B_{i,i}),
\]
where $\sbm{B_{i,j}}_{i,j=1}^N\in\cL(\cH)$ is an operator
matrix with $B_{i,j}\in\cL(\cH_j,\cH_i)$, is a completely positive map.
\end{lemma}

More generally, the statement remains true for any conditional expectation operator
$\psi:\cA\to\cA$, where $\cA$ is an arbitrary $C^*$-algebra \cite{S97} (see also
\cite{T59}), of which the map $\psi$ in Lemma \ref{L:composmap} is just a particular example.
We give the following independent proof.

\begin{proof}[\it Proof of Lemma \ref{L:composmap}]
 It is immediately clear that $\psi$ is positive. Let $M\in\BZ_+$. Assume that we have an
operator matrix
$$
\BB= \big[ [B_{\alpha, i; \beta, j}]_{i,j=1, \dots, N}
\big]_{\alpha, \beta = 1, \dots, M}
$$
that is a positive element of $\cL(\oplus_{\alpha = 1}^{M} \cH)$,
where each $B_{\alpha, i; \beta, j} \in \cL(\cH_{j}, \cH_{i})$.  The
assumption that $\BB$ is positive implies that the matrix
$$
 \widetilde \BB_{i} = [ B_{\alpha, i; \beta, i} ]_{\alpha, \beta = 1,
 \dots, M}
$$
is positive for each $i = 1, \dots, N$, since $\widetilde \BB_{i}$ is a
principal submatrix of $\BB$.  On the other hand by definition we have
$$
\psi(\BB) = \left[ \diag_{i=1}^{N}
[B_{\alpha, i; \beta, i}] \right]_{\alpha, \beta = 1, \dots, M}.
$$
which is unitarily equivalent via a permutation matrix to
$$ \diag_{i=1}^{N} \left([ B_{\alpha, i; \beta, i}]_{\alpha, \beta = 1,
\dots, M} \right) = \diag_{i=1}^{N} (\widetilde \BB_{i}).
$$
Thus positivity of $\BB$ implies positivity of $\psi(\BB)$ as required.
\end{proof}

Now observe that we can extend the Szeg\"o kernel ${\mathbb K}_{C_G(Q_{1}), \sigma}$
to a kernel $\ov{\mathbb K}_{C_G(Q_{1}), \sigma}$ of the form
\[
\ov{\mathbb K}_{C_G(Q_{1}), \sigma}\colon \BD_{G,\cG} \times \BD_{G,\cG} \to \cL(\cL(\cG), \cL(\cE))
\]
using the same formula, i.e., in the right hand side of (\ref{Szego-quiver}) we allow $B$
to be in $\cL(\cG)$ rather than just a block diagonal operator. We can then also extend the
kernel $\BK_\fD$ in condition (1) and the map $\vph$ in condition (2) of Theorem \ref{T:MS-Quiv1}
to a kernel $\ov{\BK}_\fD$ and a map $\ov{\vph}$ of the form
\[
\ov{\BK}_\fD:\{1,\ldots,N\}\times\{1,\ldots,N\}\to\cL^a(\cL(\cG),\cL(\cE)),\quad
\ov{\vph}:\cL(\cG)^{N\times N}\to\cL(\cE)^{N \times N},
\]
simply by replacing ${\mathbb K}_{C_G(Q_{1}), \sigma}$ by $\ov{\mathbb K}_{C_G(Q_{1}), \sigma}$ in
the definitions of $\ov{\BK}_\fD$ and $\vph$.

\begin{lemma}\label{L:ExtMap}
The map $\vph$ defined in part 2 of Theorem \ref{T:MS-Quiv1} is completely positive if and only if the
map $\ov{\vph}$ is completely positive.
\end{lemma}

\begin{proof}
Let $\psi$ be the completely positive map of Lemma \ref{L:composmap} with $\cH$
replaced by $\cG$, relative to the orthogonal decomposition $\cG=\oplus_{v\in Q_0}\cG_v$.
Notice that complete positivity of $\ov{\vph}$ automatically implies complete positivity
of $\vph$. To see that the converse is also true, observe that for any $B\in\cL(\cG)$ and
any $Z,Z'\in\BD_{G,\cG}$ we have
$\ov\BK_{C_G(Q_{1}),\sigma}(Z,Z')[B]=\BK_{C_G(Q_{1}),\sigma}(Z,Z')[\psi(B)]$, and thus
also $\ov{\vph}[\sbm{B_{i,j}}_{i,j=1}^N]=\vph[\sbm{\psi(B_{i,j})}_{i,j=1}^N]$. Hence the
converse statement follows from Lemma \ref{L:composmap} and the fact that compositions of
completely positive maps are again completely positive maps.
\end{proof}

\begin{proof}[Proof of Theorem \ref{T:MS-Quiv2}]
Let $\{e_1,\ldots,e_\kappa\}$ be a reordering of the orthonormal basis
$\{e_i^{(v)}\colon v\in Q_0,\ i=1,\ldots,\kappa_v\}$ of $\cV$.
Since $\ov{\vph}$ is defined on $\cL(\cV)^{N\times N}=\cL(\cV^N)$, we can apply
Theorem \ref{T:ChoiExt} to $\ov{\vph}$, obtaining that $\ov{\vph}$ is completely
positive if and only the operator matrix
\begin{equation}\label{opmat}
\mat{c}{\ov{\BK}_\fD(i,j)[e_{i'}e_{j'}^*]}_{(i,i'),(j,j')\in\{1,\dots,N\}\times\{1,\ldots,\kappa\}}
\end{equation}
is positive. Next observe that $\ov{\BK}_\fD(i,j)[e_{i'}e_{j'}^*]=\BK_\fD(i,j)[e_{i'}e_{j'}^*]$ in case
$e_{i'},e_{j'}\in\{e_1^{(v)},\ldots,e_{\kappa_v}^{(v)}\}$ for some $v\in Q_0$, and that
$\ov{\BK}_\fD(i,j)[e_{i'}e_{j'}^*]=0$ otherwise. Thus, after a reordering in the basis
$\{e_1,\ldots,e_\kappa\}$, we can identify the operator matrix (\ref{opmat}) with the
block diagonal operator matrix with the operator matrices (\ref{QLTT-Pick}) on the diagonal.
This proves our claim.
\end{proof}

We now show how the solution criterion obtained above can be used to derive the
results for the Riesz-Dunford and operator-argument functional calculus listed in
Subsection \ref{subS:quiver1}.

\paragraph{Riesz-Dunford functional calculus Nevanlinna-Pick interpolation}

The Riesz-Dun\-ford functional calculus is the special case of the tensor functional calculus
with $\cV_v=\BC$ for each $v\in Q_0$, i.e., $\cE=\cG$.
In this case, the same argument as used in the proof of Theorem \ref{T:TFC-NP2}
(i.e., invariance under cyclic permutations of the trace) now applied
to the map $\ov{\vph}$, in combination with Theorem \ref{T:ChoiExt},
gives us the following result.

\begin{theorem}\label{T:MS-Quiv3}
In case $\cV_v=\BC$ for each $v\in Q_0$, then there exists an $S\in\cS_G(\BC,\BC)$
that satisfies (\ref{QLTT}) if and only if the map $\varphi_{*}$ from
$\cL(\cG)^{N \times N}$ to $\cL(\cG)^{N \times N}$ given by
\begin{equation}
\begin{array}{l}
\varphi_{*}\left( \left[ C_{ij}\right]_{i,j=1}^N \right)=\\
     \hspace*{.5cm}=\mat{c}{\displaystyle \sum_{\ga\in\Ga}
i_{\cG_{s(\ga)}}(Z_i^{\ga})^*i^*_{\cG_{r(\ga)}}
(X_i^*C_{i,j}X_j-Y_i^*C_{i,j}Y_j)
i_{\cG_{r(\ga)}}Z^\ga i^*_{\cG_{s(\ga)}}}_{i,j=1}^N
\end{array}
\end{equation}
is a completely positive map.
If $\cG$ is separable and $\{e_1,\ldots,e_\kappa\}$ is an orthonormal basis
for $\cG$, then complete positivity of $\vph_*$ is equivalent to positivity
of the operator matrix $\BP_{QLTRD}\in\cL(\cG)_{\kappa N\times\kappa N}$
for which the entry corresponding to the pairs
$(i,i'),(j,j,)\in\{1,\ldots,N\}\times\{1,\ldots,\kappa\}$ is given by
\begin{equation}
\begin{array}{l}
\mat{c}{\BP_{QLTRD}}_{(i,i'),(j,j')}=\\[.2cm]
     \hspace*{1.2cm}=\displaystyle
\sum_{\ga\in\Ga}
     i_{\cG_{s(\ga)}}(Z_i^{\ga})^*i^*_{\cG_{r(\ga)}}
(X_i^*e_{i}e_{j'}^*X_j-Y_i^*e_{i}e_{j'}^*Y_j)
i_{\cG_{r(\ga)}}Z^\ga i^*_{\cG_{s(\ga)}}.
\end{array}
\end{equation}
\end{theorem}

The definition of $\BP_{QLTRD}$ in Theorem \ref{T:MS-Quiv2} is the same as that in
Theorem \ref{T:Q-NP}. Thus we obtain the second statement of Theorem \ref{T:Q-NP}.

\paragraph{Operator-argument functional calculus Nevanlinna-Pick interpolation}

Note that in the case of the second Muhly-Solel point-evaluation of Subsection \ref{subS:MS-alt},
specified for the setting considered here, the points are pairs
$(T,A)\in \BD(C_G(Q_1)^*)\times C_G(Q_0)$, i.e.,  $T$ corresponds to a tuple
$(T_\al\in\cL(\cV_{r(\al)},\cV_{s(\al)})\colon \al\in Q_1)$ so that $F=(T_\al^*\colon \al\in Q_1)$
is in $C_G(Q_1)$ and the operator matrix ${\bf F}$ corresponding to $F$ is a strict contraction.
Thus $T$ is an element of the set $\BD_{\tilG,\cV}$ defined in Subsection \ref{subS:quiver1}.
The element $T^*_{(n)}=T^*\otimes\cdots\otimes T^*=F\otimes\cdots\otimes F\in C_G(Q_n)$ is then given by
the tuple $(T_\ga^*\colon \ga\in Q_n)$ with $T_\ga=T^\ga$, following the notation (\ref{QOAgenpow}),
and the operator $T_{T^*_{(n)}}^{(0)}$ (as in (\ref{crea-n})) corresponds to multiplication with the
operator matrix ${\bf T^*_{(n)}}$ associated with $T^*_{(n)}$.

It is then not difficult to see that the evaluation of an element $R=(R_\ga\colon \ga\in\Ga)$ of
the Toeplitz algebra $\cF^\infty(E)=\fL_\Ga(\cV,\cV)$ in a pair
$(T,A)\in C_G(Q_0)\times\BD_{G,{\bf OA}}$ is given by the left-tangential operator-argument
functional calculus:
\[
\hat{R}(T,A)=\sum_{\ga\in\Ga} i_{s(\ga)}T^{\ga^{\top}} A_{r(\ga)}R_\ga i_{s(\ga)}^*,
\]
where $A=\diag_{v\in Q_0}(A_v)$ and $T=(T_\al\colon \al\in Q_1)$.
The corresponding Nevanlinna-Pick problem of Subsection \ref{subS:MS-alt} thus turns out to
be the {\bf QLTOA-NP} problem considered in Subsection \ref{subS:quiver1}, and one easily sees
that the Pick matrix criterion of Theorem \ref{T:MSOA-NP} is exactly the Pick matrix criterion
of Part 3 of Theorem \ref{T:Q-NP}.

%
%

  \subsection{Still more examples.}  \label{S:still}
  There are still more seemingly different examples of
  generalized Nevanlinna-Pick interpolation covered by the Muhly-Solel
  correspondence-representation formalism.
  We mention in particular the semicrossed product algebras of Peters
  \cite{Peters} (see Example 2.6 in \cite{MS98}).  A particular instance
  of this setup yields as the Toeplitz algebra $\cF^{\infty}(E)$ the algebra of
  operators on $\ell^{2}({\mathbb Z})$ having lower-triangular matrix representation
  with points equal to strictly contractive bilateral weighted shift
  operators on $\ell^{2}({\mathbb Z})$;  this algebra can also be
  seen as the Toeplitz algebra associated with the infinite quiver:
  \begin{align*}
  & Q_{0} = \{v_{k} \colon k \in {\mathbb Z}\}, \quad
  Q_{1} = \{ \alpha_{k} \colon k \in {\mathbb Z} \}, \\
  & s(\alpha_{k}) = v_{k}, \quad r(\alpha_{k}) = v_{k+1}.
  \end{align*}
  The point-evaluation
  $\widehat R(\eta)$ in this
  context has a neat interpretation in terms of the realization of
  the lower triangular operator $R$ as the input-output operator for a
  conservative time-varying linear system
  $$
  \Sigma \colon \left\{ \begin{array}{ccl}
  x(n+1) & = & A(n)  x(n) + B(n) u(n) \\
    y(n) & = & C(n) x(n) + D(n) u(n)
    \end{array}
    \right.
  $$
  (see \cite[Section 6.3]{BBFtH}).
  In order to recover the time-varying
  interpolation theory related to time-varying $H^{\infty}$-control
  and model reduction carried out in the 1990s
  \cite{ADD90, BGK92, DD92, SCK, DvdV}, one needs to work with the
  second Muhly-Solel point-evaluation specified for this setting.  WE
  leave details to another occasion.

\section{More general Schur classes}
\label{S:testfunc}

For all the classes of Schur functions discussed to this point with
the exception of the (commutative and noncommutative) polydisk
examples in Subsection \ref{subS:polydisk}, the Schur class can be
isometrically identified with the space of contractive multipliers
between two reproducing kernel Hilbert spaces.  Although there are
some issues with different kinds of point-evaluations, one can say
that the most general of these is the Muhly-Solel
correspondence-representation setup in Section \ref{S:C*-NP}, and that
all these are tied down to the setting of a {\em complete Pick kernel}
(see \cite{AgMcC}).  There are a number of generalized Hardy algebras
which go beyond these limitations.  We mention a few:

\paragraph{Higher-rank graph algebras:}
These include direct products of the graph algebras considered in Subsection
\ref{subS:quiver1} above and much more---see \cite{KP06}.
Whether these more general algebras are of interest for robust
control theory, as those coming from SNMLs (see \cite{BGM3}),
remains to be seen.  For this general setting it remains to work out
the nature of possible point-evaluations, the Schur class and a
Schur-Agler type interpolation theory.

\paragraph{Hardy algebras associated with product systems over over
more general semigroups:}
We mention that the Fock space $\cF^{2}(E)$ of
Section \ref{S:C*-NP} is a product decomposition over the semigroup
${\mathbb Z}_{+}$.  Similar constructions but over more general
semigroups, such as ${\mathbb Z}^{n}_{+}$, pick up the higher-rank
graphs of \cite{KP06} as examples.  Here also it remains to work out
the point-evaluations, Schur class and interpolation theory.  These results
should include the commutative and noncommutative interpolation
theory for the polydisk  discussed in Subsection \ref{subS:polydisk}.

\paragraph{Schur classes based on a family of test functions:}
In this approach we assume that we are given a set $X$ and a family
$\Psi$ of functions on $X$.  We then say that a positive kernel $k
\colon X \times X \to {\mathbb C}$ is {\em admissible} whenever
$k_{\psi}(x,x'):= (1 - \psi(x) \overline{\psi(x')}) k(x, x')$ is also
a positive kernel; we denote the class of all admissible kernels by
$\cK_{\Psi}$.   A function $\varphi$ on $X$ is then said to be
in the Schur-Agler class $\mathcal{SA}_{\Psi}$ associated with $\Psi$
if $k_{\varphi}(x,x'):= (1 -\varphi(x) \overline{\varphi(x')})
k(x,x')$ is a positive kernel whenever $k \in \cK_{\Psi}$.  If one
takes $X = {\mathbb D}^{d}$ and $\Psi = \{ \psi_{k}(\lambda):=
\lambda_{k} \colon k = 1, \dots, d\}$, then the associated Schur-Agler
class $\mathcal{SA}_{\Psi}$ is just the Schur-Agler class
$\mathcal{SA}_{d}$ defined in Subsection \ref{subS:polydisk}.  In addition
to the Schur-Agler class on the polydisk as discussed in Subsection
\ref{subS:polydisk}, it has been known since the paper of Abrahamse
\cite{Abrahamse} that the Schur class over a finitely-connected planar
domain fits into this framework, but with an infinite family of test
functions.
It turns out that many
of the original ideas of Agler giving rise to transfer-function
realization and Nevanlinna-Pick interpolation theorems via Agler
decompositions go through for this general setting (see \cite{CW, DMcC,
McCS}); in the case of finitely-connected planar domains, one even
gets a continuous analogue of the Agler decomposition
\eqref{Agler-decom} (see \cite{DMcC05}).
The paper \cite{DMMcC} handles a more general scenario where
the underlying set is a semigroupoid (satisfying some additional
hypotheses) and pointwise multiplication of functions is replaced by
semigroupoid convolution.  The setup gives a unified formalism
(an alternative to the \textbf{LTOA/RTOA} setup
discussed in Section \ref{S:1var}) for the simultaneous encoding of
interpolation problems of Nevanlinna-Pick and of
Carath\'eodory-Fej\'er type inspired by the ideas of Jury \cite{Jury}.
For this class to include examples of interest (e.g., the Hardy
algebras associated with the higher-rank graph algebras mentioned
above, the (unit balls of) Toeplitz algebras
$\cF^{\infty}(E)$ appearing in Section \ref{S:C*-NP} as well as the
more general commutative Schur-Agler classes in  \cite{BB05, BB07}
and noncommutative Schur-Agler classes in \cite{BGM2}),
one must identify the appropriate collection $\Psi$ of test functions
to get started.  To handle these examples, the theory from
\cite{DMMcC, DMcC, McCS} must be extended to handle matrix- or
operator-valued test functions.  Even after this is done, it appears
that something more must be incorporated in the test-function
approach in order to handle the Muhly-Solel tensor-type point-evaluation.
Work has begun on finding a single formalism containing all these
examples as special cases (see \cite{BBDtHT}).

\paragraph{Acknowledgement} We thank David Sherman for the reference
\cite{S97}, and Vladimir Bolotnikov and Gilbert Groenewald for their useful
comments and suggestions.


\end{document}